\documentclass[onefignum,onetabnum]{siamart190516}

\usepackage{amsmath}

\usepackage{amsfonts, psfrag, graphicx, amssymb,enumitem,indentfirst,float,geometry,color,enumerate,graphicx,subcaption,algorithm,algpseudocode,pifont,hyperref,mathtools,mathrsfs}
\usepackage{mathbbol}
\usepackage{pbox}
\usepackage{booktabs, cellspace, hhline}
\allowdisplaybreaks[4]
\numberwithin{equation}{section}
\numberwithin{figure}{section}

\usepackage{etoolbox}
\usepackage[title]{appendix}

\usepackage{lipsum}

\newtheorem{thm}{Theorem}[section]
\newtheorem{rem}{Remark}[section]

\newcommand{\commentout}[1]{{}} 

\newcommand{\bfA}{{\bf A}}
\newcommand{\bfa}{{\bf a}}
\newcommand{\bfB}{{\bf B}}
\newcommand{\bfb}{{\bf b}}

\newcommand{\bfc}{{\bf c}}

\newcommand{\bfE}{{\bf E}}

\newcommand{\bff}{{\bf f}}

\newcommand{\bfH}{{\bf H}}
\newcommand{\bfh}{{\bf h}}
\newcommand{\bfI}{{\bf I}}

\newcommand{\bfk}{{\bf k}}
\newcommand{\bfK}{{\bf K}}

\newcommand{\bfL}{{\bf L}}

\newcommand{\bfn}{{\bf n}}

\newcommand{\bfQ}{{\bf Q}}

\newcommand{\bfR}{{\bf R}}
\newcommand{\bfr}{{\bf r}}

\newcommand{\bft}{{\bf t}}

\newcommand{\bfu}{{\bf u}}

\newcommand{\bfv}{{\bf v}}
\newcommand{\bfw}{{\bf w}}

\newcommand{\bfz}{{\bf z}}

\newcommand{\bfgamma}{\boldsymbol{\gamma}}

\newcommand{\bfLambda}{\boldsymbol{\Lambda}}
\newcommand{\bfphi}{\boldsymbol{\phi}}

\newcommand{\bfpsi}{\boldsymbol{\psi}}

\newcommand{\vertiii}[1]{{\left\vert\kern-0.25ex\left\vert\kern-0.25ex\left\vert #1
    \right\vert\kern-0.25ex\right\vert\kern-0.25ex\right\vert}}
    \newcommand{\vertii}[1]{{\left\vert\kern-0.25ex\left\vert #1
    \right\vert\kern-0.25ex\right\vert}}
\newcommand{\verti}[1]{{\left\vert #1
    \right\vert}}

\begin{document}
\title{Solving Two Dimensional H(curl)-elliptic Interface Systems \\ with Optimal Convergence On Unfitted Meshes
\thanks{Submitted to the editors \today.
\funding{The first author was funded by NSF DMS-2012465.}}
}
\author{
Ruchi Guo \thanks{Department of Mathematics, University of California, Irvine, CA 92697 (ruchig@uci.edu). }
\and Yanping Lin \thanks{Department of Applied Mathematics, The Hong Kong Polytechnic University, Kowloon, Hong Kong, China (yanping.lin@polyu.edu.hk).} 
\and Jun Zou \thanks{Department of Mathematics, The Chinese University of Hong Kong, Shatin, N.T., Hong Kong (zou@math.cuhk.edu.hk).} 
  }
\date{}
\maketitle
\begin{abstract}
In this article, we develop and analyze a finite element method with the first family N\'ed\'elec elements of the lowest degree for solving a Maxwell interface problem modeled by a $\bfH(\text{curl})$-elliptic equation on unfitted meshes. To capture the jump conditions optimally, we construct and use $\bfH(\text{curl})$ immersed finite element (IFE) functions on interface elements while keep using the standard N\'ed\'elec functions on all the non-interface elements. We establish a few important properties for the IFE functions including the unisolvence according to the edge degrees of freedom, the exact sequence relating to the $H^1$ IFE functions and the optimal approximation capabilities. In order to achieve the optimal convergence rates, we employ a Petrov-Galerkin method in which the IFE functions are only used as the trial functions and the standard N\'ed\'elec functions are used as the test functions which can eliminate the non-conformity errors. We analyze the \textit{inf-sup} conditions under certain conditions and show the optimal convergence rates which are also validated by numerical experiments.
\end{abstract}



\section{Introduction}

This article is devoted to solve a two dimensional (2D) $\bfH(\text{curl})$-elliptic interface problem orginated from Maxwell equations on unfitted meshes. Let $\Omega\subseteq\mathbb{R}^2$ be a bounded domain, and let it contain two subdomains $\Omega^{\pm}$ occupied by media with different magnetic and electric properties. These two subdomains are partitioned by a curve, the so called interface, and we assume it is a smooth simple Jordan curve and does not touch the boundary as shown in Figure \ref{fig:domain}. The considered $\bfH(\text{curl})$-elliptic interface problem for the electric field $\bfu\in\mathbb{R}^2$ is given by
\begin{subequations}
\label{model}
\begin{align}
\label{inter_PDE}
\underline{ \text{curl}}~\mu^{-1} \text{curl}~\bfu + \beta \bfu = \bff \;\;\;\; & \text{in} \; \Omega = \Omega^-  \cup \Omega^+, 
\end{align}
with $\bff\in\bfL^2(\Omega)$, subject to the Dirichlet boundary condition:
\begin{align}
\label{bc}
\bfu\cdot\bft = 0 \;\;\;\;\; & \text{on} \; \partial\Omega,
\end{align}
where the operator $\text{curl}$ is for vector functions $\bfv=[v_1,v_2]^t$ such that $\text{curl}~\bfu=\partial_{x_1}v_2 - \partial_{x_2}v_1$ while $\underline{\text{curl}}$ is for scalar functions $v$ such that $\underline{\text{curl}}~v = \left[ \partial_{x_2}v, - \partial_{x_1}v \right]^t$ with ``$t$" denoting the transpose herein. Moreover, we consider the following jump conditions at the interface $\Gamma$:
\begin{align}
[\bfu\cdot\bft]_{\Gamma} &:=  \bfu^+\cdot\bft -  \bfu^-\cdot\bft = 0,  \label{inter_jc_1} \\
[\mu^{-1}\text{curl}~\bfu]_{\Gamma} &:=  \frac{1}{\mu^+}\text{curl}(\bfu^+) - \frac{1}{\mu^-}\text{curl}(\bfu^-)   = 0,
\label{inter_jc_2} \\
[\beta\bfu\cdot\bfn]_{\Gamma} &:=  \beta^+\bfu^+\cdot\bfn -  \beta^-\bfu^-\cdot\bfn = 0, \label{inter_jc_3}
\end{align}
\end{subequations}
where $\bfn$ denotes the normal vector to $\Gamma$, and $\mu=\mu^{\pm}$ and $\beta=\beta^{\pm}$ in $\Omega^{\pm}$ are assumed to be positive piecewise constant functions. The interface model \eqref{model} arises from each time step in a stable time-marching scheme for the eddy current computation of Maxwell equations \cite{2000AmmariBuffaNedelec,2000BeckHiptmairHoppeWohlmuth,1996Dirks}, which serves as a magneto-quasistatic approximation by dropping the displacement current. It has been frequently used in low frequency and high-conductivity applications. In this model, $\mu$ denotes the magnetic permeability and $\beta\sim \sigma/\triangle t$ is the scaling of the conductivity $\sigma$ by the time-marching step size $\triangle t$. Note that the usual variational or weak formulation of \eqref{model} can naturally take care of the jump conditions in \eqref{inter_jc_1} and \eqref{inter_jc_2} \cite{2003Monk}, whereas \eqref{inter_jc_3} comes from the underling eddy current model
\begin{equation}
\label{inter_jc_3_J}
[\sigma \bfu\cdot\bfn]_{\Gamma} = -[\mathbf{ J}\cdot\bfn]_{\Gamma}
\end{equation}
where $\mathbf{ J}$ denotes the current source. 
We shall see that all the jump conditions in \eqref{inter_jc_1}-\eqref{inter_jc_3} will be used in the construction of IFE functions such that the resulting space has optimal approximation capabilities. For simplicity, we only consider the homogeneous jump condition, and the non-homogeneous case can be handled by introducing an enriched function, see \cite{2007GongLiLi} and a recent work on theoretical analysis \cite{2020AdjeridBabukaGuoLin} by Babu\v{s}ka et al.


Interface problems widely appear in a large variety of science and engineering applications. The interface problems related to Maxwell equations are of particular importance due to the omnipresent situation of multiple materials/media appearing in electric and magnetic fields, such as the simulation of electric machines or magnetic actuators or the design of optical devices, microwave circuits and nano/micro electric devices. In particular we refer readers to the plasma simulation in magnetostatic/electrostatic field \cite{2019LuYangBaiCaoHe} and the non-destructive testing techniques such as electromagnetic induction sensors \cite{2015AmmariChenChenVolkov} testing buried low-metallic content.

Traditional finite element methods (FEMs) can be applied to solve interface problems based on interface-fitted meshes \cite{2010LiMelenkWohlmuthZou}, otherwise the numerical solution may loss accuracy \cite{2000BabuskaOsborn}. However it is time-consuming to generate interface-fitted meshes in some applications especially when the geometry is evolving. Alternatively lots of research interests have been focused on developing numerical methods with less interface-fitted mesh requirements. Typical examples include CutFEM \cite{2015BurmanClaus,2017HuangWuXiao}, generalized FEMs \cite{1983BabuskaOsborn}, multiscale FEMs \cite{2010ChuGrahamHou,2004HouWuZhang}, immersed FEMs (IFE methods) based on unfitted meshes, and methods based on non-matching meshes \cite{2005PeterCarloSangalliGiancarlo}. The fundamental idea of IFE methods is to construct special approximate functions weakly satisfying the jump conditions such that the resulting IFE space has optimal approximation capabilities on unfitted meshes. In this work, we develop and analyze an IFE method for solving $\bfH(\text{curl})$-elliptic interface problems. Due to the potentially low regularity of this type of interface problems originated from Maxwell equations, our method is based on the first family N\'ed\'elec elements of the lowest degree \cite{2000BeckHiptmairHoppeWohlmuth,2015CaiCao,2009ChenXiaoZhang,2012DuanLiTanZheng,2016DuanQiuTanZheng,2012HiptmairLiZou}.


Many numerical methods have been developed to solve Maxwell interface problems. In \cite{2012HiptmairLiZou}, the authors analyzed the standard FEM, and in particular they established the $\bfH^1(\text{curl};\Omega)$-extension theorem which is a very useful theoretical tool in this field. In \cite{2007HuangZou}, the authors explicitly specified the dependence of error bounds of the standard FEM on material parameters. The study on preconditioners can be found in \cite{2011XuZhu}. In addition, due to the potentially low regularity, there are a few works focusing on adaptive finite element methods, see \cite{2015CaiCao,2016DuanQiuTanZheng} and reference therein. Moreover, in the trend of researches reducing the requirements on interface-fitted meshes, we refer readers to non-matching mesh methods \cite{2016CasagrandeHiptmairOstrowski,2016CasagrandeWinkelmannHiptmairOstrowski,2000ChenDuZou}, finite difference methods based on matched interface and boundary (MIB) formulation \cite{2004ZhaoWei}, and an adaptive FEM \cite{2009ChenXiaoZhang} where an unfitted mesh is generated initially and interface elements are further partitioned into submeshes according to interface geometry.

However, to our best knowledge, compared with $H^1$-elliptic interface problems, there are much fewer works on interface-unfitted FEMs for Maxwell $\bfH(\text{curl})$-elliptic interface problems. One of the major obstacles hindering the research is due to the penalties. Most of interface-unfitted FEMs in the literature require some penalty techniques to ensure the consistency and optimal convergence such as IFE methods \cite{2018GuoLinLin,2015LinLinZhang} or to enforce the jump conditions such as the Nitsche's methods \cite{2015BurmanClaus,2017HuangWuXiao}, since their trial function spaces are in general non-conforming. However for $\bfH(\text{curl})$-elliptic interface problems, Hiptmair and his collaborators in \cite{2016CasagrandeHiptmairOstrowski,2016CasagrandeWinkelmannHiptmairOstrowski} show that a direct application of the Nitsche's penalty to impose the jump condition for non-matching meshes can cause the loss of one order of convergence, namely the method just fails to converge for the lowest degree spaces. This phenomenon is due to the stability term in the penalties (Theorem 2 in \cite{2016CasagrandeHiptmairOstrowski}), and we can expect that the similar result may also hold for many penalty-type interface-unfitted methods. Indeed our numerical experiments suggest that even a penalty-type IFE method may not have the expected optimal convergence rates either, and in particular it fails to converge near the interface. In addition, from the perspective of analysis, the low regularity of the solution (only $\bfH^1(\text{curl};\Omega)$) makes many existing analysis techniques developed for $H^2(\Omega)$ solutions unavailable. All these issues make the development of an optimal convergent interface-unfitted method as well as its analysis for $\bfH(\text{curl})$-elliptic interface problems especially challenging. 




The contributions of this research are multifold. We first construct IFE functions according to the jump conditions \eqref{inter_jc_1}-\eqref{inter_jc_3}, and perform the detailed geometric analysis to give sufficient conditions that guarantee the unisolvence associated with the edge degrees of freedom. The proposed IFE functions also share some nice properties of the standard N\'ed\'elec's functions such as the optimal approximation capabilities to $\bfH^1(\text{curl};\Omega)$ functions, the commutative diagram and the exact sequence connecting to the $H^1$ IFE functions. Moreover the IFE spaces are isomorphic to the standard N\'ed\'elec spaces through the edge interpolation operator, which is advantageous in moving interface problems since the size and structure of stiffness and mass matrices keep unchanged. In addition, in order to overcome the obstacle of penalties, we take the advantage of the isomorphism and employ a Petrov-Galerkin (PG) method where the IFE functions are only used as the trial functions while the standard N\'ed\'elec functions are used as the test functions. We show that the \textit{inf-sup} condition can be guaranteed regardless of interface location relative to the mesh which further yields the optimal convergence.


The underling idea of using specially constructed problem-oriented non-conforming trial functions but keeping the standard conforming test functions can be traced back to the fundamental work of Babu\v{s}ka, Caloz and Osborn \cite{1994BabuskaCalozOsborn}. The similar idea was also adopted in \cite{2004HouWuZhang} by Hou et al. for a multiscale FEM through PG formulation to remove the cell resonance error. As for IFE methods, we refer readers to \cite{2013HouSongWangZhao} for $H^1$-elliptic interface problems. Our research demonstrates that the PG formulation is particularly useful here for $\bfH^1(\text{curl};\Omega)$-elliptic interface problems since it is able to remove the non-conformity errors on interface edges without adding any penalties. However, we highlight that the analysis of the \textit{inf-sup} stability for PG methods is in general not easy. In the founding work \cite{1994BabuskaCalozOsborn}, the proof of \textit{inf-sup} stability is based on the assumption that the coefficient of the PDE is only rough in one direction. The argument in \cite{2004HouWuZhang} utilizes the fact that the specially constructed trial functions have certain approximation capabilities to the test functions. We emphasize that there is no analysis available in the literature for PG-IFE methods despite their various applications. In this work, we are able to show the \textit{inf-sup} stability under certain conditions of the discontinuous conductivity. Our approach is based on a special regular discrete decomposition together with the De Rham complex properties relating the $H^1$ and $\bfH(\text{curl})$ IFE functions. As an extra achievement of this approach, for the first time in the literature, we are also able to establish the \textit{inf-sup} stability for the IFE method solving $H^1$-elliptic interface problems in a Petrov-Galerkin scheme \cite{2013HouSongWangZhao} with discontinuous conductivity. Although this approach currently needs to assume a critical upper bound on the jump of the discontinuity, we believe it still has theoretical importance and may motivate the further analysis.

This article has additional 6 sections. In the next section, we describe some notations and assumptions frequently used in this article. In Section \ref{sec:IFE_disre}, we develop IFE functions and discuss their properties. The PG-IFE method is also presented in this section. In Section \ref{sec:analysis}, we prove the optimal approximation capabilities of the IFE spaces. In Section \ref{sec:solu_error}, we analyze the \textit{inf-sup} stability and the solution errors of the PG-IFE scheme. In Section \ref{sec:num_examp}, we present some numerical experiments to validate the theoretical analysis. Some technique results are presented in the Appendix.



\section{Notations and Assumptions}

In this section, we describe some notations and assumptions which will be frequently used in this article. Let $\mathcal{T}_h,~h\geq 0$ be a family of interface-independent and shape-regular triangular meshes of the domain $\Omega$, and let $h_T$ be the diameter of an element $T\in\mathcal{T}_h$ and $h=\max_{T\in\mathcal{T}_h}\{h_T\}$ be the mesh size. Denote the sets of nodes and edges of the mesh $\mathcal{T}_h$ by $\mathcal{N}_h$ and $\mathcal{E}_h$, respectively. In a mesh $\mathcal{T}_h$, the interface $\Gamma$ cuts some of its elements which are called interface elements and their collection is denoted by $\mathcal{T}^i_h$ while the remaining elements are called non-interface elements and their collection is $\mathcal{T}^n_h$. 
We note that it is in general not trivial to generate a satisfactory regular fitted mesh especially when the interface is complicated. But since the mesh is not required to fit the interface for IFE methods, the difficulties are just avoided. Actually the IFE methods can be and are often used on simple triangular Cartesian meshes as shown in Figure \ref{fig:unfit_mesh}, 
especially for electromagnetic waves where the computational domain can be truncated as boxes. So in the following discussion, we simply focus on the Cartesian mesh although most of the results are applicable to general triangulations unless otherwise specified. Throughout this article, the generic constants $C$ in all the estimates are independent of the mesh size and  interface location but may depend on the parameters $\mu$ and $\beta$.

Furthermore, we make the following assumption for the mesh $\mathcal{T}_h$:
\begin{itemize}
  \item[(\textbf{A1})] The mesh is generated such that the interface can only intersect each interface element $T\in\mathcal{T}^i_h$ at two distinct points which locate on two different edges of $T$.
\end{itemize} 
Note that this assumption is fulfilled for a linear interface, and thus it should hold if a curved interface is locally flat enough, i.e., the mesh is fine enough. By the assumption (\textbf{A1}), we define $\Gamma^{T}_h$ as the line connecting the two intersection points of each interface element $T$, let $\Gamma^T_h$ cut $T$ into two subelements 
$T^-_h$, $T^+_h$ and let $\widetilde{T}$ be the subregion sandwiched by $\Gamma^T_h$ and $\Gamma$ as shown by Figure \ref{fig:sandwich}. In addition, $T^-$ and $T^+$ refer to the subelements partitioned by the interface curve instead of its linear approximation $\Gamma^T_h$. Moreover we assume the interface is well-resolved by the mesh, and it can be quantitatively described in terms of the following lemma \cite{2016GuoLin}.
\begin{lemma}
\label{lemma_interface_flat}
Suppose the mesh is sufficiently fine such that $h<h_0$ for some valve value $h_0$, then on each interface element $T\in\mathcal{T}^i_h$, there exist constants $C$ independent of the interface location inside $T$ and $h_T$ such that for every two points $X_1,X_2\in\Gamma\cap T$ with their normal vectors $\mathbf{ n}(X_1),\mathbf{ n}(X_2)$ to $\Gamma$ and every point $X\in\Gamma\cap T$ with its orthogonal projection $X^{\bot}$ onto $\Gamma^T_h$,
\begin{subequations}
\label{lemma_interface_flat_eq}
\begin{align}
    &  \| X-X^{\bot} \| \le  C h_T^2, \label{lemma_interface_flat_eq1}   \\
    &  \| \mathbf{ n}(X_1)-\mathbf{ n}(X_2) \| \le C h_T . \label{lemma_interface_flat_eq2}
\end{align}
\end{subequations}
\end{lemma} 
The explicit dependence of $h_0$ on the curvature of the interface can be found in \cite{2016GuoLin}. 

\begin{figure}[h]
\centering
\begin{minipage}{.26\textwidth}
  \centering
   \includegraphics[width=1.8in]{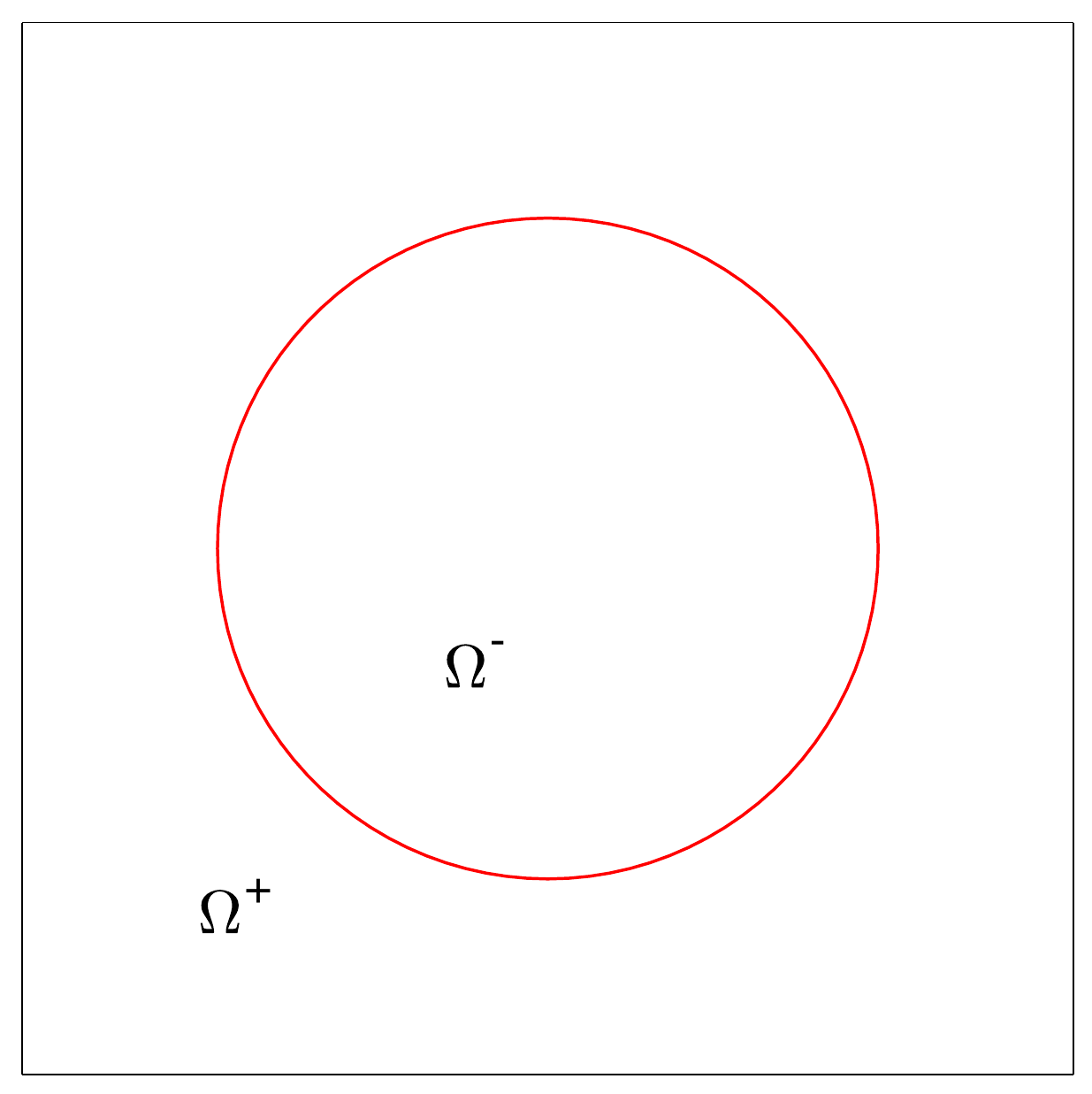}
  \caption{The model domain}
  \label{fig:domain}
\end{minipage}
\hspace{1cm}
\begin{minipage}{.26\textwidth}
  \centering
   \includegraphics[width=1.8in]{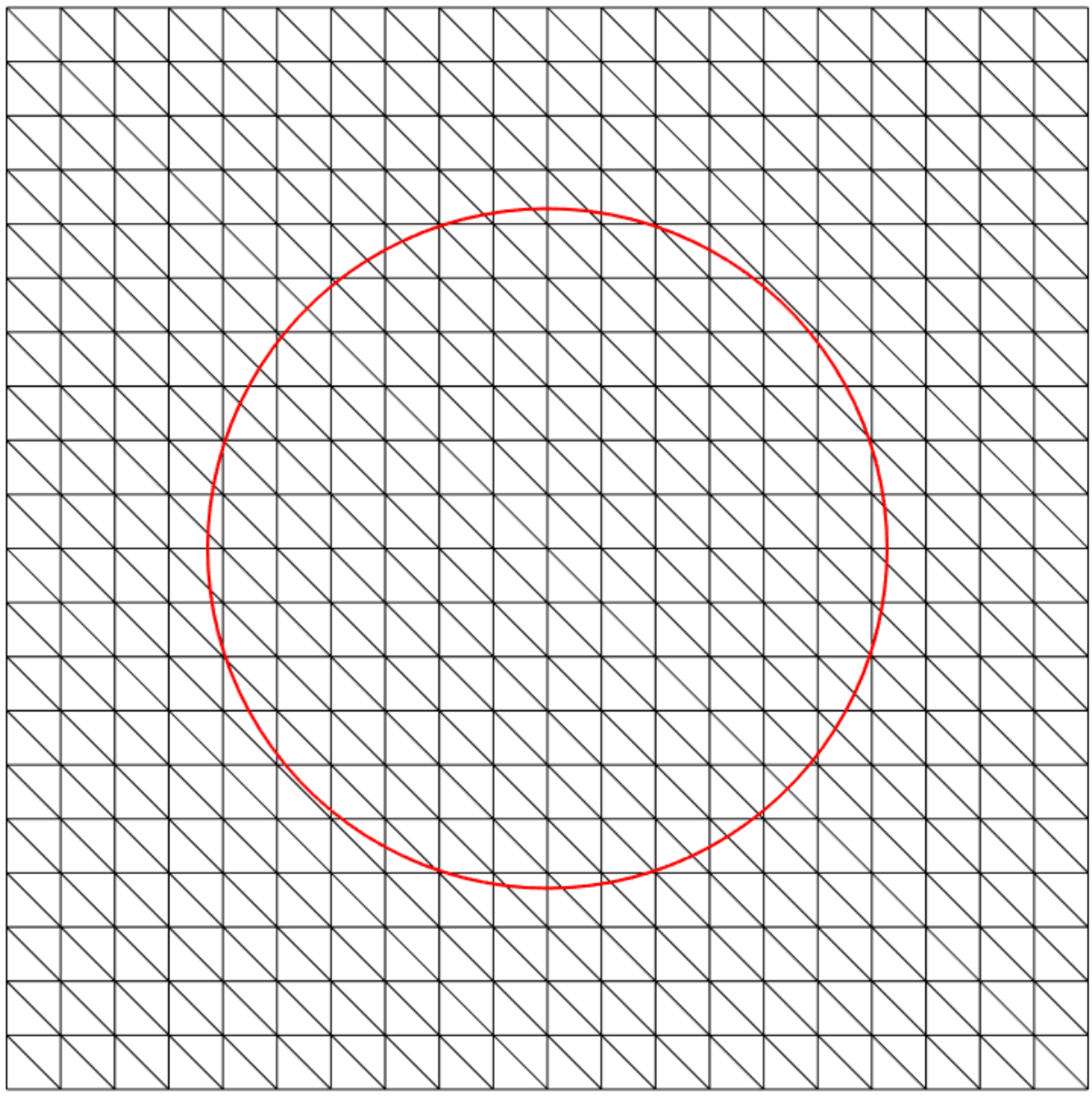}
  \caption{A unfitted mesh}
  \label{fig:unfit_mesh}
\end{minipage}
\hspace{1cm}
\begin{minipage}{.3\textwidth}
  \centering
   \includegraphics[width=2in]{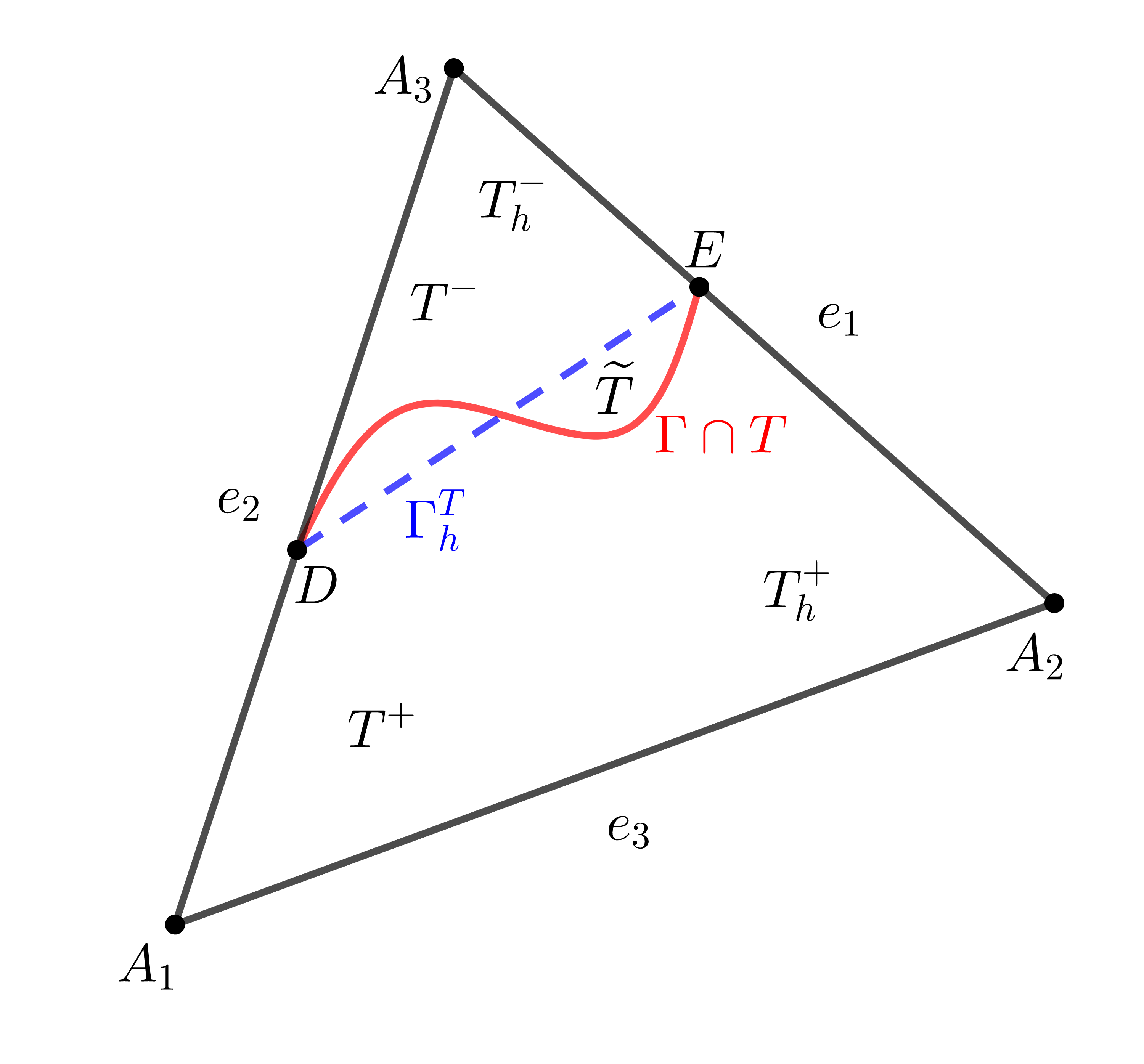}
  \caption{An interface element}
  \label{fig:sandwich}
  \end{minipage}
\end{figure}



Next we introduce some major spaces, and some other spaces will be introduced where they are used. For each subdomain $\omega\subseteq\Omega$, we let $H^k(\omega)$ and $\bfH^k(\omega)$, $k\ge0$, be the standard scalar and $\mathbb{R}^2$-vector Hilbert spaces on $\omega$; in particular $H^0(\omega)=L^2(\omega)$ and $\bfH^0(\omega) =\bfL^2(\omega)$. In addition, we introduce the $\bfH(\text{curl};\omega)$ spaces
\begin{equation}
\label{Hcurl_spa}
  \bfH^k(\text{curl};\omega) = \{ \bfv\in\bfH^k(\omega)~:~\text{curl}~\bfv \in H^k(\omega) \}.
  \end{equation}
If $|\omega\cap\Gamma|\neq 0$, we let $\omega^{\pm}=\Omega^{\pm}\cap\omega$ and further define the broken space for $k\ge\frac{1}{2}$
\begin{equation}
\label{Hcurl_spa_interf}
\widetilde{\bfH}^k(\text{curl};\omega) = \bfH^k(\text{curl};\omega^+)\cap\bfH^k(\text{curl};\omega^-) \cap \{ \bfv ~:~ \bfv ~ \text{satisfies the jump conditions} \}.
\end{equation}
For all the spaces above, we can define their subpaces $H^k_0(\omega)$, $\bfH^k_0(\omega)$, $\bfH^k_0(\text{curl};\omega)$ and $\widetilde{\bfH}^k_0(\text{curl};\omega)$ with the zero trace on $\partial \omega$. The associated norms of these spaces are denoted by $\|\cdot\|_{H^k(\omega)}$, $\|\cdot\|_{L^2(\omega)}$ (here we don't distinguish $\|\cdot\|_{H^k(\omega)}$ and $\|\cdot\|_{\bfH^k(\omega)}$ for standard Hilbert spaces for simplicity), $\|\cdot\|_{\bfH(\text{curl};\omega)}$ and $\|\cdot\|_{\bfH^1(\text{curl};\omega)}$. 

For discretization, we shall consider the first family N\'ed\'elec element of the lowest degree \cite{1980Nedelec} as the underling approximation space:
\begin{subequations}
\label{ned_spa}
\begin{align}
    &  \mathcal{ND}_h(T) = \{ \bfa + b[x_2,-x_1]^t: ~ \bfa\in\mathbb{R}^2, b\in\mathbb{R} \},  \\
    &  \mathcal{ND}_h(\Omega) = \{ \bfv\in\bfH(\text{curl};\Omega)~:~ \bfv|_T \in \mathcal{ND}_h(T) ~~ \forall T\in\mathcal{T}_h \}.
\end{align}
\end{subequations}
Similarly, $\mathcal{ND}_{h,0}(\Omega)$ denotes of the subspace of $\mathcal{ND}_h(\Omega)$ with the zero trace on $\partial\Omega$. Given an element $T$ with the edges $e_i$, $i=1,2,3$ shown in Figure \ref{fig:sandwich}, the local space $\mathcal{ND}_h(T)$ can be equipped with the well-known interpolation operator \cite{2003Monk} defined by
\begin{equation}
\label{interp_1}
\Pi_{h,T}~:~ \bfH^1(\text{curl};T) \longrightarrow \mathcal{ND}_h(T) ~~~ \text{with} ~~~ \int_{e_i} \Pi_{h,T} \bfu\cdot\bft_i ds = \int_{e_i}  \bfu\cdot\bft_i ds, ~~~ i=1,2,3.
\end{equation}
Then the global interpolation operator $\Pi_h~:~\bfH^1(\text{curl};\Omega)\rightarrow \mathcal{ND}_h(\Omega)$ can be defined piecewisely such that $\Pi_h\bfu|_T=\Pi_{h,T}\bfu$ for each element $T$. The following optimal approximation capabilities hold:
\begin{equation}
\label{interp_1_err}
\| \bfu - \Pi_{h,T}\bfu \|_{\bfH(\text{curl};T)} \le Ch_T \| \bfu \|_{\bfH^1(\text{curl};T)}~~~~ \text{and} ~~~~ \| \bfu - \Pi_h\bfu \|_{\bfH(\text{curl};\Omega)} \le Ch \| \bfu \|_{\bfH^1(\text{curl};\Omega)}.
\end{equation}

Following the \cite{2018Alberti,1999MartinMoniqueSerge,2012HiptmairLiZou}, we employ a reasonable assumption $\bfu\in \widetilde{\bfH}^1_0(\text{curl};\Omega)$ on the regularity of the solution $\bfu$ to the interface problem \eqref{model} for analysis. Note that here we only consider the case that the interface is a smooth simple Jordan curve without touching boundary and self intersection; otherwise the exact solution has further weaker regularity \cite{1999MartinMoniqueSerge}. According to the tangential continuity jump condition, we note that $\widetilde{\bfH}^1_0(\text{curl};\Omega)\subset\bfH_0(\text{curl};\Omega)$, and in particular we can further conclude $\widetilde{\bfH}^1_0(\text{curl};\Omega)\subset\bfH^s_0(\Omega)$, $s<1/2$ by Theorem 4.1 in \cite{2002Hiptmair}. In addition, the solution satisfies the following weak formulation
\begin{equation}
\label{weak_form_1}
a(\bfu,\bfv) = \int_{\Omega} \bff\cdot\bfv dX ~~~~~~ \forall \bfv \in \bfH_0(\text{curl};\Omega)
\end{equation}
where the bilinear form is given by
\begin{equation}
\label{weak_form_2}
a(\bfu,\bfv) = \int_{\Omega}\mu^{-1}\text{curl}~\bfu\cdot\text{curl}~\bfv dX + \int_{\Omega} \beta \bfu\cdot\bfv dX.
\end{equation}
It naturally leads to the following energy norm $\| \bfv \|_a=a(\bfv,\bfv)^{\frac{1}{2}}$, and it is easy to see $\| \cdot \|_a$ is equivalent to $\| \cdot \|_{\bfH(\text{curl};\Omega)}$.

We end this section by recalling the $\bfH^1(\text{curl};\Omega)$-extension operator established by Hiptmair, Li and Zou in \cite{2012HiptmairLiZou} (Theorem 3.4 and Corollary 3.5).
\begin{thm}
\label{thm_ext}
There exist two bounded linear operators
\begin{equation}
\label{thm_ext_eq0}
\bfE^{\pm}_{\emph{curl}} ~:~ \bfH^1(\emph{curl};\Omega^{\pm})\rightarrow \bfH^1(\emph{curl};\Omega)
\end{equation}
such that for each $\bfu\in\bfH^1(\emph{curl};\Omega^{\pm})$:
\begin{itemize}
  \item[1.]  $\bfE^{\pm}_{\emph{curl}}\bfu = \bfu ~~ \text{a.e.} ~\text{in} ~ \Omega^{\pm}$.
  \item[2.] $\| \bfE^{\pm}_{\emph{curl}}\bfu \|_{\bfH^1(\emph{curl};\Omega)} \le C_E \| \bfu \|_{\bfH^1(\emph{curl};\Omega^{\pm})}$ with the constant $C_E$ only depending on $\Omega$.
\end{itemize}
\end{thm}
Using these two special extension operators, we can define $\bfu^{\pm}_E = \bfE^{\pm}_{\text{curl}}\bfu^{\pm}$ which are the keys in the later analysis.


\section{IFE Discretization}
\label{sec:IFE_disre}

In this section, we first develop the so called IFE functions and discuss the construction procedure. The IFE functions are then used in a Petrov-Galerkin IFE scheme to solve the $\bfH(\text{curl})$-elliptic interface problem. In addition, we will discuss some characterization properties of the IFE functions.

\subsection{IFE Spaces And A Petrov-Galerkin IFE Scheme}

Given each interface element $T\in\mathcal{T}^i_h$, we let $\bar{\bft}$ and $\bar{\bfn}$ be the tangential and normal vectors to the segment $\Gamma^T_h$. Then we consider a linear operator 
$\mathcal{C}_T: \mathcal{ND}_h(T^-_h)\rightarrow \mathcal{ND}_h(T^+_h)$ satisfying the following conditions
\begin{subequations}
\label{weak_jc}
\begin{align}
     \mathcal{C}_T(\bfv_h)\cdot\bar{\bft} &=  \bfv_h \cdot\bar{\bft} ~~~~~~~~~~~~~  \text{at} ~ X_m, \label{weak_jc_1}  \\
     \frac{1}{\mu^{+}}\text{curl}~\mathcal{C}_T(\bfv_h) & =  \frac{1}{\mu^{-}}\text{curl}~\bfv_h ~~~~~~  \text{at} ~ X_m,  \label{weak_jc_2} \\
     \beta^+\mathcal{C}_T(\bfv_h)\cdot\bar{\bfn} &=  \beta^- \bfv_h \cdot\bar{\bfn} ~~~~~~~~~ \text{at} ~ X_m,\label{weak_jc_3}
\end{align}
\end{subequations}
where $X_m$ is a middle point of $\Gamma^T_h$. We note that the approximate jump conditions in \eqref{weak_jc} may be imposed at any point at $\Gamma^{T}_h$, and we choose the midpoint $X_m$ for simplifying the derivation of formulas of IFE functions below. Due to the lowest degree, we can show that this linear operator is not only well defined but also bijective. This is valid for both $T_h^-$ and $T_h^+$ being of general polygonal shape.
\begin{lemma}
\label{lem_C_welldefine}
The linear operator $\mathcal{C}_T$ in \eqref{weak_jc} is well defined and bijective between $\mathcal{ND}_h(T^-_h)$ and $\mathcal{ND}_h(T^+_h)$.
\end{lemma}
\begin{proof}
Note that \eqref{weak_jc} gives exactly $3$ conditions. Since $\mathcal{C}_T$ is defined between two three dimensional spaces ($\mathcal{ND}_h(T^{\pm}_h)$ has three degrees of freedom), we only need to show $\mathcal{C}_T$ is injective, namely $\mathcal{C}_T$ only has the trivial kernel $\{\mathbf{ 0}\}$. Let $\bfv=\bfa + b[x_2,-x_1]^t$ with $\bfa\in\mathbb{R}^2$ and $b\in\mathbb{R}$. Then \eqref{weak_jc_2} directly yields $b=0$, and thus \eqref{weak_jc_1} shows $\bfa=c \bf\bar{\bfn}$ for some constant $c$. Finally \eqref{weak_jc_3} leads to $c=0$, i.e., $\bfa=\mathbf{ 0}$, which finishes the proof.
\end{proof}
We can further present an equivalent description to the conditions in \eqref{weak_jc_1} and \eqref{weak_jc_2} which are useful for analysis.
\begin{lemma}
\label{lem_equiv_jc12}
On each interface element, the linear operator $\mathcal{C}_T$ satisfies
\begin{subequations}
\label{weak_eqv_jc12}
\begin{align}
     \mathcal{C}_T(\bfv_h)\cdot\bar{\bft} &=  \bfv_h \cdot\bar{\bft} ~~~~~~~~~~~~~~~~~~~~~~~~~~~~~~~~  \text{on} ~ \Gamma^{T}_h, \label{weak_eqv_jc_1}  \\
     \frac{1}{\mu^{+}} \emph{curl}~\mathcal{C}_T(\bfv_h) & =  \frac{1}{\mu^{-}}\emph{curl}~\bfv_h ~~~~~~~~~~~~~~~~~~~~~~~~~~  \text{in} ~ T.  \label{weak_eqv_jc_2} 
\end{align}
\end{subequations}
\end{lemma}
\begin{proof}
\eqref{weak_eqv_jc_2} directly follows from the fact that $\text{curl}~\bfv_h$ is a constant for each $\bfv_h\in\mathcal{ND}_h(T)$. To see \eqref{weak_eqv_jc_1}, we let $X_m=[x^m_1,x^m_2]^t$ and let $X=[x_1,x_2]^t$ be an arbitrary point on $\Gamma^T_h$, and note that $[x_2-x^m_2,-(x_1-x^m_1)]^t$ is normal to $\Gamma^T_h$, which yields \eqref{weak_eqv_jc_2}. 
\end{proof}

In particular, we can establish an explicit formulation for the operator $\mathcal{C}_T$ using the definition 
\eqref{weak_jc} and the basic form of functions in $\mathcal{ND}_h(T)$ (see \eqref{ned_spa}). Let us denote $X=[x_1,x_2]^t\in\mathbb{R}^2$, let $D$ be one of the intersection points of $\Gamma$ and edges of $T$ shown in Figure \ref{fig:sandwich}, and take the rotation matrix $\bfR=[0,1;-1,0]$, then
\begin{equation}
\begin{split}
\label{explicit_form_CT}
&\mathcal{C}_T(\bfv_h) = \bfv_h + b_1 \bfR \left[ X -D \right] +  b_2 \bar{\bfn} , ~~~~~ \forall \bfv_h\in \mathcal{ND}_h(T^-_h) \\
\text{with} ~~ & b_1 = \frac{1}{2}\left( 1- \frac{\mu^+}{\mu^-} \right) \text{curl}~\bfv_h ~~~ \text{and} ~~~ b_2 = \left( \frac{\beta^-}{\beta^+} - 1 \right) \bfv_h(X_m)\cdot\bar{\bfn}  - \frac{b_1|\Gamma^T_h|}{2}.
\end{split}
\end{equation}
Now we can define the local IFE space on each interface element as
\begin{equation}
\label{IFE_loc_spa}
\mathcal{IND}_h(T) = \{ \bfz_h~:~ \bfz_h=\bfz^-_h=\bfv_h ~ \text{on} ~ T^-_h ~ \text{and} ~ \bfz_h =\bfz^+_h = \mathcal{C}_T(\bfv_h) ~ \text{on} ~ T^+_h ~\text{for every} ~ \bfv_h \in \mathcal{ND}_h(T^-_h) \}.
\end{equation}
Due to the bijectivity of $\mathcal{C}_T$, we already have $\text{dim}(\mathcal{IND}_h(T))=3$ which is as the same as the standard N\'ed\'elec space $\mathcal{ND}_h(T)$. By \eqref{weak_eqv_jc_1}, we can see that the space $\mathcal{IND}_h(T)$ is locally $\bfH(\text{curl};T)$-conforming, i.e.,
\begin{equation}
\label{IFE_loc_spa_Hcurl}
\mathcal{IND}_h(T) \subset \bfH(\text{curl};T).
\end{equation}
However due to \eqref{weak_eqv_jc_2}, it is important to see $\mathcal{IND}_h(T) \not\subset \bfH^1(\text{curl};T)$, namely even the local IFE space has weaker regularity than the standard N\'ed\'elec space. So many standard analysis techniques such as the scaling argument can not be applied directly to estimate the approximation capabilities. A noval approach for this issue will be introduced in Section \ref{sec:analysis}.

Next, in order to construct suitable global IFE spaces, it is important to select local basis functions in $\mathcal{IND}_h(T)$ having the degrees of freedom associated with the element edges. In particular, for an element $T$ with the edges $e_i$ and the corresponding tangential vectors $\bft_i$, $i=1,2,3$, we impose the degrees of freedom for functions in $\mathcal{IND}_h(T)$:
\begin{equation}
\label{edge_dof}
\int_{e_i} \bfz\cdot\bft_i ds = v_i, ~~~~ i=1,2,3
\end{equation}
with some values $v_i\in\mathbb{R}$. The proof of the unisolvence of IFE functions, regardless of the interface location or the parameters $\mu$ and $\beta$, according to the degrees of freedom in \eqref{edge_dof} is postponed to the next subsection. It actually guarantees the existence of local IFE shape functions by taking $v_i$, $i=1,2,3$ to be $0$ or $1$ in \eqref{edge_dof}, namely there exist $\bfpsi_i\in\mathcal{IND}_h(T)$ such that
\begin{equation}
\label{IFE_fun_1}
\int_{e_j} \bfpsi_i\cdot\bft_j ds = \delta_{ij}, ~~~~ i,j =1,2,3.
\end{equation}
Then the local IFE space \eqref{IFE_loc_spa} on each interface element $T\in\mathcal{T}^i_h$ can be rewritten as
\begin{equation}
\label{IFE_loc_spa_2}
\mathcal{IND}_h(T) = \text{Span}\{ \bfpsi_1, \bfpsi_2, \bfpsi_3 \}.
\end{equation}
Here we plot some examples of IFE shape functions in Figure \ref{fig:ife_shapefun} where the interface is the red line and the parameters for the media below and above the interface, denoted by $T^-$ and $T^+$, are $\mu^-=1/2$, $\beta^-=1$ and $\mu^+=1$, $\beta^+=10$, respectively. According to the plots, we can clearly see the vector fields are discontinuous across the interface which indeed reduces the local regularity of the IFE functions.

\begin{figure}[H]
\centering
\begin{subfigure}{.32\textwidth}
     \includegraphics[width=1.6in]{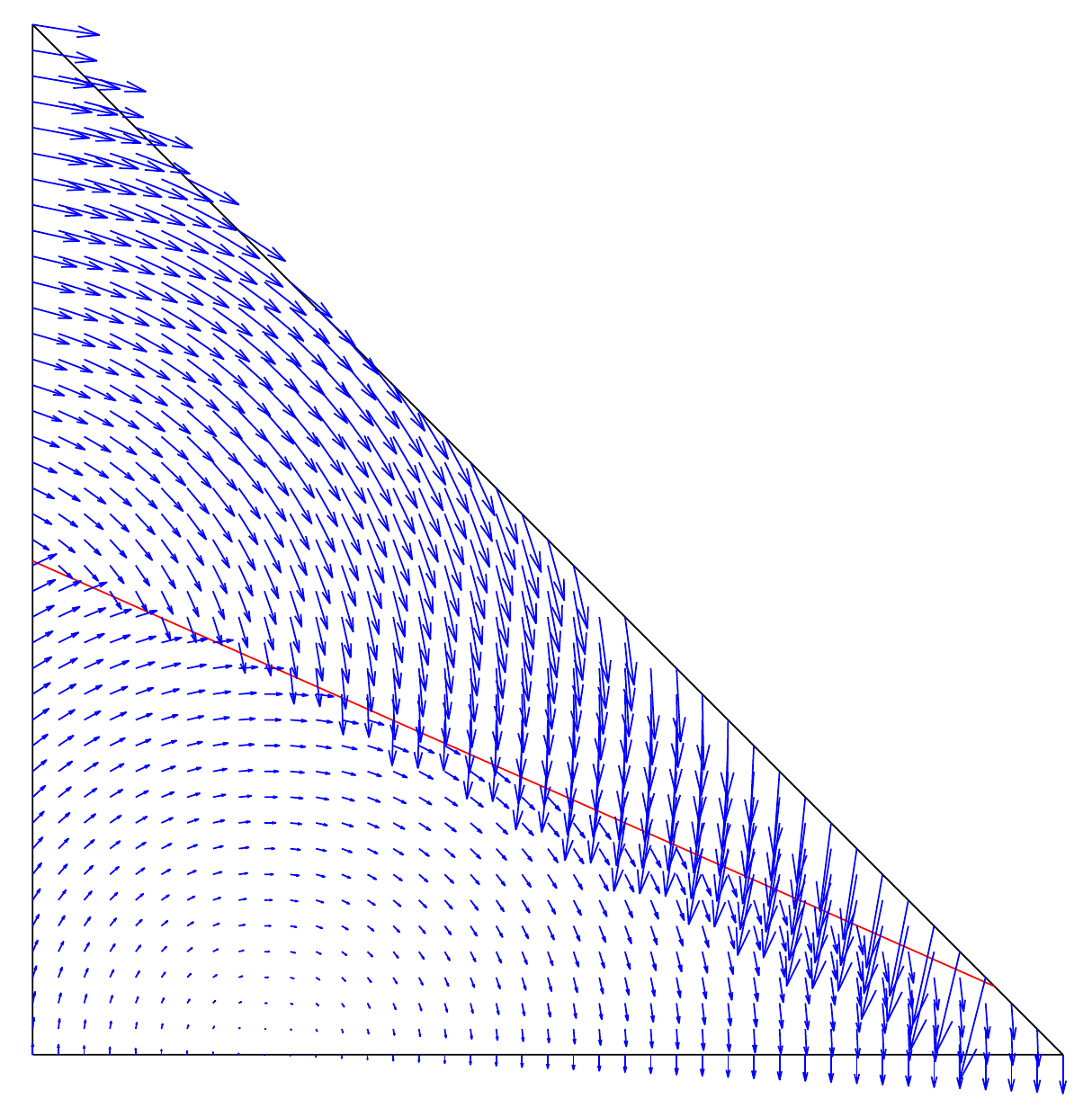}
     \label{ife_shapefun_1} 
\end{subfigure}
~
\begin{subfigure}{.32\textwidth}
     \includegraphics[width=1.6in]{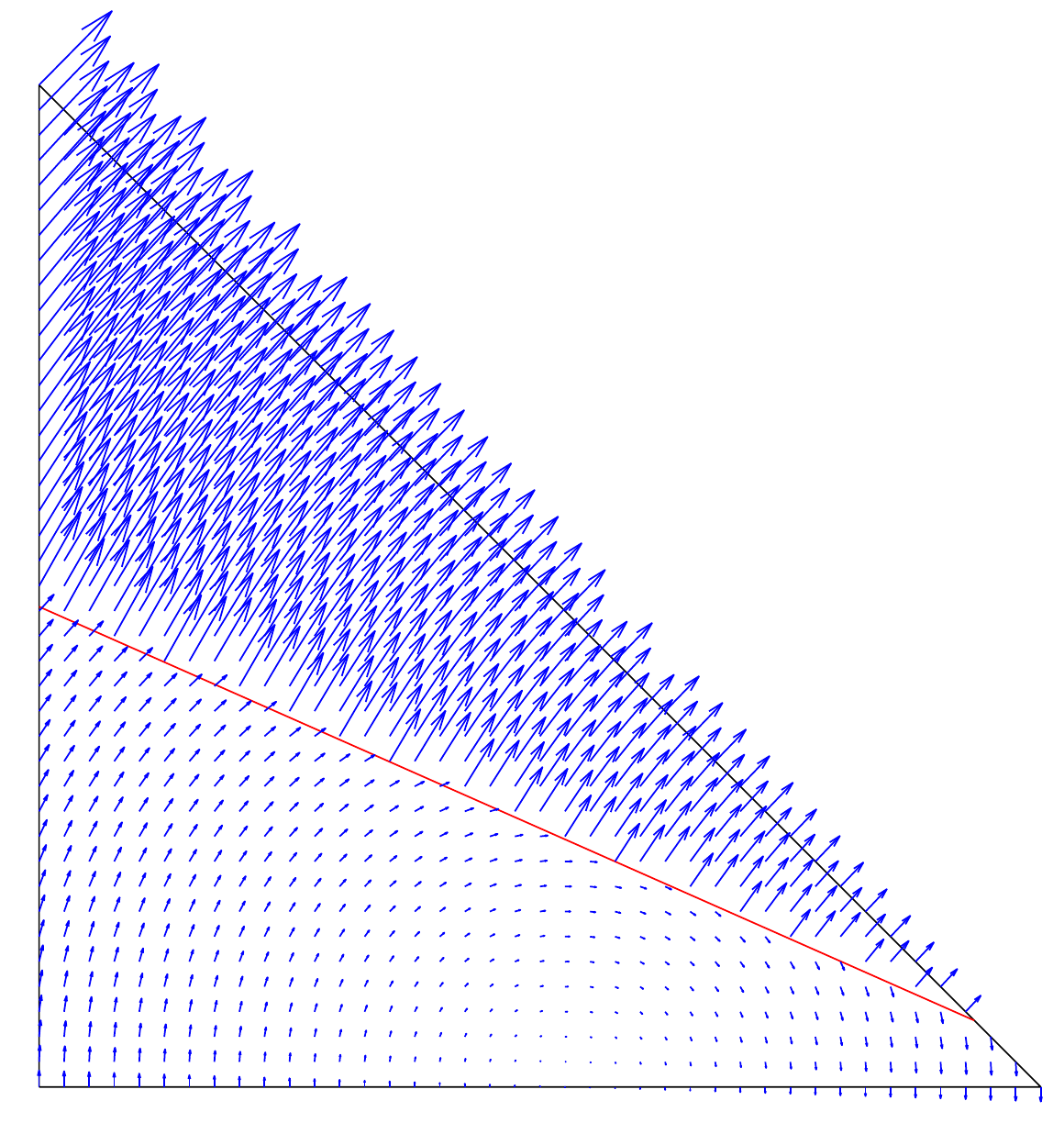}
     \label{ife_shapefun_2} 
\end{subfigure}
~
\begin{subfigure}{.32\textwidth}
     \includegraphics[width=1.6in]{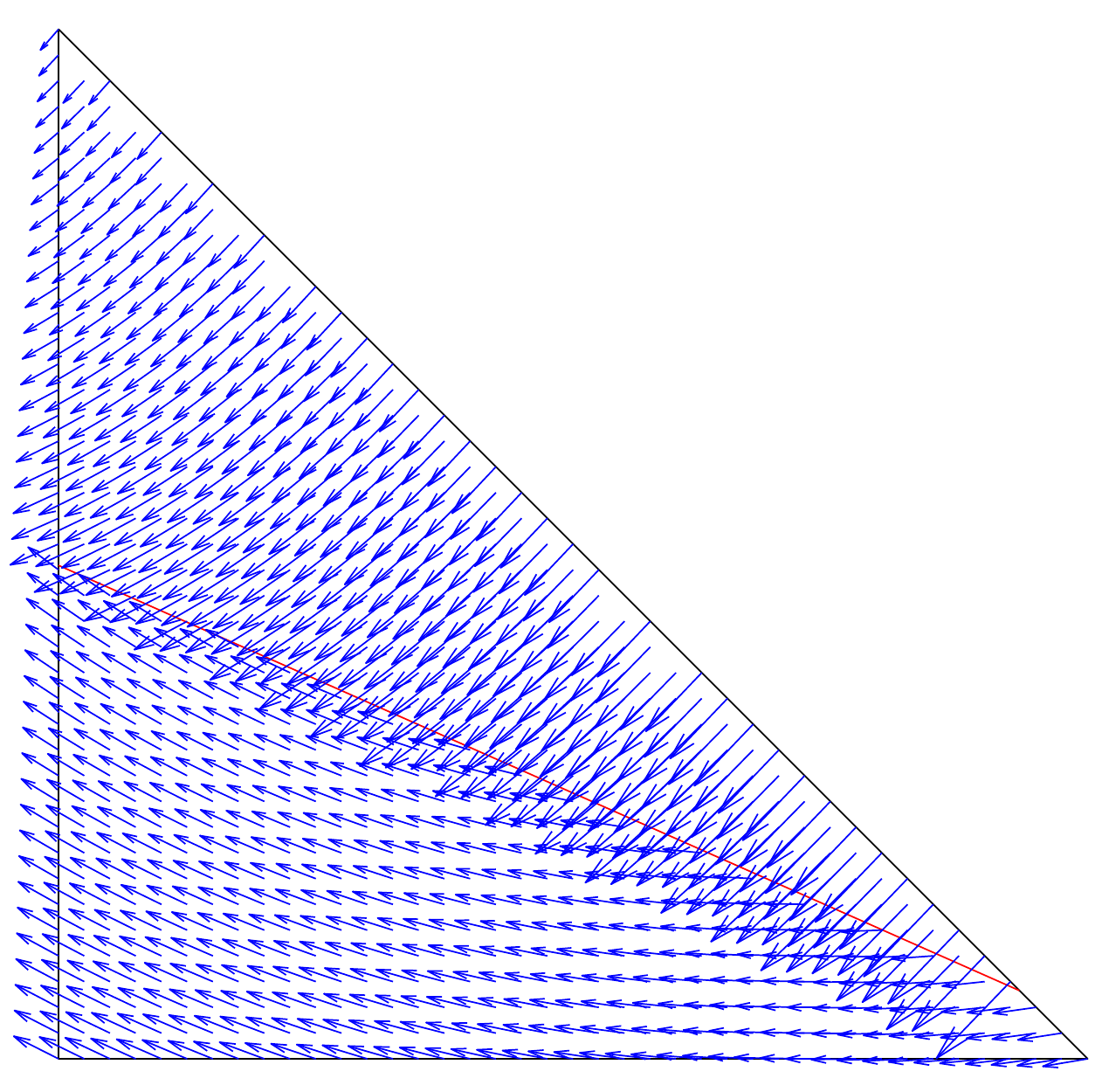}
     \label{ife_shapefun_3} 
\end{subfigure}
     \caption{IFE shape functions associated with each edge of an interface element.}
  \label{fig:ife_shapefun} 
\end{figure}

As usual, the local IFE spaces on non-interface elements are simply defined as the standard first family N\'ed\'elec space $\mathcal{IND}_h(T)=\mathcal{ND}_h(T)$. Hence, we can define the global IFE space:
\begin{equation}
\label{IFE_glob_spa}
\mathcal{IND}_h(\Omega) = \bigg\{ \bfv\in \bfL^2(\Omega)~:~ \bfv|_T \in \mathcal{IND}_h(T) ~~ \forall T\in\mathcal{T}_h, ~~ \int_e [\bfv\cdot\bft]_e ds =0 ~~ \forall e\in\mathring{\mathcal{E}}_h \bigg\},
\end{equation}
We also use $\mathcal{IND}_{h,0}(\Omega)$ to denote the subspace with the zero trace on $\partial\Omega$. We highlight that the proposed IFE space $\mathcal{IND}_h(\Omega)$ is isomorphic to the standard N\'ed\'elec space $\mathcal{ND}_h(\Omega)$. Then the proposed IFE scheme is to find $\bfu_h\in \mathcal{IND}_{h,0}(\Omega)$ such that
\begin{equation}
\label{weak_form_3}
a(\bfu_h,\bfv_h) = \int_{\Omega} \bff\cdot\bfv_h ~ dX ~~~~~~ \forall \bfv_h\in \mathcal{ND}_{h,0}(\Omega)
\end{equation}
where the bilinear form is defined in \eqref{weak_form_2}. 

Although the local IFE spaces $\mathcal{IND}_h(T)$ are subspaces of $\bfH(\text{curl};T)$, the global space in \eqref{IFE_glob_spa} is not $\bfH(\text{curl};\Omega)$-conforming. To see this, we note that $\int_e [\bfv\cdot\bft]_e ds =0$ does not lead to $[\bfv\cdot\bft]_e =0$ since $\bfv\cdot\bft$ is not a constant on $e$, and thus we do not have tangential continuity on $e$. So it has weaker regularity than the standard N\'ed\'elec space. Actually this non-conformity (discontinuity) widely appears in many interface-unfitted methods either on the interface edges \cite{2016GuoLin,2015LinLinZhang} or on the interface itself \cite{2015BurmanClaus,2017HuangWuXiao}. Interior penalties are in general needed to handle the non-conformity to ensure consistency or stability such that the optimal convergence can be obtained for various interface problems. But the penalties, particularly the stabilization term, may cause the loss of convergence order for $\bfH(\text{curl})$-elliptic interface problems if the first family N\'ed\'elec spaces are used; see the discussion and numerical results in
\cite{2016CasagrandeHiptmairOstrowski,2016CasagrandeWinkelmannHiptmairOstrowski}. Indeed, for the IFE methods, numerical results in Section \ref{sec:num_examp} indicate that the solutions of the penalty-type scheme or the standard Galerkin scheme do not converge at all near the interface, and this shall pollute the solution on the whole domain. This issue may be circumvented by some special treatment near the interface such as using high-order elements, second family N\'ed\'elec elements or discontinuous $\mathbb{P}_k$ spaces \cite{2020LiuZhangZhangZheng}, but all these approaches require at least $H^2$ regularity near the interface. Since many Maxwell-type interface models such as the 
$\bfH(\text{curl})$-elliptic problems here in general produce low regularity solutions around the interface, the aforementioned cures may not work.

Fortunately, the isomorphism between $\mathcal{IND}_h(\Omega)$ and $\mathcal{ND}_h(\Omega)$ actually motivates and enables us to test the equation \eqref{inter_PDE} by the standard functions in $\mathcal{ND}_h(\Omega)$ but use the special functions in $\mathcal{IND}_h(\Omega)$ for approximation, and this yields a Petrov-Galerkin type scheme without any penalties. Both analysis and numerical experiments below suggest the Petrov-Galerkin IFE (PG-IFE) scheme has optimal convergence rates. We note that the PG-IFE method was used in \cite{2013HouSongWangZhao} to solve 
$H^1$-elliptic interface problems of which, however, the rigorous analysis still remains open. 


\subsection{Univsolven of IFE Functions}

We then proceed to describe the detailed construction approach for shape functions satisfying \eqref{edge_dof} and prove the unisolvence. Without loss of generality, we consider an element with the vertices $A_i$ and edges $e_i$, $i=1,2,3$ with $e_1=A_2A_3$, $e_2=A_3A_1$ and $e_3=A_1A_2$, assume $A_1$ locates at the origin $(0,0)$, $A_1A_2$ is along with the $x_1$ axis, and the interface $\Gamma$ cuts the edges $A_1A_2$ and $A_1A_3$ with two points $D$ and $E$, i.e., $\Gamma^T_h=\overline{DE}$, see the left plot in Figure \ref{fig:ref_elem_1}. Let $d=|A_1D|/|A_1A_2|\in(0,1]$ and $e=|A_1E|/|A_1A_3|\in(0,1]$. 
 Consider a reference element $\hat{T}$ shown in the right plot in Figure \ref{fig:ref_elem_1} having the vertices $\hat{A}_1=(0,0)$, $\hat{A}_2=(1,0)$, $\hat{A}_3=(0,1)$ and edges $\hat{e}_1=\hat{A}_2\hat{A}_3$, $\hat{e}_2=\hat{A}_3\hat{A}_1$, $\hat{e}_3=\hat{A}_1\hat{A}_2$ with the corresponding tangential vectors $\hat{\bft}_1=\frac{1}{\sqrt{2}}[-1,1]^t$, $\hat{\bft}_2=[0,-1]^t$ and $\hat{\bft}_3=[1,0]^t$.
Then the affine mapping is given by $F_T=B_T\hat{X}$ with the Jacobian matrix $B_T$ from the reference element $\hat{T}$ to the physical element $T$. By this set-up, we actually have $\hat{D}=F^{-1}_T(D)=(d,0)$, $\hat{E}=F^{-1}_T(E)=(0,e)$ and the midpoint $\hat{X}_m=[d/2,e/2]^t$. Moreover we let $\hat{\bar{\bft}}=[\hat{t}_1,\hat{t}_2]^t$ and $\hat{\bar{\bfn}}=[\hat{n}_1,\hat{n}_2]^t$ be the images of $\bar{\bft}$ and $\bar{\bfn}$ under $F^{-1}_T$. We emphasize that $\hat{\bar{\bft}}$ is also the tangential vector of $\hat{E}\hat{D}$ but $\hat{\bar{\bfn}}$ may not be normal to $\hat{E}\hat{D}$ anymore, and they all may not be unit vectors due to scaling. We further let $\hat{\bar{\bfn}}'=[e,d]^t$ be the normal vector to $\hat{E}\hat{D}$. By the well-known Piola transformation \cite{1991BrezziFortin}, an IFE function $\bfz$ can be transformed from a function $\hat{\bfz}$ on the reference element: 
\begin{equation}
\label{piola_trans}
\bfz(X)=B_T^{-t}(\hat{\bfz}\circ F^{-1}_T)(X).
\end{equation}
Then $\hat{\bfz}$ should satisfy the following jump conditions on the reference element
\begin{equation}
\label{ref_jump_cond}
[\hat{\bfz}\cdot\hat{\bar{\bft}}]_{\hat{\Gamma}^T_h}=0, ~~~~~ [\mu^{-1}\text{curl}~\hat{\bfz}]_{\hat{\Gamma}^T_h}=0, ~~~~~~ [\beta\hat{\bfz}\cdot\hat{\bar{\bfn}}]_{\hat{X}_m}=0.
\end{equation}

\begin{figure}[h]
\centering
  \includegraphics[width=\textwidth]{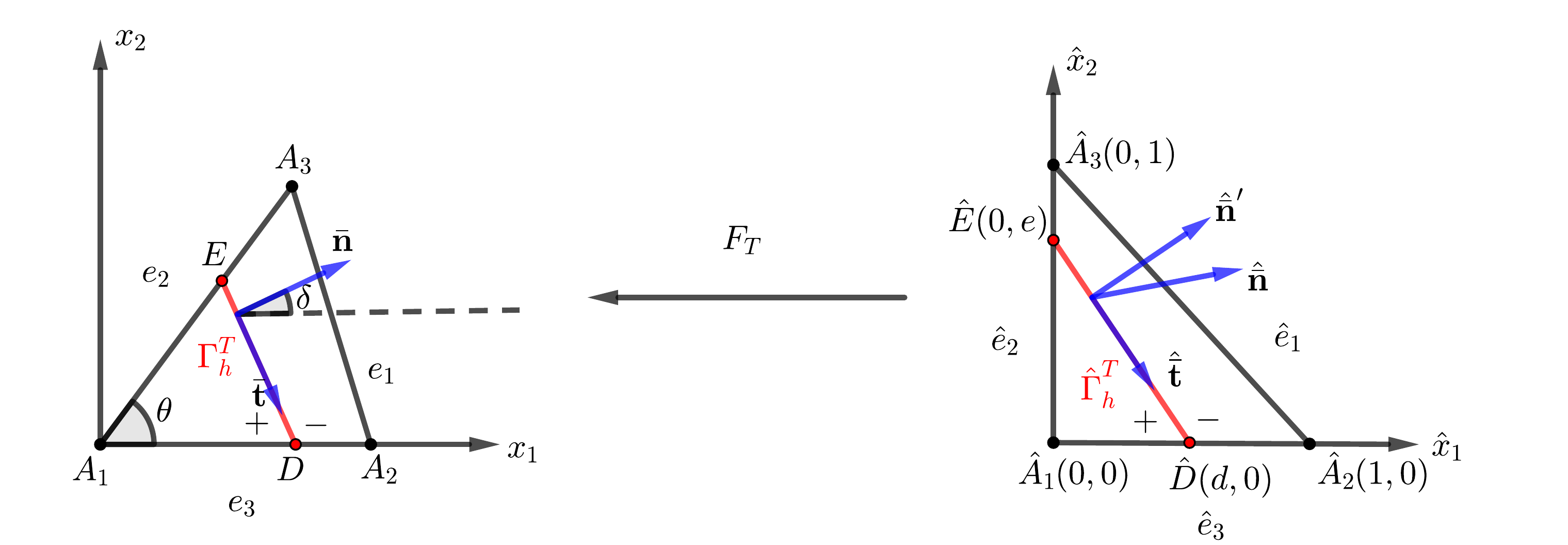}
  \caption{The affine mapping between the physical element and the reference element}
  \label{fig:ref_elem_1}
\end{figure}

Let $\hat{\phi}_i$ be the local basis functions of $\mathcal{ND}_h(\hat{T})$ associated with the edge $\hat{e}_i$, $i=1,2,3$. Thus, using the first condition in \eqref{ref_jump_cond} together with \eqref{edge_dof} for $i=1$ we have the following expression for $\hat{\bfz}$: 
\begin{equation}
\label{IFE_fun_ref_1}
\hat{\bfz}
=
\begin{cases}
      & \hat{\bfz}^- = v_1\hat{\bfphi}_1 + c_2 \hat{\bfphi}_2 + c_3 \hat{\bfphi}_3, ~~~~~~~~~~~~~~~~~~~~~~~~ \text{in} ~ T^-_h, \\
      & \hat{\bfz}^+ =  \hat{\bfz}^- + b_{1} [\hat{x}_2, -(\hat{x}_1-d)]^t + b_2 [e,d]^t, ~~~~~~~ \text{in} ~ T^+_h,
\end{cases}
\end{equation}
where the vectors $[\hat{x}_2, -(\hat{x}_1-d)]^t$ and $[e,d]^t$ are orthogonal to $\hat{E}\hat{D}$, and $\bfc=[c_2,c_3]^t$ and $\bfb=[b_1,b_2]^t$ are unknown coefficients to be determined. Using the rest two in \eqref{ref_jump_cond}, we can rewrite down the following equation for $\bfb$ and $\bfc$
\begin{equation}
\label{IFE_fun_ref_2}
\mathbf{ A} \bfb= \bfgamma v_1 + \mathbf{ B} \bfc
\end{equation}
\begin{equation*}
\label{IFE_fun_ref_3}
\text{where} ~~~ \mathbf{ A} = \left[\begin{array}{cc} 1 & 0 \\ \alpha & 2\alpha \end{array}\right], ~~~~ 
\bfgamma = \left[\begin{array}{c} \kappa \\ -\lambda\hat{\bfphi}_1(\hat{X}_m)\cdot\hat{\bar{\bfn}} \end{array}\right], ~~~~ 
\mathbf{ B} = \left[\begin{array}{cc} \kappa & \kappa \\ -\lambda\hat{\bfphi}_2(\hat{X}_m)\cdot\hat{\bar{\bfn}} & -\lambda\hat{\bfphi}_3(\hat{X}_m)\cdot\hat{\bar{\bfn}} \end{array}\right],
\end{equation*}
with $\alpha=\frac{1}{2}(e\hat{n}_1+d\hat{n}_2)=\frac{1}{2}\hat{\bar{\bfn}}'\cdot\hat{\bar{\bfn}}$, $\kappa=1-\frac{\mu^+}{\mu^-}$ and $\lambda=1-\frac{\beta^-}{\beta^+}$. Furthermore we can use \eqref{edge_dof} for $i=2,3$ to obtain 
\begin{equation}
\label{IFE_fun_ref_4}
\mathbf{ I}_2 \bfc + \mathbf{ R} \bfb = \bfv,
\end{equation}
where $\mathbf{ I}_2$ is the $2\times2$ identity matrix, $\mathbf{ R}=de[-1,-1;0,1]$ and $\bfv=[v_2,v_3]^t$. Solving the linear systems \eqref{IFE_fun_ref_2} and \eqref{IFE_fun_ref_4}, we can compute all the unknown coefficients in \eqref{IFE_fun_ref_1}. The solvability is addressed by the following theorem.



\begin{thm}[Unisolvence]
\label{thm_unisol}
Suppose $T$ does not have obtuse angles, then for each $[v_1,v_2,v_3]\in\mathbb{R}^3$, there exists a unique piecewise polynomial $\bfz$ in the IFE space $\mathcal{IND}_h(T)$ satisfying the degrees of freedom \eqref{edge_dof}.
\end{thm}
\begin{proof}
Under the notations above, by the assumption that $T$ does not have obtuse angles, we can verify that 
\begin{equation}
\label{thm_unisol_eq1}
\hat{\bar{\bfn}}'\cdot\hat{\bar{\bfn}} >0 ~~~ \text{and} ~~~ \frac{de(\hat{n}_1+\hat{n}_2)}{\hat{\bar{\bfn}}'\cdot\hat{\bar{\bfn}}} \in [0,1]. 
\end{equation}
Since the derivation is quite technical and elementary, we put it to Appendix \ref{App_A1}. By the first inequality above, we know that $\alpha=\frac{1}{2}\hat{\bar{\bfn}}'\cdot\hat{\bar{\bfn}}>0$ and thus $\mathbf{ A}$ is invertible. So \eqref{IFE_fun_ref_2} gives the formula to compute $\bfb$ in terms of $\bfc$. Putting it into \eqref{IFE_fun_ref_4}, we have the following linear system
\begin{equation}
\label{IFE_fun_ref_5}
( \mathbf{ I}  +  \mathbf{ R} \mathbf{ A}^{-1} \mathbf{ B} ) \bfc = \bfv - \mathbf{ A}^{-1} \bfgamma v_1.
\end{equation}
We only need to show the non-singularity of the matrix in $\mathbf{ I}  +  \mathbf{ R} \mathbf{ A}^{-1} \mathbf{ B}$ for the reference element. Direct computation shows that the matrix $\mathbf{ I} +  \mathbf{ R} \mathbf{ A}^{-1} \mathbf{ B}$ has two eigenvalues
$
1 -de \kappa ~\text{and}  ~ 1-\frac{de(\hat{n}_1 + \hat{n}_2) \lambda}{ e\hat{n}_1 + d\hat{n}_2 }.
$
Because $d,e\in[0,1]$, using the second inequality in \eqref{thm_unisol_eq1}, we have
\begin{equation*}
\label{thm_unisol_eq2}
1 - de \kappa \ge \min\Big\{ 1, \frac{\mu^+}{\mu^-} \Big\}>0 ~~~ \text{and} ~~~ 1-\frac{de(\hat{n}_1 + \hat{n}_2) \lambda}{ e\hat{n}_1 + d\hat{n}_2 } \ge \min\Big\{ 1, \frac{\beta^-}{\beta^+} \Big\}>0
\end{equation*}
which finishes the proof.
\end{proof}

Theorem \ref{thm_unisol} guarantees the unique existence of IFE shape functions in \eqref{IFE_fun_1} 
theoretically. Furthermore, we can prove the following properties of these shape functions.
\begin{thm}
\label{thm_bound}
Suppose $T$ does not have obtuse angles and is shape regular, then for $i=1,2,3$ there holds
\begin{subequations}
\label{thm_bound_eq0}
\begin{align}
 &  \| \bfpsi_i \|_{L^{\infty}(T)} \le C h^{-1}_T,    \label{thm_bound_eq01} \\
&  \frac{1}{\mu}\emph{curl}~\bfpsi_i =  \frac{ 2 \emph{det} B_T^{-1}}{(1-de)\mu^- + de \mu^+} \le C h^{-2}_T.  \label{thm_bound_eq02} 
\end{align}
\end{subequations}
\end{thm}
\begin{proof}
See the Appendix \ref{App_A2}.
\end{proof}
\eqref{thm_bound_eq02} shows that for each fixed interface element $\text{curl}~\bfpsi_i$, $i=1,2,3$ equal the same constant. It can be understood as the generalization of the properties of the standard N\'ed\'elec functions $\bfphi_i$ since $\text{curl}~\bfphi_i = 2\text{det} B^{-1}_T$ for $i=1,2,3$.

\begin{rem}
\label{rem_shapeT}
We note that the unisolvence of Theorem \ref{thm_unisol} only requires that $T$ doe not have an obtuse angle while the boundedness of Theorem \ref{thm_bound} extraly requires the shape regularity. These conditions can be certainly satisfied for Cartesian meshes used for IFE methods mainly in this article.
\end{rem}

\subsection{Characterization of Immersed Elements}

In this subsection, we follow the spirit of the well-known De Rham Complex \cite{2006ArnoldFalkWinther} to derive some analog properties for the proposed $\bfH(\text{curl})$ immersed elements and the $H^1$ immersed elements for elliptic equations in the literature \cite{2016GuoLin,2004LiLinLinRogers}. These results have not appeared in any literature and serve as the foundation in the analysis of the \textit{inf-sup} stability in Section \ref{sec:solu_error}. 

First, we mimic the interpolation \eqref{interp_1} to define a similar interpolation operator for the IFE space $\mathcal{IND}_h(\Omega)$:
\begin{equation}
\begin{split}
\label{interp_4}
& \widetilde{\Pi}_h~:~ \widetilde{\bfH}^1(\text{curl};\Omega) \longrightarrow \mathcal{IND}_h(\Omega) ~~ \text{with} ~~  \int_e  \widetilde{\Pi}_h\bfu\cdot\bft ds = \int_e \bfu\cdot\bft ds, ~~ \forall e\in\mathcal{E}_h. 
\end{split}
\end{equation}  
Again we have the local interpolation $\widetilde{\Pi}_{h,T}=\widetilde{\Pi}_h|_T$ for each element $T$. Since IFE functions reduce to the standard N\'ed\'elec functions on non-interface elements, we simply have $\widetilde{\Pi}_{h,T}={\Pi}_{h,T}$ on $T\in\mathcal{T}^n_h$. In addition, the isomorphism between the IFE space $\mathcal{IND}_h(\Omega)$ and the standard N\'ed\'elec space $\mathcal{ND}_h(\Omega)$ can be described by the following interpolation operator:
\begin{equation}
\begin{split}
\label{iso_map}
\mathbb{\Pi}_h~:~ \mathcal{IND}_h(\Omega) \longrightarrow \mathcal{ND}_h(\Omega) ~~~ 
\text{with} ~~  \int_e \mathbb{\Pi}_h\bfu_h\cdot\bft ds = \int_e \bfu_h\cdot\bft ds, ~~ \forall e\in\mathcal{E}_h
\end{split}
\end{equation}
where we note that $\mathbb{\Pi}_h$ is well-defined due to the degrees of freedom of the global IFE functions in \eqref{IFE_glob_spa}. We can also define the local mapping $\mathbb{\Pi}_{h,T}=\mathbb{\Pi}_h|_T$ on $\mathcal{IND}_h(T)$. We note that $\mathbb{\Pi}_h$ and $\mathbb{\Pi}_{h,T}$ can be understood as the interpolation operators $\Pi_h$ and $\Pi_{h,T}$ in \eqref{interp_1} applied to the space $\mathcal{IND}_h(\Omega)$, while $\mathbb{\Pi}^{-1}_h$ and $\mathbb{\Pi}^{-1}_{h,T}$ can be understood as the interpolation operators $\widetilde{\Pi}_h$ and $\widetilde{\Pi}_{h,T}$ in \eqref{interp_4} applied to the space $\mathcal{ND}_h(\Omega)$. To show how the IFE functions are related to their FE counterparts through $\mathbb{\Pi}_h$, here we plot an example of their $x$-component in Figure \ref{fig:ife_fe_fun} where we can clearly see they are exactly the same away from the interface, and the IFE functions have jumps across the interface while the FE functions loss the jump information.
\begin{figure}[H]
\centering
\begin{subfigure}{.32\textwidth}
     \includegraphics[width=2.2in]{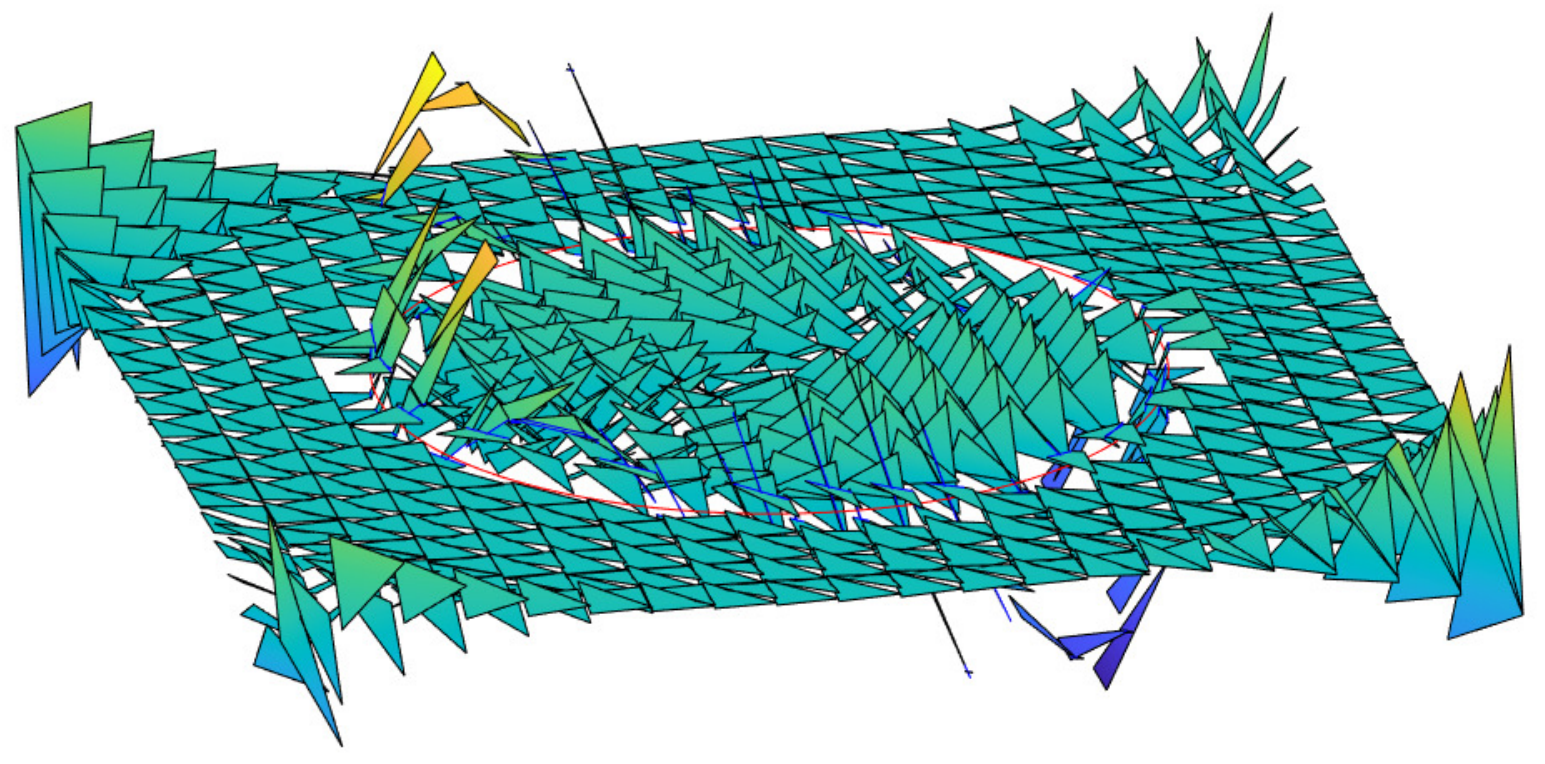}
     \label{ife_fun_1} 
\end{subfigure}
~~~~~
\begin{subfigure}{.32\textwidth}
     \includegraphics[width=2.2in]{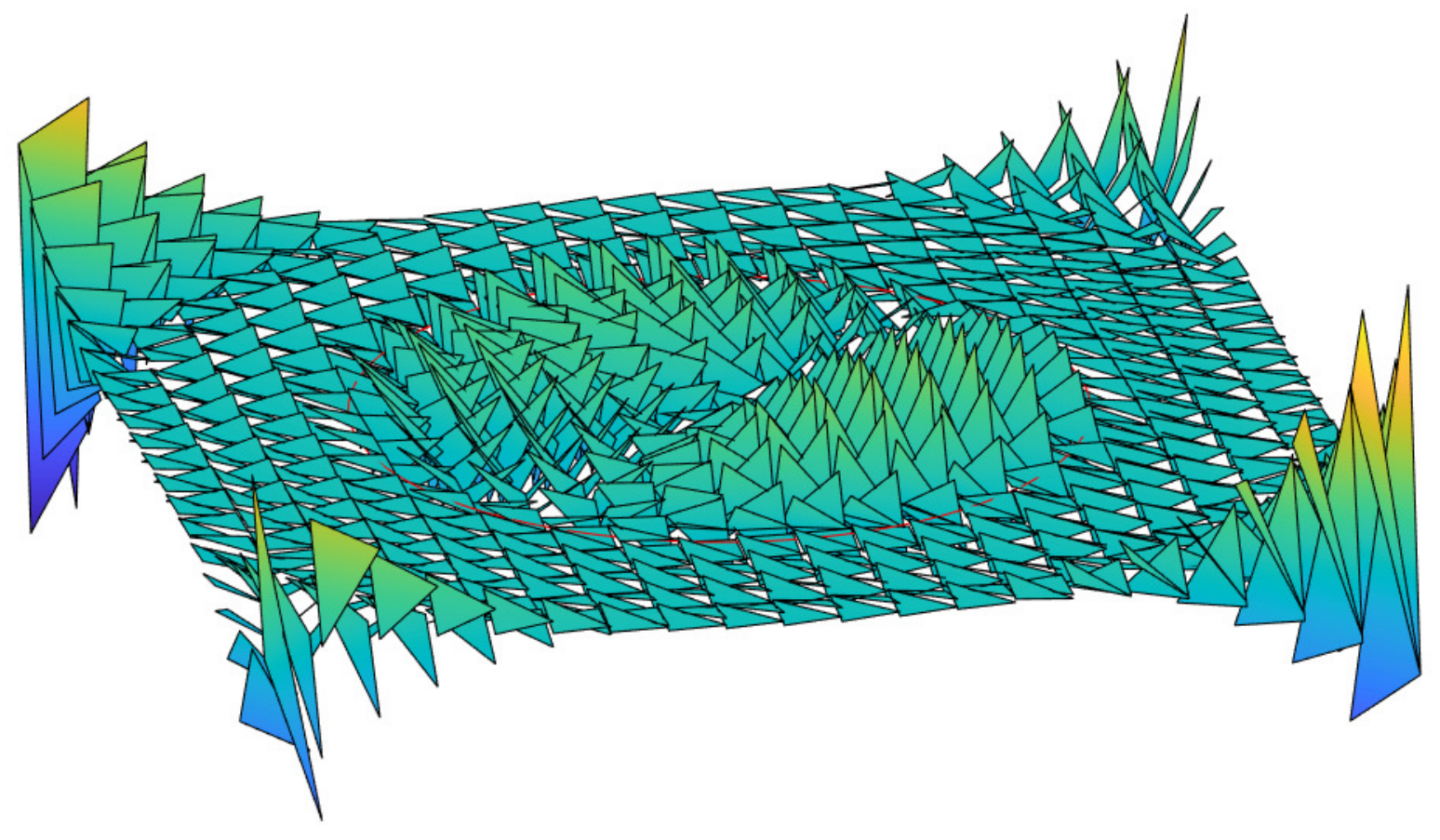}
     \label{fe_fun_1} 
\end{subfigure}
     \caption{A Global $\bfH(\text{curl})$ IFE function and its FE isomorphic image.}
  \label{fig:ife_fe_fun} 
\end{figure}

\begin{figure}[H]
\centering
\begin{subfigure}{.32\textwidth}
     \includegraphics[width=2.2in]{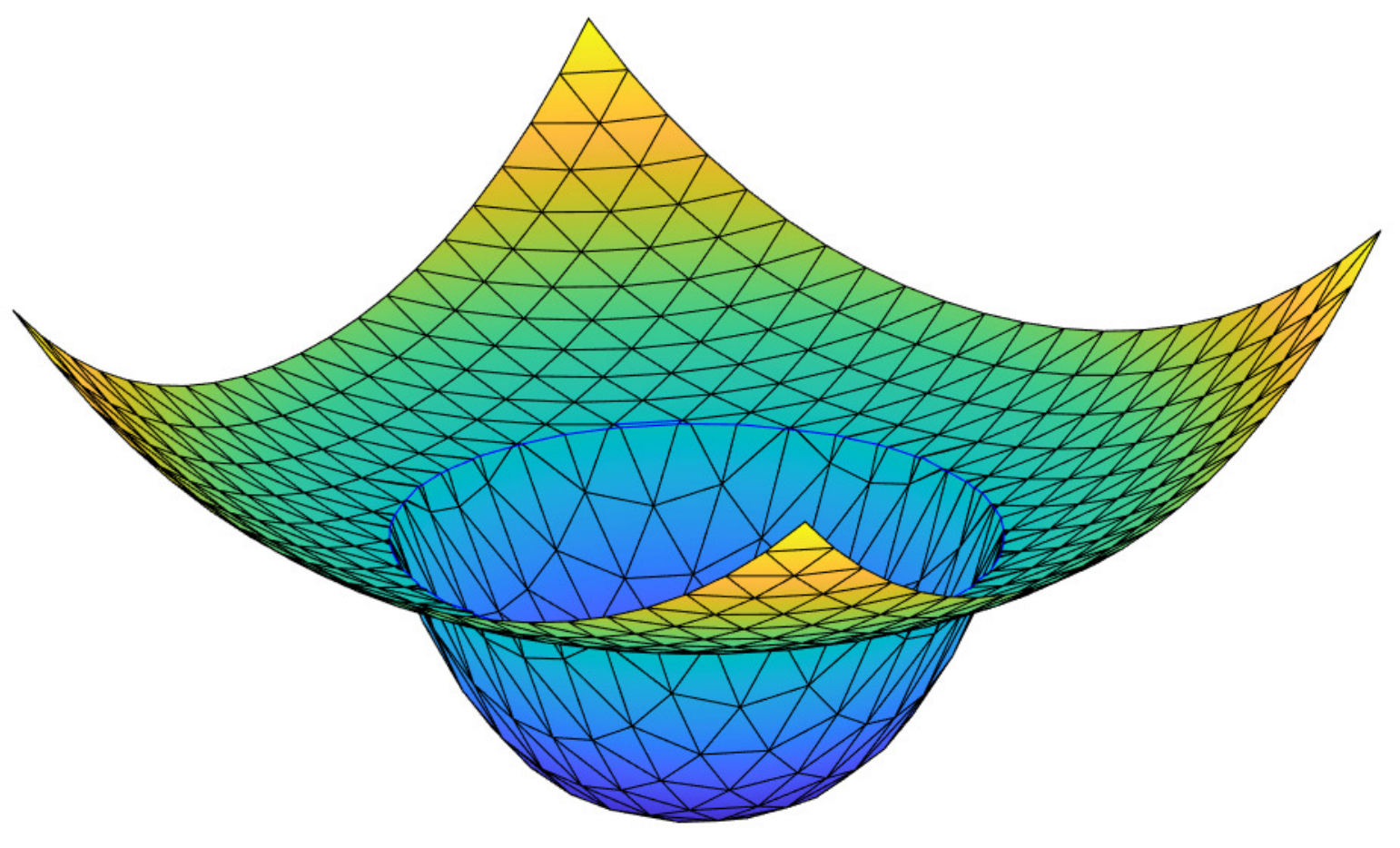}
     \label{H1_ife_fun_1} 
\end{subfigure}
~~~~~
\begin{subfigure}{.32\textwidth}
     \includegraphics[width=2.2in]{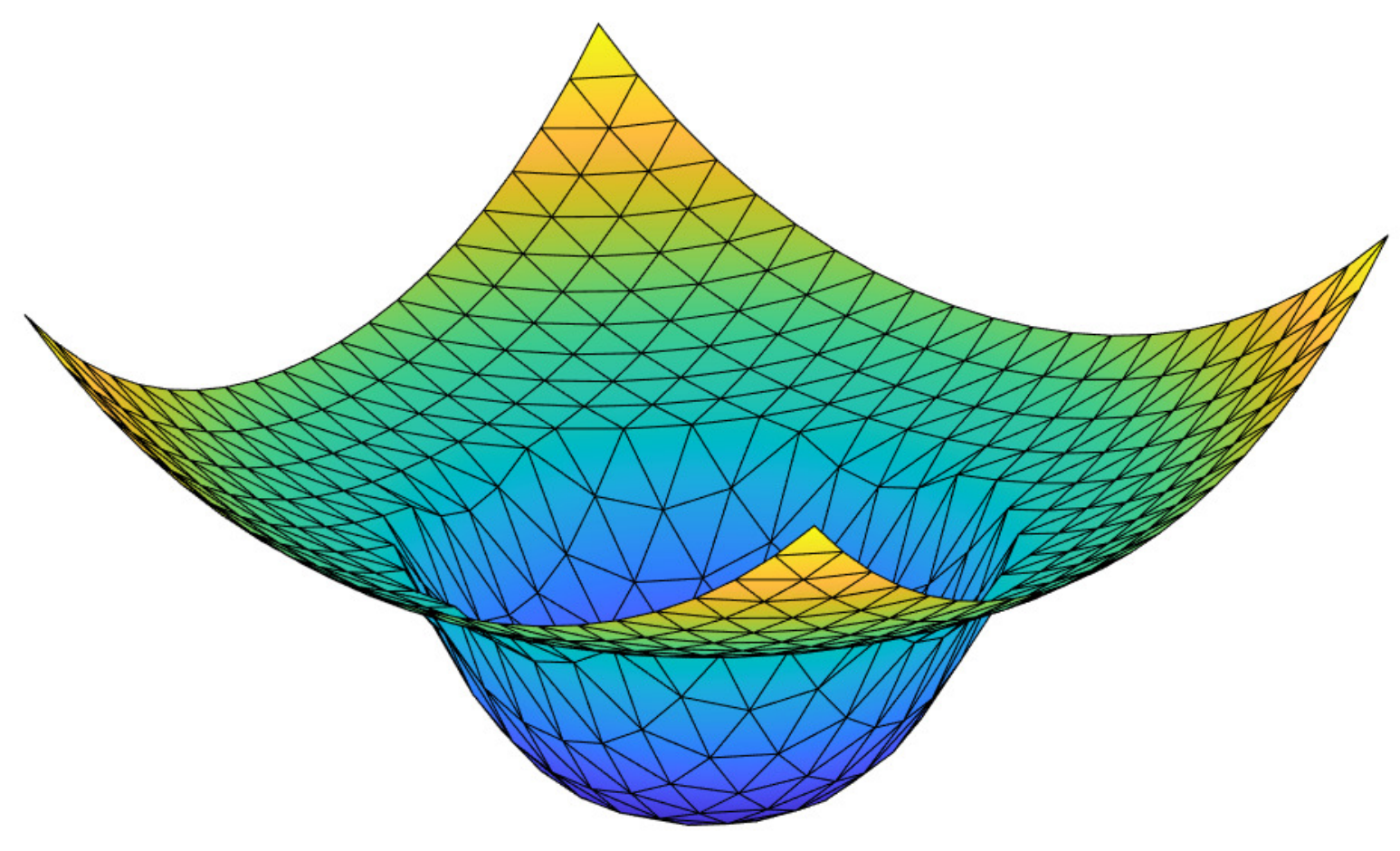}
     \label{H1_fe_fun_1} 
\end{subfigure}
     \caption{A Global $H^1$ IFE function and its FE isomorphic image.}
  \label{fig:H1_ife_fe_fun} 
\end{figure}

Next we recall the $H^1$ immersed elements in the literature \cite{2004LiLinLinRogers}. We define $S_h(T):=\mathbb{P}_1(T)$ as the standard scalar linear polynomial space, and let $S_h(\Omega)\subset H^1_0(\Omega)$ be the continuous piecewise linear finite element space. Also we define $I_{h,T}$ and $I_h$ as the standard local and global nodal interpolation operator, respectively. The scalar solution $u^{\pm}:=u|_{\Omega^{\pm}}$ of the $H^1$-elliptic interface problems should satisfy the jump conditions at the interface:
\begin{equation}
\label{elli_jump_cd}
u^+ = u^- ~~~~ \text{and} ~~~~ \beta^+ \nabla u^+\cdot\bfn = \beta^- \nabla u^-\cdot\bfn ~~~ \text{at} ~ \Gamma
\end{equation}
where $\beta^{\pm}$ are assumed the same as those in \eqref{model} also referring to conductivity in physics. We define the underling Sobolev space
$$\widetilde{H}^2_0(\Omega)=H^2_0(\Omega^-)\cap H^2_0(\Omega^+)\cap\{v:v~\text{satisfies the jump conditions in \eqref{elli_jump_cd}} \}. 
$$
Then the local IFE space $\widetilde{S}_{h,T}$ on each interface element, for example the one shown in Figure \ref{fig:sandwich}, is defined as
\begin{equation}
\label{p_IFE}
\widetilde{S}_h(T) = \{ z_h~:~ z^{\pm}_h= z_h|_{T^{\pm}_h}\in \mathbb{P}_1(T^{\pm}_h),~ z^+_h=z^-_h, ~~ \beta^-\nabla z^-_h\cdot\bar{\bfn} = \beta^+\nabla z^+_h\cdot\bar{\bfn} ~~ \text{at} ~ \Gamma^T_h \}.
\end{equation}
Clearly there holds $\widetilde{S}_h(T)\subset H^1(T)$, and one can pick the shape functions from $\widetilde{S}_h(T)$ having the nodal value degrees of freedom as the same as the standard linear finite element shape functions. We also let $\widetilde{S}_h(\Omega)$ be the global IFE space continuous at the mesh nodes, but we in general have $\widetilde{S}_h(\Omega)\subsetneq H^1(\Omega)$ due to the discontinuities across interface edges. However the continuity imposed at mesh nodes still enables us to define the nodal interpolation operators $\tilde{I}_{h,T}$ and $\tilde{I}_h$. We refer readers to \cite{2004LiLinLinRogers} for more details about these $H^1$ IFE spaces and the estimates of the interpolation errors. Moreover, we can also define the isomorphism $\mathbb{I}_h$ between $S_h(\Omega)$ and $\widetilde{S}_h(\Omega)$ such that
\begin{equation}
\label{I_IFE}
\mathbb{I}_h~:~ \widetilde{S}_h(\Omega)\rightarrow S_h(\Omega) ~~ \text{with} ~~ \mathbb{I}_hv_h(X) = v_h(X) ~~ \forall X\in\mathcal{N}_h,
\end{equation}
 where again $\mathbb{I}_h$ can be understood as $I_h$ applied to $\widetilde{S}_h(\Omega)$ and $\mathbb{I}^{-1}_h$ can be understood as $\tilde{I}_h$ applied to $S_h(\Omega)$. The new notations $\mathbb{I}_h$ and $\mathbb{\Pi}_h$ are used to emphasize the isomorphism and to avoid confusion. Moreover, we let $S_{h,0}(\Omega)$ and $\widetilde{S}_{h,0}(\Omega)$ be the corresponding subspaces with zero traces on $\partial\Omega$. Here we also plot an example of $H^1$ IFE functions and its FE isomorphic image in Figure \ref{fig:H1_ife_fe_fun} where we can see the IFE function can capture more detailed jump information at the interface.


\begin{equation}
\label{DR_comp_glob}
\left.\begin{array}{ccc}
H^2_0(\Omega) & \xrightarrow[]{~~\nabla~~} & \bfH^1_0(\text{curl};\Omega) \\
~~~~\bigg\downarrow I_{h}  &  & ~~~~\bigg\downarrow \Pi_{h} \\
S_{h,0}(\Omega) & \xrightarrow[]{~~\nabla~~} & \mathcal{ND}_{h,0}(\Omega)
\end{array}\right.
~~~ 
\left.\begin{array}{ccc}
\widetilde{H}^2_0(\Omega) & \xrightarrow[]{~~\nabla~~} & \widetilde{\bfH}^1_0(\text{curl};\Omega) \\
~~~~\bigg\downarrow \tilde{I}_{h}  &  & ~~~~\bigg\downarrow \tilde{\Pi}_{h} \\
\widetilde{S}_{h,0}(\Omega) & \xrightarrow[]{~~\nabla~~} & \mathcal{IND}_{h,0}(\Omega)
\end{array}\right.
\end{equation}

According to the well-known De Rham Complex \cite{2006ArnoldFalkWinther,2006ArnoldFalkWintherExterior}, we plot the diagram for $H^2_0(\Omega)$ and $\bfH^1_0(\text{curl};\Omega)$ spaces in the left of \eqref{DR_comp_glob}, and for simplicity we assume the underlying Sobolev spaces have some extra smoothness such that the interpolation are well-defined. Then the following commuting property holds for the standard FE spaces
\begin{equation}
\label{commut_1}
 \Pi_{h} \circ \nabla = \nabla \circ I _{h} ~~~~ \text{on} ~H^2_0(\Omega).
\end{equation}
Furthermore, by the assumption that $\partial\Omega$ has simple topology, the exact sequence also shows the following identity
\begin{equation}
\label{commut_2}
\begin{split}
\nabla S_{h,0}(\Omega) = \text{Ker(curl)}\cap\mathcal{ND}_{h,0}(\Omega) = \{ \bfz_h\in \mathcal{ND}_{h,0}(\Omega)~:~ \text{curl}~\bfz_h = 0\} 
\end{split}
\end{equation}
namely $\nabla S_{h,0}(\Omega)$ is the curl-free subspace of $\mathcal{ND}_{h,0}(\Omega)$. The proposed IFE spaces share the similar properties shown by the right plot in \eqref{DR_comp_glob} where the gradient $\nabla$ is understood element-wisely for IFE spaces. Here we note that functions in $\nabla\widetilde{H}^2_0(\Omega)$ satisfy the jump conditions \eqref{inter_jc_1} and \eqref{inter_jc_3} because of their own jump conditions \eqref{elli_jump_cd} and satisfy \eqref{inter_jc_2} because they are curl-free, and thus $\nabla\widetilde{H}^2_0(\Omega)\subseteq\widetilde{\bfH}^1_0(\text{curl};\Omega)$. The next lemma shows the diagram on the right of \eqref{DR_comp_glob} is well-defined.

\begin{lemma}
\label{lem_DR_1}
For IFE spaces, there holds $\nabla \widetilde{S}_{h,0}(\Omega) \subseteq \mathcal{IND}_{h,0}(\Omega)$.
\end{lemma}
\begin{proof}
We note that $\widetilde{S}_h(T)\subset \mathcal{IND}_h(T)$ on non-interface elements is trivial, and $\widetilde{S}_h(T)\subset \mathcal{IND}_h(T)$ on interface elements simply follows from the observation that $\nabla v_h$, $v_h\in \widetilde{S}_h(T)$, satisfies the jump conditions in \eqref{weak_jc}. But the global result is non-trivial due to the weaker regularity on interface edges, and the derivation is based on the continuity of $\widetilde{S}_h(\Omega)$ at mesh nodes. For each $v_h\in \widetilde{S}_h(\Omega)$ and each boundary edge $e$, we immediately have $\nabla v_h\cdot\bft_e=0$ where $\bft_e$ is the tangential vector of $e$ since the interface is assumed not to touch the boundary. For each edge $e\in\mathcal{E}_h$ not on boundary, we need to show $\int_e[\nabla v_h\cdot\bft_e]ds=0$. Again this is trivial for non-interface edges. For an interface edge $e$, we let $A_1$ and $A_2$ be the two nodes of $e$, let $\bft_e$ be oriented from $A_1$ to $A_2$ and let $T_1$ and $T_2$ be its neighbor elements. Then the continuity at the interface intersection point of $e$ yields
\begin{equation}
\label{lem_DR_1_eq1}
\int_e \nabla v_h|_{T_1}\cdot \bft_e ds = \int_e \partial_{\bft_e} v_h|_{T_1} ds = v_h|_{T_1}(A_2) - v_h|_{T_1}(A_2). 
\end{equation}
The similar identity also holds for $T_2$. Therefore the continuity at mesh nodes yields the desired result. 
\end{proof}

Now we can show the commuting property of the IFE space.

\begin{thm}
\label{thm_DR_4}
The diagram on the right of \eqref{DR_comp_glob} is commutative, namely $\widetilde{\Pi}_h \circ \nabla = \nabla \circ \tilde{I}_h$, $\text{on} ~ \widetilde{H}^2_0(\Omega)$.
\end{thm}
\begin{proof}
The result on non-interface elements is trivial, and we only discuss the interface element. On each interface element $T$ with the vertices $A_i$, $i=1,2,3$ and the edges $e_i$, $i=1,2,3$ as shown in Figure \ref{fig:sandwich}, we let $\bft_i$ be the tangential vectors of $e_i$ with the anti-clockwise orientation. Without loss of generality, we focus on the edge $e_1$. For each $v\in \widetilde{H}^2_0(\Omega)$, by the similar derivation to \eqref{lem_DR_1_eq1},
and the continuity of $v$, we have
\begin{equation*}
\int_{e_1}  \nabla \tilde{I}_{h,T} v\cdot\bft_1 ds = \tilde{I}_{h,T} v(A_3)-\tilde{I}_{h,T} v(A_2)  = v(A_3)-v(A_2) = \int_{e_1} \nabla v\cdot\bft_1 ds =  \int_{e_1} \widetilde{\Pi}_{h,T}(\nabla v)\cdot\bft_1 ds.
\end{equation*}
Similar arguments apply to other edges, and thus we have the desired result due to the unisolvence and Lemma \ref{lem_DR_1}.
\end{proof}
\begin{rem}
\label{rem_DR}
By similar arguments to Theorem \ref{thm_DR_4}, the following commuting properties also hold
\begin{equation}
\label{rem_DR_eq0}
\mathbb{\Pi}_h \circ\nabla = \nabla \circ  \mathbb{I}_h  ~~~ \text{on} ~ \widetilde{S}_h(\Omega) ~~~~~ \text{and} ~~~~~ \mathbb{\Pi}^{-1}_h \circ\nabla = \nabla \circ  \mathbb{I}^{-1}_h  ~~~ \text{on} ~ S_h(\Omega).
\end{equation}
\end{rem}
Furthermore, we can show that $\mathbb{\Pi}_h$ yields an isomorphism between $\text{Ker(curl)}\cap\mathcal{IND}_{h,0}(\Omega)$ and $\text{Ker(curl)}\cap\mathcal{ND}_{h,0}(\Omega)$.
\begin{thm}
\label{thm_curl0}
$\mathbb{\Pi}_h$ is an isomorphism between $\emph{Ker(curl)}\cap\mathcal{IND}_{h,0}(\Omega)$ and $\emph{Ker(curl)}\cap\mathcal{ND}_{h,0}(\Omega)$. 
\end{thm}
\begin{proof}
Let's focus on an interface element $T$, and recall that $\bfpsi_i$, $i=1,2,3$, are the $\bfH(\text{curl})$ IFE shape functions. By the identity in \eqref{thm_bound_eq02} we can let $\tau=\mu^{-1}\text{curl}~\bfpsi_i$, $i=1,2,3$. For each $\bfz_h\in \mathcal{ND}_{h,0}(\Omega)$, we note that $\mu^{-1}\text{curl}~ \mathbb{\Pi}^{-1}_{h,T} \bfz_h$ is a constant, and then the integration by parts yields
\begin{equation}
\begin{split}
\label{thm_curl0_eq1}
\frac{1}{\mu} \text{curl}~ \mathbb{\Pi}^{-1}_{h,T} \bfz_h = \sum_{i=1}^3 \int_{e_i} \bfz_h\cdot\bft_i ds \frac{1}{\mu} \text{curl}~ \bfpsi_i = \int_{\partial T} \bfz_h\cdot\bft \tau ds = \int_T \text{curl}~\bfz_h dX ~\tau = |T| \tau \text{curl}~\bfz_h.
\end{split}
\end{equation}
The identity above shows that $\text{curl}~\bfz_h=0$ if and only if $ \text{curl}~ \mathbb{\Pi}^{-1}_{h,T} \bfz_h = 0$, and similar results also hold on non-interface elements. Thus we have the desired result. 
\end{proof}
\begin{rem}
\label{rem_identity}
We note that $\text{Ker(curl)}\cap\mathcal{ND}_{h,0}(\Omega)$ consists of piecewise constant vectors; so does $\text{Ker(curl)}\cap\mathcal{IND}_{h,0}(\Omega)$, but the functions of the latter one on interface elements are piecewise constant vectors on each subelement. Moreover, if $\beta^-=\beta^+$ then the local IFE functions of piecewise constant vectors on each interface element will reduce to the same constant vector. Therefore, by the self-preserving property of the isomorphism $\mathbb{\Pi}_h$, if $\beta$ is continuous, there holds
\begin{equation}
\label{thm_curl0_eq0}
\mathbb{\Pi}_h|_{\emph{Ker(curl)}\cap\mathcal{IND}_h(\Omega)} = \mathcal{I}
\end{equation}
where $\mathcal{I}$ denotes the identity mapping.
\end{rem}
Another corollary of the isomorphism $\mathbb{\Pi}_h$ and $\mathbb{I}_h$ is property of the exact sequence similar to \eqref{commut_2}.


\begin{thm}
\label{thm_DR_3}
For IFE spaces, the sequence $\widetilde{S}_{h,0}(\Omega)\xrightarrow[]{\nabla} \mathcal{IND}_{h,0}(\Omega) \xrightarrow[]{\emph{curl}}Q_h$ is exact, where $Q_h\subset L^2(\Omega)$ is a piecewise constant space, namely
\begin{equation}
\label{thm_DR_3_eq0}
\nabla \widetilde{S}_{h,0}(\Omega) = \emph{Ker(curl)}\cap\mathcal{IND}_{h,0}(\Omega) = \{ \bfz_h\in \mathcal{IND}_{h,0}(\Omega)~:~ \emph{curl}~\bfz_h = 0\}.
\end{equation}
\end{thm}
\begin{proof}
By Lemma \ref{lem_DR_1}, it is easy to see $\nabla \widetilde{S}_{h,0}(\Omega)\subset  \text{Ker(curl)}\cap\mathcal{IND}_{h,0}(\Omega)$. As for the reverse direction, for each $\bfz_h\in\text{Ker(curl)}\cap\mathcal{IND}_{h,0}(\Omega)$, Theorem \ref{thm_curl0} suggests $\mathbb{\Pi}_h\bfz_h\in\text{Ker(curl)}\cap\mathcal{ND}_{h,0}(\Omega)$. Then due to the exact sequence for $S_{h,0}(\Omega)$ and $\mathcal{ND}_{h,0}(\Omega)$, there exists $s_h\in S_{h,0}(\Omega)$ such that $\nabla s_h = \mathbb{\Pi}_h\bfz_h$. Now we take $\tilde{s}_h=\mathbb{I}^{-1}_hs_h$, and use \eqref{rem_DR_eq0} to obtain $\nabla \tilde{s}_h = \nabla \mathbb{I}^{-1}_h s_h = \mathbb{\Pi}^{-1}_h \nabla s_h = \mathbb{\Pi}^{-1}_h \mathbb{\Pi}_h\bfz_h = \bfz_h$
which has finished the proof.
\end{proof}


\section{Approximation Capabilties}
\label{sec:analysis}
In this section, we analyze the approximation capabilities of the IFE space \eqref{IFE_glob_spa} through the interpolation $\widetilde{\Pi}_h$ in \eqref{interp_4}.
We note that the estimate on non-interface elements directly follows from \eqref{interp_1_err} since the standard N\'ed\'elec spaces are used, but the estimate on interface elements is in general one of the major challenges for IFE due to the insufficient regularity near the interface, and the weaker regularity for the considered problem makes the analysis even more difficult. 

\begin{figure}[h]
\centering
\begin{minipage}{.38\textwidth}
  \centering
  \includegraphics[width=2.2in]{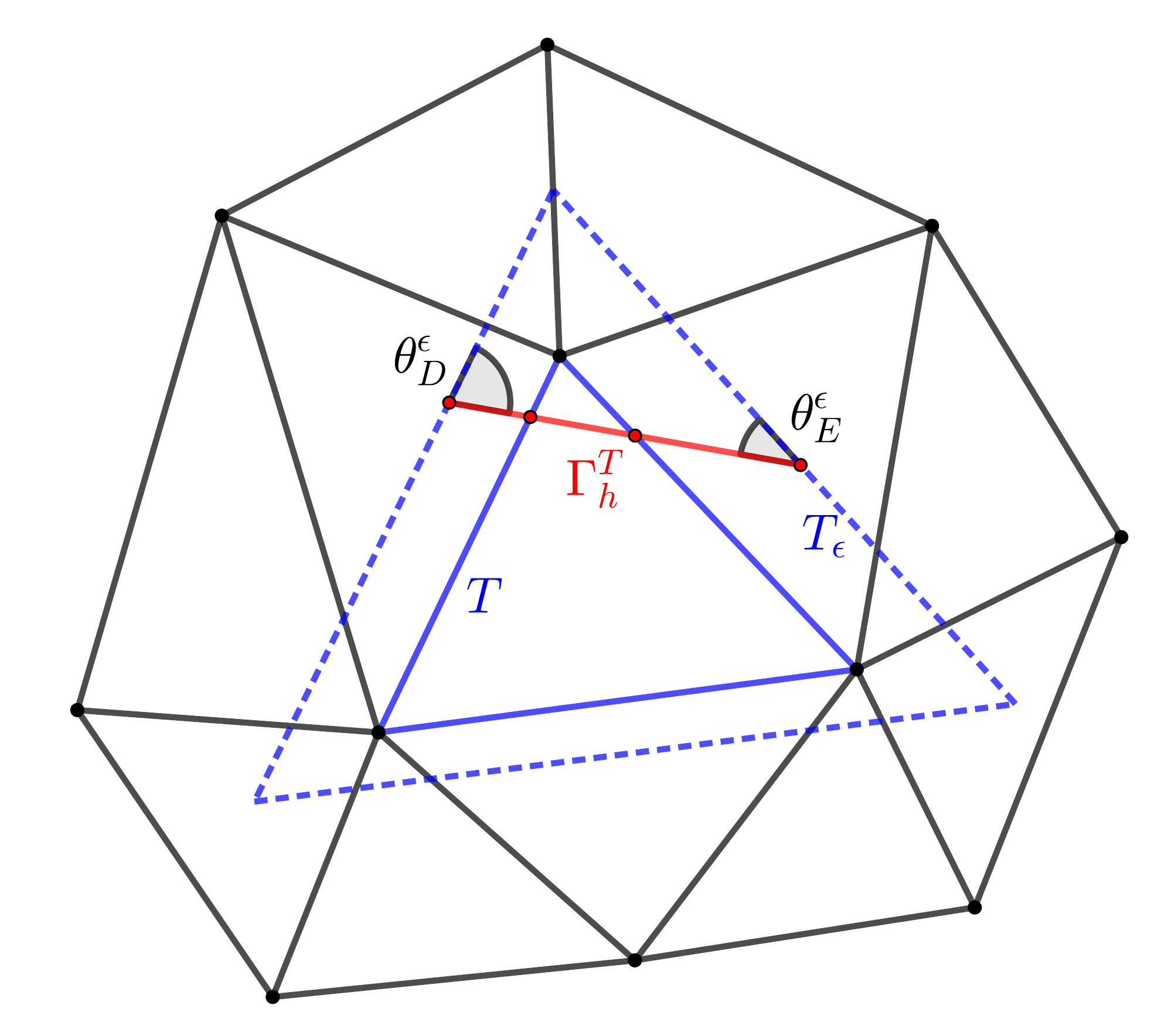}
  \caption{A fictitious element and a patch}
  \label{fig:patch}
\end{minipage}
\hspace{1cm}
\begin{minipage}{.4\textwidth}
  \centering
  \ \includegraphics[width=2in]{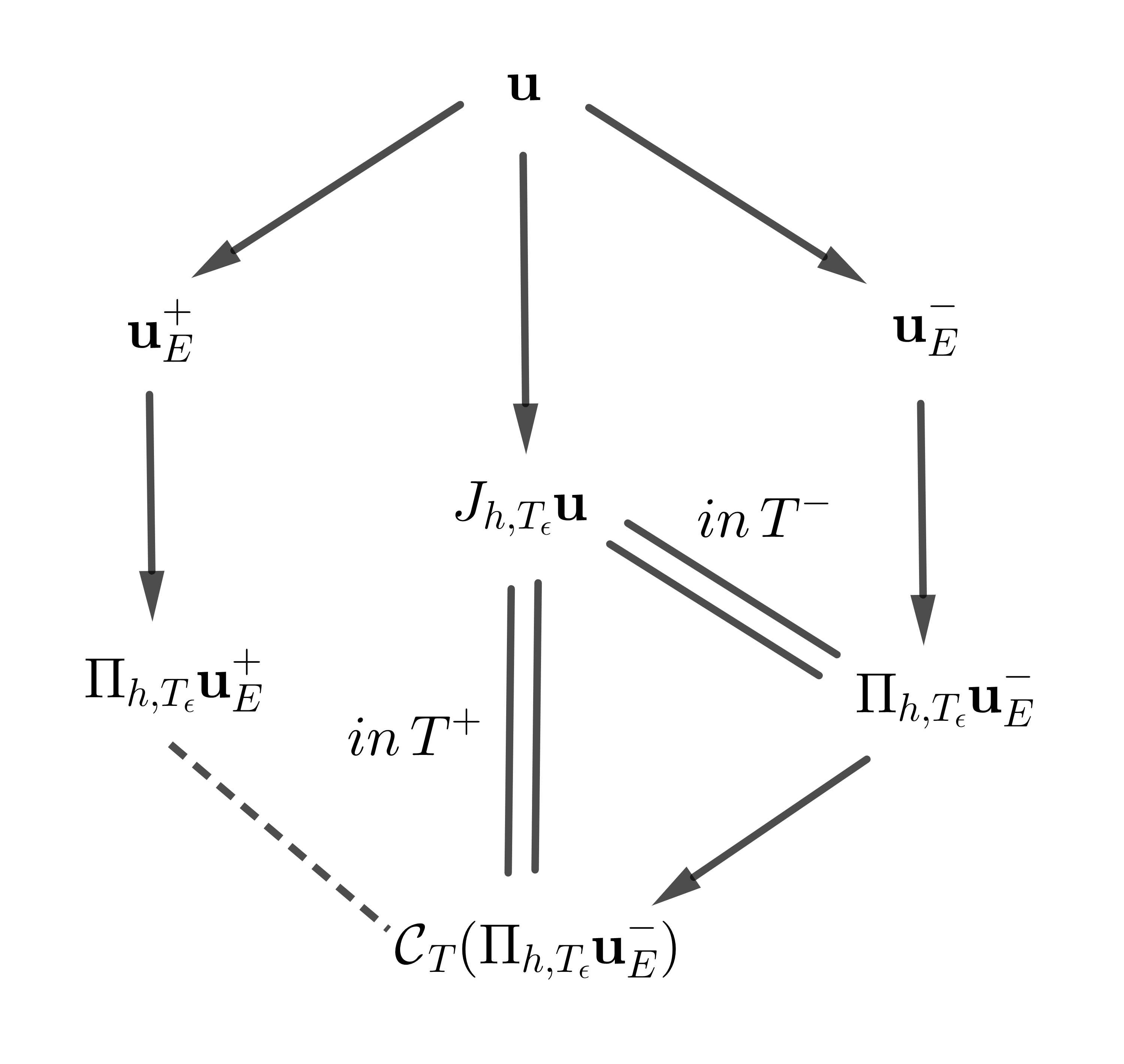}
  \caption{The diagram for interpolation errors}
  \label{fig:diagram}
  \end{minipage}
\end{figure}

In this article, we employ the general framework and techniques recently developed in \cite{2020AdjeridBabukaGuoLin,2019GuoLin}. For this purpose, we first introduce the patch $\omega_T$ of an element $T$
\begin{equation}
\label{patch_T}
\omega_T = \cup\{ T'\in\mathcal{T}_h~:~ T'\cap T \neq \emptyset \}.
\end{equation}
Then for each $T\in\mathcal{T}^i_h$ we define its fictitious element $T_{\epsilon}$ as its homothetic image:
\begin{equation}
\label{fic_elem}
T_{\epsilon} = \{ X\in\mathbb{R}^2~:~\exists Y\in T ~ \text{s.t.}~ \overrightarrow{OX}=\epsilon\overrightarrow{OY} \},
\end{equation}
where $O$ is the homothetic center which can be simply chosen as the centroid of $T$ and $\epsilon\ge1$ is a scaling factor. In the following analysis, we assume there exists a fixed $\epsilon_0>1$ such that for each $T$ there holds $T_{\epsilon_0}\subseteq \omega_T$. It is easy to see that this assumption is fulfilled if the mesh is regular, see the illustration by Figure \ref{fig:patch}. Without loss of generality, from now on we shall fix $\epsilon=\epsilon_0$ for all fictitious elements. The reason for using fictitious elements is that each subelement of $T_{\epsilon}$ has regular shape, namely their angles and length of sides are all bounded regardless of the interface location. We refer readers to Lemma 3.2 in \cite{2019GuoLin} for more details. This property is important for making the generic constants in all the analysis below independent of the interface location. In particular, we let $\Gamma^{T_{\epsilon}}_h$ be the extension of the straight line $\Gamma^T_h$ to $T_{\epsilon}$ and let $\theta_D^{\epsilon}$ and $\theta_E^{\epsilon}$ be the angles formed by the edges of the fictitious element and $\Gamma^{T_\epsilon}_h$ as shown in Figure \ref{fig:patch}. Then Lemma 3.2 in \cite{2019GuoLin} yields the two constants $\delta_1$ and $\delta_2>0$ independent of the interface location such that for every interface element and its fictitious element, 
\begin{equation}
\label{interf_ext}
|\Gamma^{T_{\epsilon}}_h| \ge \delta_1 h_T  ~~~~~~ \text{and} ~~~~~~ \theta^{\epsilon}_D,\theta^{\epsilon}_E\ge \delta_2.
\end{equation}

Now we are ready to define the special interpolation operator
\begin{equation}
\begin{split}
\label{interp_2}
J_{h,T_{\epsilon}}~:~ \widetilde{\bfH}^1(\text{curl};T_{\epsilon}) \longrightarrow \mathcal{IND}_h(T_{\epsilon})  ~~~~
 \text{with} ~~~~  J_{h,T_{\epsilon}} \bfu =
 \begin{cases}
      & J^-_{h,T_{\epsilon}} \bfu = \Pi_{h,T_{\epsilon}} \bfu^-_E ~~~~~~~~~~~ \text{in} ~ T^-_{\epsilon}, \\
      & J^+_{h,T_{\epsilon}} \bfu = \mathcal{C}_T(\Pi_{h,T_{\epsilon}} \bfu^-_E) ~~~~~ \text{in} ~ T^+_{\epsilon} ,
\end{cases}
\end{split}
\end{equation}
where $\Pi_{h,T_{\epsilon}} \bfu^-_E$ is a polynomial generated on $T_{\epsilon}$ but used onto $T^-$ to apply $\mathcal{C}_T$. Actually, since polynomials can be naturally extended to everywhere, in the following analysis we can and always use $\Pi_{h,T_{\epsilon}} \bfu^-_E$, $\mathcal{C}_T(\Pi_{h,T_{\epsilon}} \bfu^-_E)$ (those polynomials defined on subelements) on the whole $T_{\epsilon}$. Moreover, we note that $\mathcal{ND}_h(T^{\pm})$ and $\mathcal{ND}_h(T^{\pm}_{\epsilon})$ consist of the same polynomials, and thus we shall only use $\mathcal{ND}_h(T_{\epsilon})$ to denote the  N\'ed\'elec space associated with the element $T$ for simplicity.
The motivation behind the special interpolation operator \eqref{interp_2} to analyze the approximation capabilities is a relation between different extension and interpolation operators including $\bfE^{\pm}_{\text{curl}}$, $\mathcal{C}_T$ and $\Pi_{h,T_{\epsilon}}$, illustrated by the diagram in Figure \ref{fig:diagram}. This diagram actually suggests a delicate decomposition of the interpolation error $J_{h,T_{\epsilon}}\bfu-\bfu$ into the errors of $\Pi_{h,T_{\epsilon}}\bfu^{\pm}_E - \bfu^{\pm}_E$ and the error from $\mathcal{C}_T$, i.e., $\Pi_{h,T_{\epsilon}}\bfu^+_E - \mathcal{C}_T(\Pi_{h,T_{\epsilon}}\bfu^-_E)$ in which only the latter one is unclear indicated by a dashed line in this diagram and will be analyzed below while the rest are well-established indicated by the solid arrows.

For each interface element $T\in\mathcal{T}^i_h$ and the associated fictitious element $T_{\epsilon}$ we prepare the following lemma.
\begin{lemma}
\label{quiv_jc_3}
There exists a function $\rho\in L^{\infty}(\Gamma^{T_{\epsilon}}_h)$ such that $\|\rho\|_{L^{\infty}(\Gamma^{T_{\epsilon}}_h)}\le Ch^{-1}_T$ and
\begin{equation}
\label{quiv_jc_3_eq0}
\int_{\Gamma^{T_{\epsilon}}_h} \rho (\bfv_h\cdot\bar{\bfn}) (\bfw_h\cdot\bar{\bfn}) ds = (\bfv_h(X_m)\cdot\bar{\bfn}) (\bfw_h(X_m)\cdot\bar{\bfn}) ~~~~~ \forall \bfv_h,\bfw_h\in \mathcal{ND}_h(T).
\end{equation}
\end{lemma}
\begin{proof}
We consider a local Cartesian system in which the origin is one of the intersection points of $\Gamma^{T_{\epsilon}}_h$, the $\xi_1$-axis is the straight line $\Gamma^{T_{\epsilon}}_h$ and the $\xi_2$-axis is the one perpendicular to $\xi_1$. We let $\xi_m$ be the point corresponding to $X_m$ and let $l=|\Gamma^{T_{\epsilon}}_h|$. Note that for each $\bfv_h\in\mathcal{ND}_h(T)$, $\bfv_h\cdot\bar{\bfn}$ restricted onto $\Gamma^{T_{\epsilon}}_h$ belongs $\mathbb{P}_1(\Gamma^{T_{\epsilon}}_h)$. Therefore we only need to show the existence of $\rho$ such that
\begin{equation}
\label{quiv_jc_3_eq1}
 \int^l_0 \rho v d\xi_1 = v(\xi_m)~~~~~ \forall v\in \mathbb{P}_2(0,l).
\end{equation}
Using the scaling argument $\hat{\xi}_1=\xi_1/l$, we clearly see the existence of $\hat{\rho}\in\mathbb{P}_2(0,1)$ such that
$
\int^1_0 \hat{\rho} \hat{v} d\hat{\xi}_1 = \hat{v}(\hat{\xi}_m)
$
$\forall\hat{v}\in\mathbb{P}_2(0,1)$ since the induced mass matrix is certainly non-singular. Then it is easy to see $\| \hat{\rho} \|_{L^{\infty}(0,1)}\le C$. Hence taking $\rho=\hat{\rho}l^{-1}$ fulfills \eqref{quiv_jc_3_eq1} and $\|\rho\|_{L^{\infty}(\Gamma^{T_{\epsilon}}_h)}\le Ch^{-1}_T$ due to the first inequality in \eqref{interf_ext}.
\end{proof}
Lemma \ref{quiv_jc_3} enables us to present an equivalent description for \eqref{weak_jc_3} which is useful in the analysis.

\begin{lemma}
\label{lem_equiv_jc}
For each $\bfv_h\in \mathcal{ND}_h(T_{\epsilon})$, the third condition of $\mathcal{C}_T$ defined in \eqref{weak_jc_3} is equivalent to
\begin{equation}
\label{weak_eqv_jc_3}
\int_{\Gamma^{T_{\epsilon}}_h} \rho \beta^+( \mathcal{C}_T(\bfv_h)\cdot\bar{\bfn})(\bfw_h\cdot\bar{\bfn})  ds  = \int_{\Gamma^{T_{\epsilon}}_h} \rho \beta^- (\bfv_h\cdot\bar{\bfn} )(\bfw_h\cdot\bar{\bfn}) ds ~~~~ \forall \bfw_h\in\mathcal{ND}_h(T_{\epsilon}). 
\end{equation}
\end{lemma}
\begin{proof}
The identity directly follows from the definition of $\rho$ in Lemma \ref{quiv_jc_3}.
\end{proof}
Next we define the following special norm which is suitable to handle the jump conditions.
\begin{equation}
\label{spec_norm}
\vertiii{ \bfv_h }^2_{T_{\epsilon}} = h_T | \bfv_h(X_m)\cdot\bar{\bfn} |^2 + \| \bfv_h\cdot\bar{\bft} \|^2_{L^2(\Gamma^{T_{\epsilon}}_h)} + \| \text{curl}~\bfv_h \|^2_{L^2(\Gamma^{T_{\epsilon}}_h)}, ~~~ \forall \bfv_h\in\bfH^1(\text{curl};T_{\epsilon}),
\end{equation}

\begin{lemma}
\label{lem_norm_equiv}
The norm equivalence $h_T^{1/2}\vertiii{\cdot}_{T_{\epsilon}}\simeq \|\cdot\|_{\bfH(\emph{curl};T_{\epsilon})}$ holds on $\mathcal{ND}_h(T_{\epsilon})$ where the hidden constant is independent of the interface location.
\end{lemma}
\begin{proof}
Given each $\bfv_h\in\mathcal{ND}_h(T_{\epsilon})$, since $\bfv_h$ is simply a polynomial, using the first inequality in \eqref{interf_ext}, Lemma \ref{quiv_jc_3} and the trace inequality together with the inverse inequality, we have
\begin{equation*}
\begin{split}
\label{lem_norm_equiv_eq1}
& h^{1/2}_T |\bfv_h(X_m)\cdot\bar{\bfn} | \le \|\bfv_h\cdot\bar{\bfn} \|_{L^2(\Gamma^{T_{\epsilon}}_h)} \le C h^{-1/2}_T \| \bfv_h \|_{L^2(T_{\epsilon})}, \\
&\| \bfv_h \cdot\bar{\bft} \|_{L^2(\Gamma^{T_{\epsilon}}_h)} \le C h_T^{-1/2}\| \bfv_h \|_{L^2(T_{\epsilon})}, ~~~ \| \text{curl} ~ \bfv_h \|_{L^2(\Gamma^{T_{\epsilon}}_h)} \le Ch_T^{-1/2} \| \text{curl} ~ \bfv_h \|_{L^2(T_{\epsilon})}. 
\end{split}
\end{equation*}
This yields $\vertiii{\bfv_h}_{T_{\epsilon}} \le Ch^{-1/2}_T \| \bfv_h \|_{\bfH(\text{curl};T_{\epsilon})}$. For the reverse direction, by Taylor expansion, we have
\begin{equation}
\label{lem_norm_equiv_eq2}
\bfv_h(X) = \bfv_h(X_m) + \nabla \bfv_h(X_m)(X-X_m) = \bfv_h(X_m)\cdot\bar{\bfn} + \bfv_h(X_m)\cdot\bar{\bft}  - \frac{1}{2}\text{curl}~\bfv_h ~ \bfK( X - X_m)
\end{equation} 
where $\bfK=[0,1;-1,0]\in\mathbb{R}^{2\times2}$ since $\text{curl}~\bfv_h$ is a constant. Then we have
\begin{equation}
\label{lem_norm_equiv_eq3}
\| \bfv_h \|_{L^2(T_{\epsilon})} \le Ch_T \verti{ \bfv_h(X_m)\cdot\bar{\bfn} } + Ch_T \verti{ \bfv_h(X_m)\cdot\bar{\bft} } + Ch_T^2\verti{ \text{curl}~\bfv_h }.
\end{equation}
We note that $\bfv_h\cdot\bar{\bft}$ can be understood as a polynomial defined on $\Gamma^{T_{\epsilon}}_h$, and $|\text{curl}~\bfv_h|$ is a constant, and thus they can be simply bounded by the standard trace inequality on $\Gamma^{T_{\epsilon}}_h$. Therefore, we have
\begin{equation}
\label{lem_norm_equiv_eq5}
\| \bfv_h \|_{L^2(T_{\epsilon})} \le Ch_T \verti{ \bfv_h(X_m)\cdot\bar{\bfn} } + Ch_T^{1/2} \| \bfv_h \cdot\bar{\bft} \|_{L^2(\Gamma^{T_{\epsilon}}_h)} + Ch^{3/2}_T\| \text{curl}~\bfv_h \|_{L^2(\Gamma^{T_{\epsilon}}_h)} \le Ch^{1/2}_T \vertiii{\bfv_h }_{T_{\epsilon}}.
\end{equation}
Finally, we also have $\| \text{curl} ~ \bfv_h \|_{L^2(T_{\epsilon})} = (|T|/|\Gamma^{T_{\epsilon}}_h|)^{1/2} \| \text{curl} ~ \bfv_h \|_{L^2(\Gamma^{T_{\epsilon}}_h)} \le Ch^{1/2}_T  \| \text{curl} ~ \bfv_h \|_{L^2(\Gamma^{T_{\epsilon}}_h)} $ because of the mesh regularity and the first inequality in \eqref{interf_ext}, which has finished the proof.
\end{proof}

Since a linear approximation of the interface is used for constructing IFE functions, we need to estimate the jumps on this approximated interface.
\begin{lemma}
\label{lem_u_gam_h}
For $\bfu\in\bfH^1(\emph{curl};\Omega)$ and for each interface element $T$ and the associated $T_{\epsilon}$, there holds
\begin{subequations}
\label{lem_u_gam_eq0}
\begin{align}
    &  \| \bfu^+_E\cdot\bar{\bft} - \bfu^-_E\cdot\bar{\bft} \|_{L^2(T_{\epsilon})} \le Ch_T \left( \| \bfu^+_E \|_{H^1(T_{\epsilon})} + \| \bfu^-_E \|_{H^1(T_{\epsilon})} \right),  \label{lem_u_gam_eq01} \\
    &  \| \beta^+ \bfu^+_E\cdot\bar{\bfn} - \beta^- \bfu^-_E\cdot\bar{\bfn} \|_{L^2(T_{\epsilon})} \le C h_T \left( \| \bfu^+_E \|_{H^1(T_{\epsilon})} + \| \bfu^-_E \|_{H^1(T_{\epsilon})} \right),  \label{lem_u_gam_eq02} \\
    &  \| (\mu^{+})^{-1} \emph{curl}~\bfu^+_E - (\mu^{-})^{-1} \emph{curl}~\bfu^-_E \|_{L^2(T_{\epsilon})} \le C h_T \left( \| \emph{curl}~ \bfu^+_E \|_{H^1(T_{\epsilon})} + \| \emph{curl}~ \bfu^-_E \|_{H^1(T_{\epsilon})} \right).  \label{lem_u_gam_eq03}
\end{align}
\end{subequations}
\end{lemma}
\begin{proof}
Since the proof is similar to Lemma 3.3 in \cite{2019GuoLin} through a strip argument, we only present the proof for \eqref{lem_u_gam_eq01} for simplicity, and put it in the Appendix \ref{app_4}.
\end{proof}
Now we can estimate the error $\Pi_{h,T_{\epsilon}}\bfu^+_E - \mathcal{C}_T(\Pi_{h,T_{\epsilon}}\bfu^-_E)$ indicted by the dashed line of the diagram in Figure \ref{fig:diagram}.

\begin{lemma}
\label{lem_interp_error_ife_1}
Suppose $\bfu\in \bfH^1(\emph{curl};\Omega)$, then for each element $T$ and the associated $T_{\epsilon}$
\begin{equation}
\label{lem_interp_error_ife_1_eq0}
\vertiii{ \Pi_{h,T_{\epsilon}}\bfu^+_E - \mathcal{C}_T(\Pi_{h,T_{\epsilon}}\bfu^-_E) }_{T_{\epsilon}} \le C  h^{1/2}_T \left( \| \bfu_E^+ \|_{\bfH^1(\emph{curl};\omega_T)} +  \| \bfu_E^- \|_{\bfH^1(\emph{curl};\omega_T)}  \right).
\end{equation}
\end{lemma}
\begin{proof}
Let
$\bfw_h= \Pi_{h,T_{\epsilon}}\bfu^+_E - \mathcal{C}_T(\Pi_{h,T_{\epsilon}}\bfu^-_E) \in \mathcal{ND}_h(T_{\epsilon})$, and we need to estimate each term in the definition \eqref{spec_norm}. First of all, since $\bfw_h\in\mathcal{ND}_h(T_{\epsilon})$, we use Lemmas \ref{quiv_jc_3} and \ref{lem_equiv_jc} to obtain
\begin{equation}
\begin{split}
\label{lem_interp_error_ife_1_eq1}
|\bfw_h(X_m)\cdot\bar{\bfn}|^2 &=  \| \sqrt{\rho} \bfw_h \cdot\bar{\bfn} \|^2_{L^2(\Gamma^{T_{\epsilon}}_h)} = \int_{L^2(\Gamma^{T_{\epsilon}}_h)} \left(\Pi_{h,T_{\epsilon}}\bfu^+_E -  \frac{\beta^-}{\beta^+} \Pi_{h,T_{\epsilon}}\bfu^-_E \right)\cdot\bar{\bfn} ~ \rho\bfw_h \cdot\bar{\bfn} ds \\
 & \le \vertii{ \sqrt{\rho} \left(\Pi_{h,T_{\epsilon}}\bfu^+_E -  \frac{\beta^-}{\beta^+} \Pi_{h,T_{\epsilon}}\bfu^-_E \right)\cdot\bar{\bfn} }_{L^2(\Gamma^{T_{\epsilon}}_h)}  \| \sqrt{\rho} \bfw_h \cdot\bar{\bfn} \|_{L^2(\Gamma^{T_{\epsilon}}_h)}. 
 \end{split}
 \end{equation}
We note that $\| \sqrt{\rho} \bfw_h\cdot\bar{\bfn} \|^2_{L^2(\Gamma^{T_{\epsilon}}_h)}=(\bfw_h(X_m)\cdot\bar{\bfn})^2\ge0$ suggests that the norm $\| \sqrt{\rho} \bfw_h\cdot\bar{\bfn} \|_{L^2(\Gamma^{T_{\epsilon}}_h)}$ is well-defined. We also note that $\rho$ may not be positive and thus $\sqrt{\rho}$ may be complex, but we emphasize that this notation is only used for convenience since in all the derivation above only $\rho$ itself appears. Since $\Pi_{h,T_{\epsilon}}\bfu^+_E -  \frac{\beta^-}{\beta^+} \Pi_{h,T_{\epsilon}}\bfu^-_E $ is a polynomial, applying the trace inequality for polynomials \cite{2003WarburtonHesthaven} and the bound for $\rho$ in Lemma \ref{quiv_jc_3}, we induce from \eqref{lem_interp_error_ife_1_eq1}
 \begin{equation*}
 \begin{split}
\label{lem_interp_error_ife_1_eq2}
& |\bfw_h(X_m)\cdot\bar{\bfn}|  \le C h^{-1}_T \vertii{ \left(\Pi_{h,T_{\epsilon}}\bfu^+_E -  \frac{\beta^-}{\beta^+} \Pi_{h,T_{\epsilon}}\bfu^-_E \right)\cdot\bar{\bfn} }_{L^2(T_{\epsilon})} \\
 \le & C h^{-1}_T \left( \vertii{ \left(\Pi_{h,T_{\epsilon}}\bfu^+_E - \bfu^+_E \right)\cdot\bar{\bfn} }_{L^2(T_{\epsilon})} + \vertii{  \left( \frac{\beta^-}{\beta^+} \Pi_{h,T_{\epsilon}}\bfu^-_E - \frac{\beta^-}{\beta^+} \bfu^-_E \right)\cdot\bar{\bfn} }_{L^2(T_{\epsilon})} +  \vertii{ \left( \frac{\beta^-}{\beta^+} \bfu^-_E - \bfu^+_E \right)\cdot\bar{\bfn} }_{L^2(T_{\epsilon})} \right).
\end{split}
\end{equation*}
Applying the estimate \eqref{interp_1_err} for $\Pi_{h,T_{\epsilon}}$ on $T_{\epsilon}$ and \eqref{lem_u_gam_eq02} to the estimate above, we have
\begin{equation}
\label{lem_interp_error_ife_1_eq3}
|\bfw_h(X_m)\cdot\bar{\bfn}| \le C  \left( \| \bfu^+_E \|_{\bfH^1(\text{curl};T_{\epsilon})} + \| \bfu^-_E \|_{\bfH^1(\text{curl};T_{\epsilon})} \right).
\end{equation}
By similar arguments, using \eqref{lem_u_gam_eq01} and \eqref{lem_u_gam_eq03}, we have the following estimates
\begin{equation}
\begin{split}
\label{lem_interp_error_ife_1_eq5}
&\| \bfw_h \cdot\bar{\bft} \|_{L^2(\Gamma^{T_{\epsilon}}_h)} \le Ch^{1/2}_T \left( \| \bfu^+_E \|_{\bfH^1(\text{curl};T_{\epsilon})} + \| \bfu^-_E \|_{\bfH^1(\text{curl};T_{\epsilon})} \right), \\
&\| \text{curl}~\bfw_h \|_{L^2(\Gamma^{T_{\epsilon}}_h)} \le C h^{1/2}_T \left( \| \bfu^+_E \|_{\bfH^1(\text{curl};T_{\epsilon})} + \| \bfu^-_E \|_{\bfH^1(\text{curl};T_{\epsilon})} \right).
\end{split}
\end{equation}
Then the desired result follows from \eqref{lem_interp_error_ife_1_eq3}-\eqref{lem_interp_error_ife_1_eq5} together with the assumption that $T_{\epsilon}\subseteq\omega_T$.
\end{proof}

\begin{rem}
\label{rem_lem_interp_error_ife_1}
Lemma \ref{lem_interp_error_ife_1} is the major result towards the approximation capabilities of the $\bfH(\emph{curl};\Omega)$ IFE spaces. One of the keys in the proof is to represent a point-valued norm as a weighted $L^2$-norm as shown in \eqref{lem_interp_error_ife_1_eq1} through Lemma \ref{quiv_jc_3}. This treatment enables us to bound the error by the $\bfH^1(\emph{curl};T_{\epsilon})$-norm; otherwise $\bfH^2$ regularity is needed.
%
\end{rem}


Next, we use the idea of the diagram \ref{fig:diagram} to prove the following interpolation error estimate.
\begin{thm}
\label{thm_interp_err_2}
Suppose $\bfu\in \bfH^1(\emph{curl};\Omega)$, then on each interface element $T$ and $T_{\epsilon}$
\begin{equation}
\label{thm_interp_err_2_eq0}
\| J^{\pm}_{h,T_{\epsilon}} \bfu - \bfu^{\pm}_E \|_{\bfH(\emph{curl};T_{\epsilon})} \le C h_T \left( \| \bfu_E^+ \|_{\bfH^1(\emph{curl};\omega_T)} +  \| \bfu_E^- \|_{\bfH^1(\emph{curl};\omega_T)}  \right).
\end{equation}
\end{thm}
\begin{proof}
By the definition in \eqref{interp_2}, the estimate of
$J^-_{h,T_{\epsilon}} \bfu - \bfu^-_E= \Pi_{h,T_{\epsilon}} \bfu^-_E -  \bfu^-_E $ in $T^-_{\epsilon}$ directly follows from applying \eqref{interp_1_err} to $\Pi_{h,T_{\epsilon}}$ which is simply the right side of the diagram in Figure \ref{fig:diagram}. So we only need to estimate the error on $T^+_{\epsilon}$. We first note the following error decomposition
\begin{equation}
\begin{split}
\label{thm_interp_err_2_eq1}
\| J^+_{h,T_{\epsilon}} \bfu - \bfu^+_E \|_{\bfH(\text{curl};T^+_{\epsilon})} 
\le  \| \mathcal{C}_T(\Pi_{h,T_{\epsilon}} \bfu^-_E)- \Pi_{h,T_{\epsilon}} \bfu^+_E \|_{\bfH(\text{curl};T^+_{\epsilon})} + 
\| \Pi_{h,T_{\epsilon}} \bfu^+_E -  \bfu^+_E  \|_{\bfH(\text{curl};T^+_{\epsilon})}.
\end{split}
\end{equation}
Again the second term in the right hand side of \eqref{thm_interp_err_2_eq1} follows from applying \eqref{interp_1_err} to $\Pi_{h,T_{\epsilon}}$. For the first term, using the norm equivalence in Lemma \ref{lem_norm_equiv} together with the estimate in Lemma \ref{lem_interp_error_ife_1}, we have
\begin{equation}
\begin{split}
\label{thm_interp_err_2_eq2}
\| \mathcal{C}_T(\Pi_{h,T_{\epsilon}} \bfu^-_E)- \Pi_{h,T_{\epsilon}} \bfu^+_E \|_{\bfH(\text{curl};T^+_{\epsilon})}
& \le Ch^{1/2}_T \vertiii{ \mathcal{C}_T(\Pi_{h,T_{\epsilon}} \bfu^-_E)- \Pi_{h,T_{\epsilon}} \bfu^+_E }_{T_{\epsilon}} \\
& \le  Ch_T \left( \| \bfu_E^+ \|_{\bfH^1(\text{curl};\omega_T)} +  \| \bfu_E^- \|_{\bfH^1(\text{curl};\omega_T)}  \right)
\end{split}
\end{equation}
which has finished the proof by the assumption that $T_{\epsilon}\subseteq\omega_T$.
\end{proof}

We note that Theorem \ref{thm_interp_err_2} already guarantees that the local IFE spaces have optimal approximation capabilities on interface elements. However in order to estimate the approximation capabilities of the global IFE space $\mathcal{IND}_h(\Omega)$, we need to employ the interpolation operators $\widetilde{\Pi}_{h,T}$ and $\widetilde{\Pi}_h$ in \eqref{interp_4}.

\begin{thm}
\label{thm_interp_err_3}
Suppose $\bfu\in \bfH^1(\emph{curl};\Omega)$, then for each interface element $T$
\begin{equation}
\label{thm_interp_err_3_eq0}
\| \widetilde{\Pi}_{h,T} \bfu - \bfu \|_{L^2(T)} \le C h_T \left( \| \bfu_E^+ \|_{\bfH^1(\emph{curl};\omega_T)} +  \| \bfu_E^- \|_{\bfH^1(\emph{curl};\omega_T)}  \right).
\end{equation}
\end{thm}
\begin{proof}
Given an interface element $T$ with the edges $e_i$, $i=1,2,3$, the triangular inequality yields
\begin{equation}
\label{thm_interp_err_3_eq1}
\| \widetilde{\Pi}_{h,T} \bfu - \bfu \|_{L^2(T)} \le \| \widetilde{\Pi}_{h,T} \bfu - J_{h,T_{\epsilon}}\bfu \|_{L^2(T)} + \| J_{h,T_{\epsilon}}\bfu - \bfu \|_{L^2(T)}.
\end{equation}
The estimate of the second term in \eqref{thm_interp_err_3_eq1} directly follows from Theorem \ref{thm_interp_err_2}. So we only need to estimate the first term. Note that $\widetilde{\Pi}_{h,T} \bfu - J_{h,T_{\epsilon}}\bfu = \widetilde{\Pi}_{h,T}\left( \bfu - J_{h,T_{\epsilon}}\bfu \right)$, and we then consider a piecewise-defined function $\bfw = \bfu_E - J_{h,T_{\epsilon}} \bfu$ with $\bfw^{\pm} = \bfu^{\pm}_E - J^{\pm}_{h,T_{\epsilon}} \bfu$ in $T^{\pm}_h$ partitioned by $\Gamma^T_h$. We note that $\bfw$ slightly differs from $\bfu - J_{h,T_{\epsilon}} \bfu$ only in $\widetilde{T}$ since $\bfu$ is partitioned by the interface $\Gamma$ itself where $\widetilde{T}$ is the subelement sandwiched by $\Gamma^T_h$ and $\Gamma$ shown by Figure \ref{fig:sandwich}, but we only need that these two functions equal on $\partial T$. Using the IFE shape functions in \eqref{IFE_fun_1}, we can write
\begin{equation}
\label{thm_interp_err_3_eq2}
\widetilde{\Pi}_{h,T}(\bfu - J_{h,T_{\epsilon}}\bfu) = \sum_{i=1}^3 \int_{e_i} (\bfu - J_{h,T_{\epsilon}}\bfu)\cdot\bft_i ds \bfpsi_i = \sum_{i=1}^3 \int_{e_i} \bfw\cdot\bft_i ds \bfpsi_i .
\end{equation}
By H\"older's inequality and the boundedness \eqref{thm_bound_eq01}, we have
\begin{equation}
\begin{split}
\label{thm_interp_err_3_eq3}
\| \widetilde{\Pi}_{h,T} (\bfu - J_{h,T_{\epsilon}}\bfu) \|_{L^2(T)} & \le \sum_{i=1}^3 \verti{ \int_{e_i} \bfw\cdot\bft_i ds } \| \bfpsi_i \|_{L^2(T)} \\
& \le C h^{1/2}_T \sum_{i=1}^3  \| \bfw \|_{L^2(e_i)}  \le C h^{1/2}_T \sum_{i=1}^3 \left( \| \bfw^- \|_{L^2(e_i)} + \| \bfw^+ \|_{L^2(e_i)} \right).
\end{split}
\end{equation}
Then by the scaling arguments for N\'ed\'elec elements in \cite{1999CiarletZou}(Lemma 3.2), we have
\begin{equation}
\label{thm_interp_err_3_eq4}
\| \bfw^- \|_{L^2(e_i)} = \vertii{ \Pi_{h,T_{\epsilon}} \bfu^-_E  - \bfu^-_E }_{L^2(e_i)} \le Ch_T^{1/2} \left( \| \bfu_E^+ \|_{\bfH^1(\text{curl};T_{\epsilon})} +  \| \bfu_E^- \|_{\bfH^1(\text{curl};T_{\epsilon})}  \right).
\end{equation}
In addition, the triangular inequality yields
\begin{equation}
\begin{split}
\label{thm_interp_err_3_eq5}
\| \bfw^+ \|_{L^2(e_i)} &= \vertii{  \mathcal{C}_T(\Pi_{h,T_{\epsilon}} \bfu^-_E) -\bfu^+_E }_{L^2(e_i)} \le  \vertii{  \mathcal{C}_T(\Pi_{h,T_{\epsilon}} \bfu^-_E) - \Pi_{h,T_{\epsilon}} \bfu^+_E }_{L^2(e_i)}
+ \vertii{  \Pi_{h,T_{\epsilon}} \bfu^+_E - \bfu^+_E }_{L^2(e_i)}.
\end{split}
\end{equation}
Since $ \mathcal{C}_T(\Pi_{h,T_{\epsilon}} \bfu^-_E) - \Pi_{h,T_{\epsilon}}\bfu^+_E$ is a polynomial, by the trace inequality for polynomials \cite{2003WarburtonHesthaven}, the estimate in Lemma \ref{lem_interp_error_ife_1} and the norm equivalence in Lemma \ref{lem_norm_equiv}, we have
\begin{equation}
\begin{split}
\label{thm_interp_err_3_eq6}
\vertii{  \mathcal{C}_T(\Pi_{h,T_{\epsilon}} \bfu^-_E) - \Pi_{h,T_{\epsilon}} \bfu^+_E }_{L^2(e_i)} & \le Ch^{-1/2}_T \vertii{  \mathcal{C}_T(\Pi_{h,T_{\epsilon}} \bfu^-_E) - \Pi_{h,T_{\epsilon}} \bfu^+_E }_{L^2(T)} \\
& \le  Ch_T^{1/2} \left( \| \bfu_E^+ \|_{\bfH^1(\text{curl};T_{\epsilon})} +  \| \bfu_E^- \|_{\bfH^1(\text{curl};T_{\epsilon})}  \right).
\end{split}
\end{equation}
The estimate for $ \vertii{  \Pi_{h,T_{\epsilon}} \bfu^+_E - \bfu^+_E }_{L^2(e_i)}$ is as the same as \eqref{thm_interp_err_3_eq4}. Putting these two estimates into \eqref{thm_interp_err_3_eq5} and combining it with \eqref{thm_interp_err_3_eq4}, we have 
\begin{equation}
\label{thm_interp_err_3_eq7}
\| \widetilde{\Pi}_{h,T}\bfw \|_{L^2(T)} \le Ch_T \left( \| \bfu_E^+ \|_{\bfH^1(\text{curl};T_{\epsilon})} +  \| \bfu_E^- \|_{\bfH^1(\text{curl};T_{\epsilon})}  \right)
\end{equation}
which has finished the proof by the assumption that $T_{\epsilon}\subseteq\omega_T$. 
\end{proof}

As for $\text{curl}~(\widetilde{\Pi}_{h,T} \bfu - \bfu)$, if we directly apply integration by parts to the curl of \eqref{thm_interp_err_3_eq2}, it may cause some troubles around the interface. So we employ a rather different argument based on the Soblev embedding theorem which is inspired by the works in \cite{2004LiLinLinRogers}. Since $\text{curl}~\bfu\in H^1(\Omega)$, the Soblev embedding theorem suggests that $\text{curl}~\bfu\in L^3(\Omega)$. Then we have the following estimate.

\begin{thm}
\label{thm_interp_err_4}
Suppose $\bfu\in \bfH^1(\emph{curl};\Omega)$, then
\begin{equation}
\label{thm_interp_err_4_eq0}
\| \emph{curl}(\widetilde{\Pi}_{h,T} \bfu - \bfu) \|_{L^2(T)} \le C h_T \left( \| \bfu_E^+ \|_{\bfH^1(\emph{curl};\omega_T)} +  \| \bfu_E^- \|_{\bfH^1(\emph{curl};\omega_T)} + \| \emph{curl}~\bfu_E^+ \|_{L^3(\omega_T)} +  \| \emph{curl}~\bfu_E^- \|_{L^3(\omega_T)}  \right).
\end{equation}
\end{thm}
\begin{proof}
Similar to \eqref{thm_interp_err_3_eq1}, we have
\begin{equation}
\label{thm_interp_err_4_eq_extra1}
\| \text{curl}(\widetilde{\Pi}_{h,T} \bfu - \bfu) \|_{L^2(T)} \le \| \text{curl}( \widetilde{\Pi}_{h,T} \bfu - J_{h,T_{\epsilon}}\bfu ) \|_{L^2(T)} + \| \text{curl}(J_{h,T_{\epsilon}}\bfu - \bfu) \|_{L^2(T)}.
\end{equation}
Again the second term in \eqref{thm_interp_err_4_eq_extra1} directly follows from Theorem \ref{thm_interp_err_2}. For the first term, we also consider the piecewise-defined function $\bfw = \bfu_E - J_{h,T_{\epsilon}} \bfu$ with $\bfw^{\pm} = \bfu^{\pm}_E - J^{\pm}_{h,T_{\epsilon}} \bfu$ in $T^{\pm}_h$ partitioned by $\Gamma^T_h$. 
By the identity in \eqref{thm_bound_eq02} we let $\tau=\mu^{-1}\text{curl}~\bfpsi_i$, $i=1,2,3$. Then the similar derivation to \eqref{thm_curl0_eq1}, i.e., the integration by parts on $T^{\pm}_h$, leads to
\begin{equation}
\begin{split}
\label{thm_interp_err_4_eq1}
\frac{1}{\mu} \text{curl}~ \widetilde{\Pi}_{h,T}(\bfu - J_{h,T_{\epsilon}} \bfu)= \int_{\partial T} (\bfu - J_{h,T_{\epsilon}} \bfu)\cdot\bft_i ds ~ \tau= \int_{\partial T} \bfw\cdot\bft_i ds ~ \tau =  \int_{\Gamma^T_h} [\bfw\cdot\bar{\bft}]_{\Gamma^T_h} ds ~\tau + \int_T \text{curl}~\bfw dX ~\tau
\end{split}
\end{equation}
where $\bar{\bft}$ denotes the unit tangential vector to $\Gamma^T_h$ in the clockwise orientation of $T^-_h$. Then applying the integration by parts to the subregion $\widetilde{T}$ and using jump conditions \eqref{weak_eqv_jc_1}, \eqref{inter_jc_1}, we actually have
\begin{equation}
\begin{split}
\label{thm_interp_err_4_eq2}
\int_{\Gamma^T_h} [\bfw\cdot\bar{\bft}]_{\Gamma^T_h} ds =  \int_{\Gamma^T_h} \bfu^-_E\cdot\bar{\bft} -   \bfu^+_E\cdot\bar{\bft} ds = \int_{\widetilde{T}} \text{curl}(\bfu^-_E - \bfu^+_E ) dX.
\end{split}
\end{equation}
By H\"older's inequality and the geometric estimate \eqref{lemma_interface_flat_eq1}, we have
\begin{equation}
\label{thm_interp_err_4_eq3}
\verti{ \int_{\Gamma^T_h} [\bfw\cdot\bar{\bft}]_{\Gamma^T_h} ds }\le \|  \text{curl}(\bfu^-_E - \bfu^+_E ) \|_{L^3(\widetilde{T})} |\widetilde{T}|^{2/3} \le Ch^2_T \left( \| \text{curl}~ \bfu^-_E \|_{L^3(\omega_T)} + \| \text{curl}~ \bfu^+_E \|_{L^3(\omega_T)} \right)
\end{equation}
where we have also used $\widetilde{T}\subset T \subset \omega_T$ in the last inequality. Also, by H\"older's inequality and Theorem \ref{thm_interp_err_2}, we have
\begin{equation}
\label{thm_interp_err_4_eq4}
\verti{ \int_T \text{curl}~\bfw dX } \le \| \text{curl}~\bfw \|_{L^2(T)} |T|^{1/2} \le Ch^2_T \left( \| \bfu_E^+ \|_{\bfH^1(\text{curl};\omega_T)} +  \| \bfu_E^- \|_{\bfH^1(\text{curl};\omega_T)}  \right).
\end{equation}
Moreover, in \eqref{thm_interp_err_4_eq1} we use the inequality in \eqref{thm_bound_eq02} to obtain $\|\tau\|_{L^2(T)}\le Ch^{-2}_T|T|^{1/2}\le Ch^{-1}_T$. Now putting this estimate together with \eqref{thm_interp_err_4_eq3} and \eqref{thm_interp_err_4_eq4} into \eqref{thm_interp_err_4_eq1}, we have 
\begin{equation}
\begin{split}
\label{thm_interp_err_4_eq5}
& \| \mu^{-1} \text{curl}~\widetilde{\Pi}_{h,T} (\bfu - J_{h,T_{\epsilon}} \bfu) \|_{L^2(T)} \\
\le & C h_T \left( \| \bfu_E^+ \|_{\bfH^1(\text{curl};\omega_T)} +  \| \bfu_E^- \|_{\bfH^1(\text{curl};\omega_T)} + \| \text{curl}~\bfu_E^+ \|_{L^3(\omega_T)} +  \| \text{curl}~\bfu_E^- \|_{L^3(\omega_T)}  \right)
\end{split}
\end{equation}
which yields the desired result with \eqref{thm_interp_err_4_eq_extra1}.
\end{proof}

Finally we can provide the estimate for the global interpolation $\widetilde{\Pi}_h$ defined by \eqref{interp_4}.
\begin{thm}
\label{thm_interp_err_5}
Suppose $\bfu\in \bfH^1(\emph{curl};\Omega)$, then there exists a constant $C$ such that
\begin{equation}
\label{thm_interp_err_5_eq0}
\| \widetilde{\Pi}_{h} \bfu - \bfu \|_{\bfH^1(\emph{curl};\Omega)} \le C h \left( \| \bfu \|_{\bfH^1(\emph{curl};\Omega^-)} +  \| \bfu \|_{\bfH^1(\emph{curl};\Omega^+)} \right).
\end{equation}
\end{thm}
\begin{proof}
Combining Theorem \ref{thm_interp_err_3}, Theorem \ref{thm_interp_err_4} and using the standard estimates \eqref{interp_1_err} on non-interface elements, and then using the finite overlapping property of the patches, we have
\begin{equation}
\label{thm_interp_err_5_eq1}
\| \widetilde{\Pi}_{h} \bfu - \bfu \|_{\bfH^1(\text{curl};\Omega)} \le Ch \left( \| \bfu^+_E \|_{\bfH^1(\text{curl};\Omega)} + \| \bfu^-_E \|_{\bfH^1(\text{curl};\Omega)}  + \| \text{curl}~ \bfu^+_E \|_{L^3(\Omega)}  + \| \text{curl}~ \bfu^-_E \|_{L^3(\Omega)} \right).
\end{equation}
By the Sobolev embedding theorem, we have $\| \text{curl}~ \bfu^{\pm}_E \|_{L^3(\Omega)}\le C \| \text{curl}~ \bfu^{\pm}_E \|_{H^1(\Omega)}$ where the constant $C$ only depends $\Omega$. Furthermore, the boundedness of the $\bfH^1(\text{curl};\Omega)$ extensions in Theorem \ref{thm_ext} yields the desired result.
\end{proof}

%


\section{Solution Errors of The IFE Scheme}
\label{sec:solu_error}

In this section, we analyze the PG-IFE scheme \eqref{weak_form_3}. 

\subsection{Local Stability Resutls}
As suggested by the framework \cite{2020AdjeridBabukaGuoLin,2019GuoLin}, the stability of an IFE method generally highly relies on the stability of the linear operator $\mathcal{C}_T$ used to construct IFE functions given by the analysis below. 
Again, the image and inverse image of $\mathcal{C}_T$ are all polynomials and thus can be used on the whole element $T$ instead of only the associated subelements on which they are defined. Without loss of generality, in the following analysis we only consider the interface element configurations in Figure \ref{fig:interf_case} where $T^-_h$ is assumed to be triangular. We first recall a norm equivalence result from Lemma 3.6 in \cite{2019GuoLin}.
\begin{lemma}
\label{lem_norm_poly_equiv}
For each interface element with the configuration shown in Figure \ref{fig:interf_case}, if $|A_1D|\ge\frac{1}{2}|A_1A_2|$ and $|A_1E|\ge\frac{1}{2}|A_1A_3|$ (case 1), there holds
\begin{subequations}
\label{lem_norm_poly_equiv0}
\begin{align}
\| \cdot \|_{L^2(T^-_h)} \simeq \| \cdot \|_{L^2(T)}, ~~~~ \text{on} ~~ \mathcal{ND}_h(T), \label{lem_norm_poly_equiv01}   
\end{align}
and if $|A_1D|\le\frac{1}{2}|A_1A_2|$ or $|A_1E|\le\frac{1}{2}|A_1A_3|$ (case 2)
\begin{align}
\| \cdot \|_{L^2(T^+_h)} \simeq \| \cdot \|_{L^2(T)}, ~~~~ \text{on} ~~ \mathcal{ND}_h(T).  \label{lem_norm_poly_equiv02}   
\end{align}
\end{subequations}
\end{lemma}

\begin{figure}[H]
\centering
\begin{subfigure}{.3\textwidth}
     \includegraphics[width=2in]{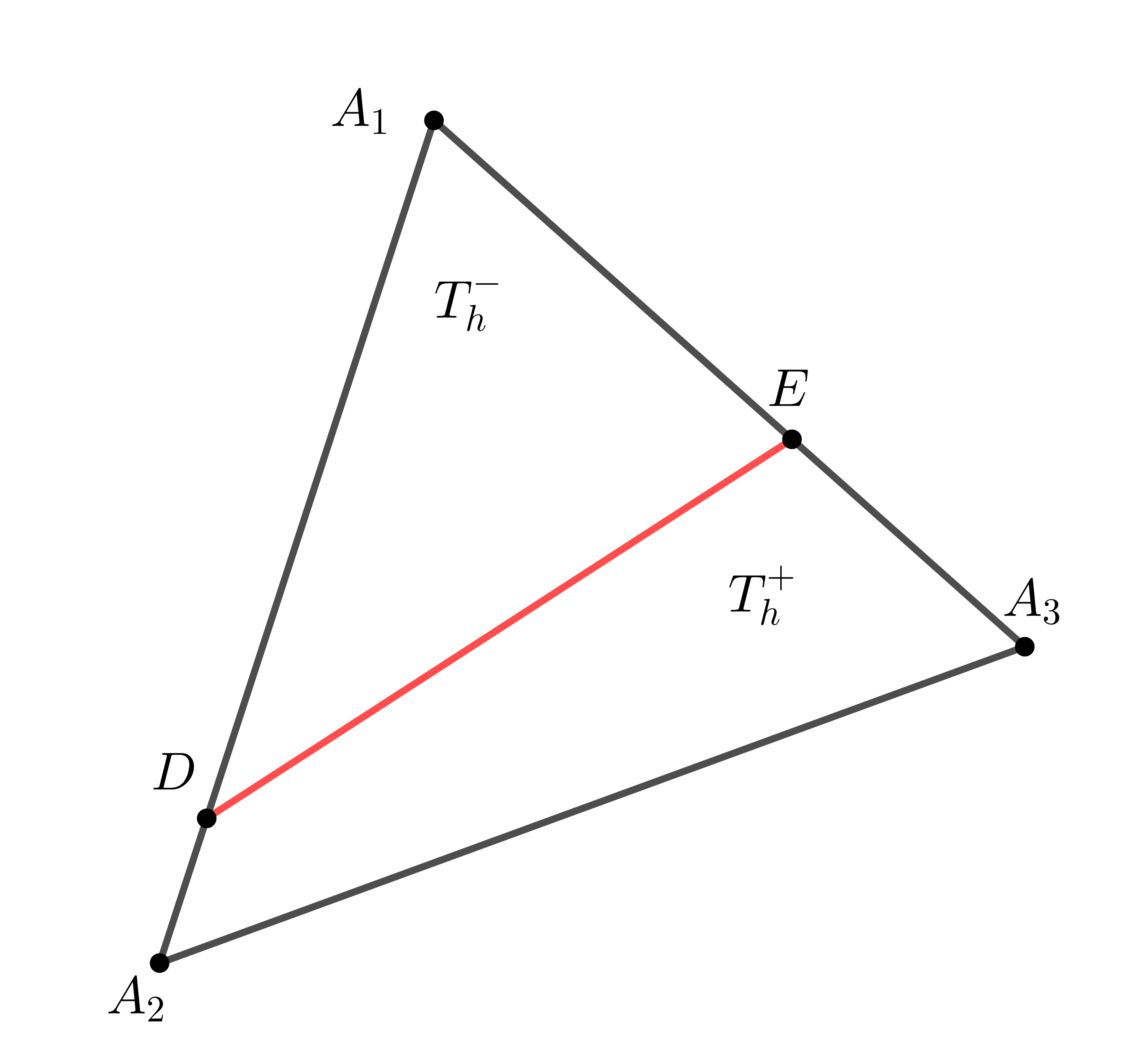}
     \caption{Case 1}
     \label{interf_case_1} 
\end{subfigure}
~~~~~~~~~~~~~~~~
\begin{subfigure}{.3\textwidth}
     \includegraphics[width=2in]{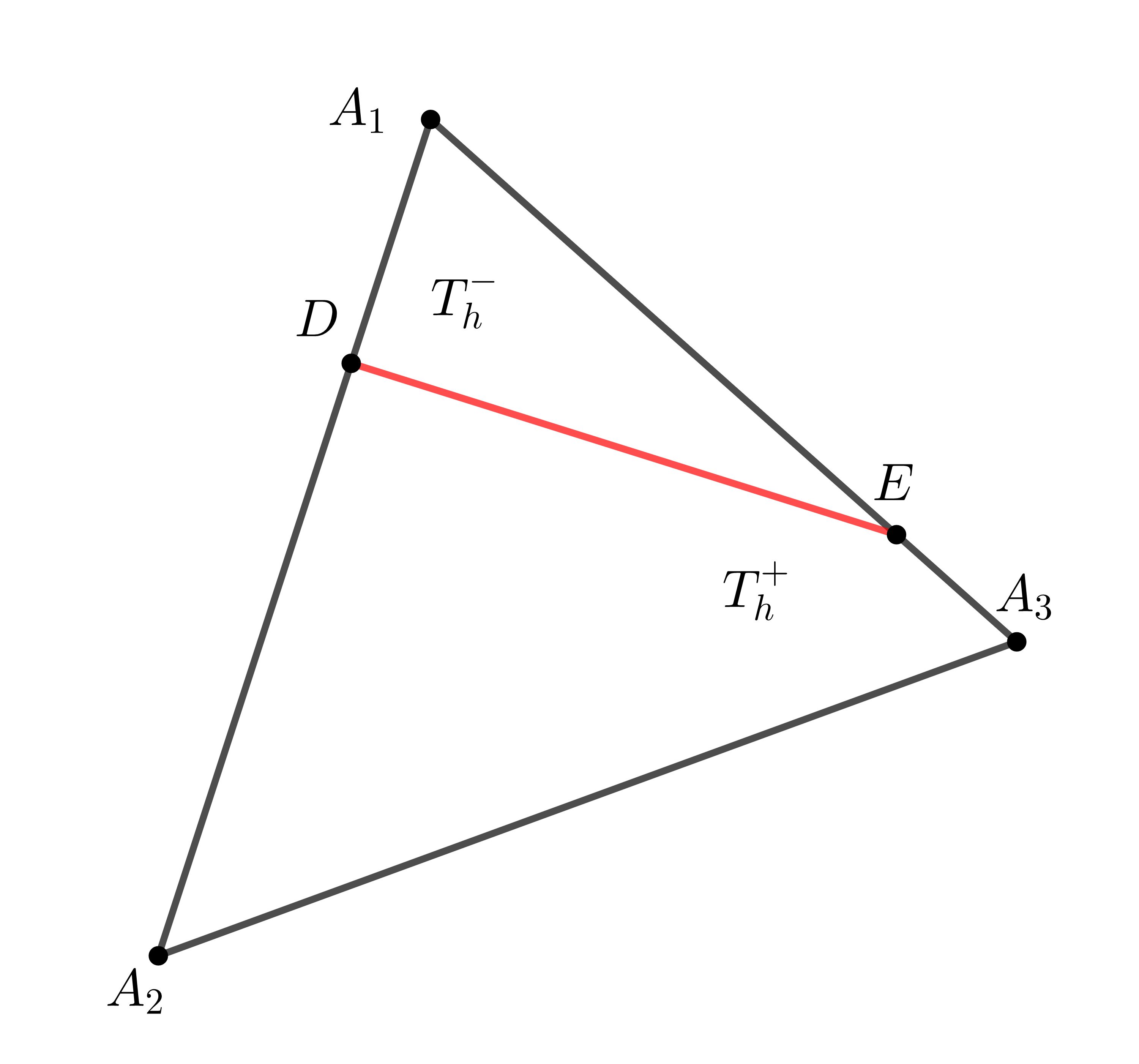}
     \caption{Case 2}
     \label{interf_case_2} 
\end{subfigure}
     \caption{Interface element configuration}
  \label{fig:interf_case} 
\end{figure}

Then we can show the stability of the extension operator $\mathcal{C}_T$ in the following.

\begin{lemma}
\label{lem_C_stab}
For each interface element $T$, if $|A_1D|\ge\frac{1}{2}|A_1A_2|$ and $|A_1E|\ge\frac{1}{2}|A_1A_3|$ (Case 1 in Figure \ref{fig:interf_case}),
\begin{subequations}
\label{lem_C_stab_eq0}
\begin{align}
 \| \mathcal{C}_T(\bfv_h) \|_{L^2(T^-_h)} \le C \| \bfv_h \|_{L^2(T^-_h)}, ~~~ \text{and} ~~~  \| \mathcal{C}^{-1}_T(\bfv_h) \|_{L^2(T^-_h)} \le C \| \bfv_h \|_{L^2(T^-_h)}, ~~~ \forall \bfv_h\in\mathcal{ND}_h(T);\label{lem_C_stab_eq01}   
\end{align}
if $|A_1D|\le\frac{1}{2}|A_1A_2|$ or $|A_1E|\le\frac{1}{2}|A_1A_3|$ (Case 2 in Figure \ref{fig:interf_case}),
\begin{align}
 \| \mathcal{C}_T(\bfv_h) \|_{L^2(T^+_h)} \le C \| \bfv_h \|_{L^2(T^+_h)} ~~~ \text{and} ~~~ \| \mathcal{C}^{-1}_T(\bfv_h) \|_{L^2(T^+_h)} \le C \| \bfv_h \|_{L^2(T^+_h)}, ~~~ \forall \bfv_h\in\mathcal{ND}_h(T)  . \label{lem_C_stab_eq02}   
\end{align}
\end{subequations}
\end{lemma}
\begin{proof} 
For simplicity, we only prove \eqref{lem_C_stab_eq01} here. Let $D=[x_{D,1},x_{D,2}]^t$ and $E=[x_{E,1},x_{E,2}]^t$. By the explicit formula in \eqref{explicit_form_CT} and its derivation, we can write
\begin{equation}
\label{lem_C_stab_eq1}
\mathcal{C}_T(\bfv_h) = \bfv_h + b_1 [x_2 - x_{D,2}, -(x_1 - x_{D,1})]^t + b_2 \bar{\bfn} , ~~~~~ \forall \bfv_h\in \mathcal{ND}_h(T)
\end{equation}
where $\bar{\bfn}$ and $[x_2 - x_{D,2}, -(x_1 - x_{D,1})]^t$ with $[x_1,x_2]^t\in\Gamma^T_h$ are norm vectors to $\Gamma^T_h$ and assumed to have the same orientation, and
\begin{equation}
\label{lem_C_stab_eq3}
b_1 = \frac{1}{2}\left( 1- \frac{\mu^+}{\mu^-} \right) \text{curl}~\bfv_h ~~~ \text{or} ~~~
b_2 = \left( \frac{\beta^-}{\beta^+} - 1 \right) \bfv_h(X_m)\cdot\bar{\bfn}  - \frac{b_1|\Gamma^T_h|}{2}, 
\end{equation}
\begin{equation}
\label{lem_C_stab_eq2}
 b_1 = \frac{1}{2}\left( \frac{\mu^-}{\mu^+} - 1 \right) \text{curl}~\mathcal{C}_T(\bfv_h),~~~ \text{or} ~~~ b_2 = \left( 1 - \frac{\beta^+}{\beta^-}  \right) \mathcal{C}_T(\bfv_h)(X_m) - \frac{b_1|\Gamma^T_h|}{2}.
\end{equation}
Note that $\text{curl}~\bfv$ is a constant. Then if $|A_1D|\ge\frac{1}{2}|A_1A_2|$ and $|A_1E|\ge\frac{1}{2}|A_1A_3|$, the first identity in \eqref{lem_C_stab_eq3} leads to
\begin{equation}
\label{lem_C_stab_eq4}
|b_1|\le C | \text{curl}~\bfv_h | \le Ch^{-1}_T \| \text{curl}~\bfv_h \|_{L^2(T)} \le Ch^{-2}_T \| \bfv_h \|_{L^2(T)} \le Ch^{-2}_T \| \bfv_h \|_{L^2(T^-_h)},
\end{equation}
where in the last two inequalities above we have used the inverse inequality and norm equivalence \eqref{lem_norm_poly_equiv01}. In addition, we note that there exists a constant $C$ such that $|\triangle A_1DE|/|DE|\ge Ch$ where $C$ is independent of interface location. So by the trace inequalities for polynomials \cite{2003WarburtonHesthaven} we have $|\bfv_h(X_m)\cdot\bar{\bfn}|\le Ch^{-1}_T\|\bfv_h \|_{L^2(\triangle A_1DE)} = Ch^{-1}_T \| \bfv_h \|_{L^2(T^-_h)}$. Putting this estimate and \eqref{lem_C_stab_eq4} into the second identity in \eqref{lem_C_stab_eq3}, we obtain
\begin{equation}
\label{lem_C_stab_eq5}
|b_2| \le Ch^{-1}_T \| \bfv_h \|_{L^2(T^-_h)} + Ch_T |b_1| \le Ch^{-1}_T \| \bfv_h \|_{L^2(T^-_h)} .
\end{equation}
Now using \eqref{lem_C_stab_eq1}, we have
\begin{equation}
\label{lem_C_stab_eq6}
 \| \mathcal{C}_T(\bfv_h) \|_{L^2(T^-_h)} \le \| \bfv_h \|_{L^2(T^-_h)} + Ch^2_T |b_1| + Ch_T |b_2| \le C \| \bfv_h \|_{L^2(T^-_h)}
\end{equation}
which has finished the proof of the first one in \eqref{lem_C_stab_eq01}. The argument for the second one in \eqref{lem_C_stab_eq01} is similar in which we need to use the second identities of \eqref{lem_C_stab_eq2}.
\end{proof}

The stability enables us to prove the trace inequality of IFE functions.

\begin{thm}
\label{lem_trace_inequa}
For each interface element $T$ and its edge $e$, there holds
\begin{equation}
\label{lem_trace_inequa_eq0}
\| \bfz_h \|_{L^2(e)} \le Ch^{-1/2}_T \| \bfz_h \|_{L^2(T)}, ~~~~~ \forall \bfz_h\in \mathcal{IND}_h(T).
\end{equation}
\end{thm}
\begin{proof}
Without loss of generality, we only consider the case that $|A_1D|\ge\frac{1}{2}|A_1A_2|$ and $|A_1E|\ge\frac{1}{2}|A_1A_3|$ as shown by Case 1 in Figure \ref{fig:interf_case}. On the interface edge $A_1A_2$, the standard trace inequality \cite{2003WarburtonHesthaven} yields
\begin{equation}
\label{lem_trace_inequa_eq1}
\| \bfz_h \|_{L^2(A_1D)} \le Ch^{-1/2}_T \| \bfz_h \|_{L^2(\triangle A_1DE)} ~~~ \text{and} ~~~ \| \bfz_h \|_{L^2(DA_2)} \le Ch^{-1/2}_T \| \bfz_h \|_{L^2(\triangle DA_2E)},
\end{equation}
where the constant $C$ is independent of interface location since the distance from the point $E$ to $A_1A_2$ is lower bounded regardless of interface location. Then the result on this edge follows. The argument for the interface edge $A_1A_3$ is similar. For the non-interface edge $A_2A_3$, by the standard trace inequality and \eqref{lem_C_stab_eq01} we obtain
\begin{equation}
\begin{split}
\label{lem_trace_inequa_eq2}
\| \bfz_h^+ \|_{L^2(A_2A_3)} & \le Ch^{-1/2}_T \| \bfz^+_h \|_{L^2(T)} = Ch^{-1/2}_T ( \| \mathcal{C}^{-1}_T(\bfz^-_h) \|_{L^2(T^-_h)} + \| \bfz^+_h \|_{L^2(T^+_h)} ) \\
& \le Ch^{-1/2}_T ( \| \bfz^-_h \|_{L^2(T^-_h)} + \| \bfz^+_h \|_{L^2(T^+_h)} ) 
\end{split}
\end{equation}
which yields the desired result. The argument for the Case 2 in Figure \ref{fig:interf_case} is similar and relies on \eqref{lem_C_stab_eq02}.
\end{proof}

Now based on the trace inequality, we show the main result in this subsection which is the stability estimate for the isomorphism $\mathbb{\Pi}_h$.

\begin{thm}
\label{thm_iso_stab}
These exist constants $c$ and $C$ such that
\begin{subequations}
\label{thm_iso_stab_eq0}
\begin{align}
    & c \| \bfz_h \|_{L^2(T)} \le  \| \mathbb{\Pi}_{h,T} \bfz_h \|_{L^2(T)} \le C \| \bfz_h \|_{L^2(T)} ~~~ \forall \bfz_h\in\mathcal{IND}_h(T), \label{thm_iso_stab_eq01} \\ 
    & c \| \emph{curl}~\bfz_h \|_{L^2(T)} \le \| \emph{curl}~\mathbb{\Pi}_{h,T} \bfz_h \|_{L^2(T)} \le C \| \emph{curl}~\bfz_h \|_{L^2(T)} ~~~ \forall \bfz_h\in\mathcal{IND}_h(T).  \label{thm_iso_stab_eq02} 
\end{align}
\end{subequations}
\end{thm}
\begin{proof}
These inequalities for non-interface elements are trivial, and we only consider the interface elements here. For \eqref{thm_iso_stab_eq01}, using the argument similar to \eqref{thm_interp_err_3_eq3} with \eqref{thm_bound_eq01} we obtain
\begin{equation}
\label{thm_iso_stab_eq1}
\vertii{\mathbb{\Pi}_{h,T}\bfz_h}_{L^2(T)} = \vertii{\sum_{i=1}^3 \int_{e_i} \bfz_h\cdot\bft_i ds \bfpsi_i}_{L^2(T)}
\le \sum_{i=1}^3 C h^{1/2}_T \| \bfz_h \|_{L^2(e_i)}.
\end{equation}
Applying the trace inequality in Theorem \ref{lem_trace_inequa} to \eqref{thm_iso_stab_eq1} yields the right inequality of \eqref{thm_iso_stab_eq01}. The left one can be obtained by applying this argument to $\mathbb{\Pi}^{-1}_{h,T}$. In addition \eqref{thm_iso_stab_eq02} is a direct consequence of the identity in \eqref{thm_curl0_eq1}.
\end{proof}

\begin{figure}[H]
\centering
\begin{subfigure}{.3\textwidth}
     \includegraphics[width=1.9in]{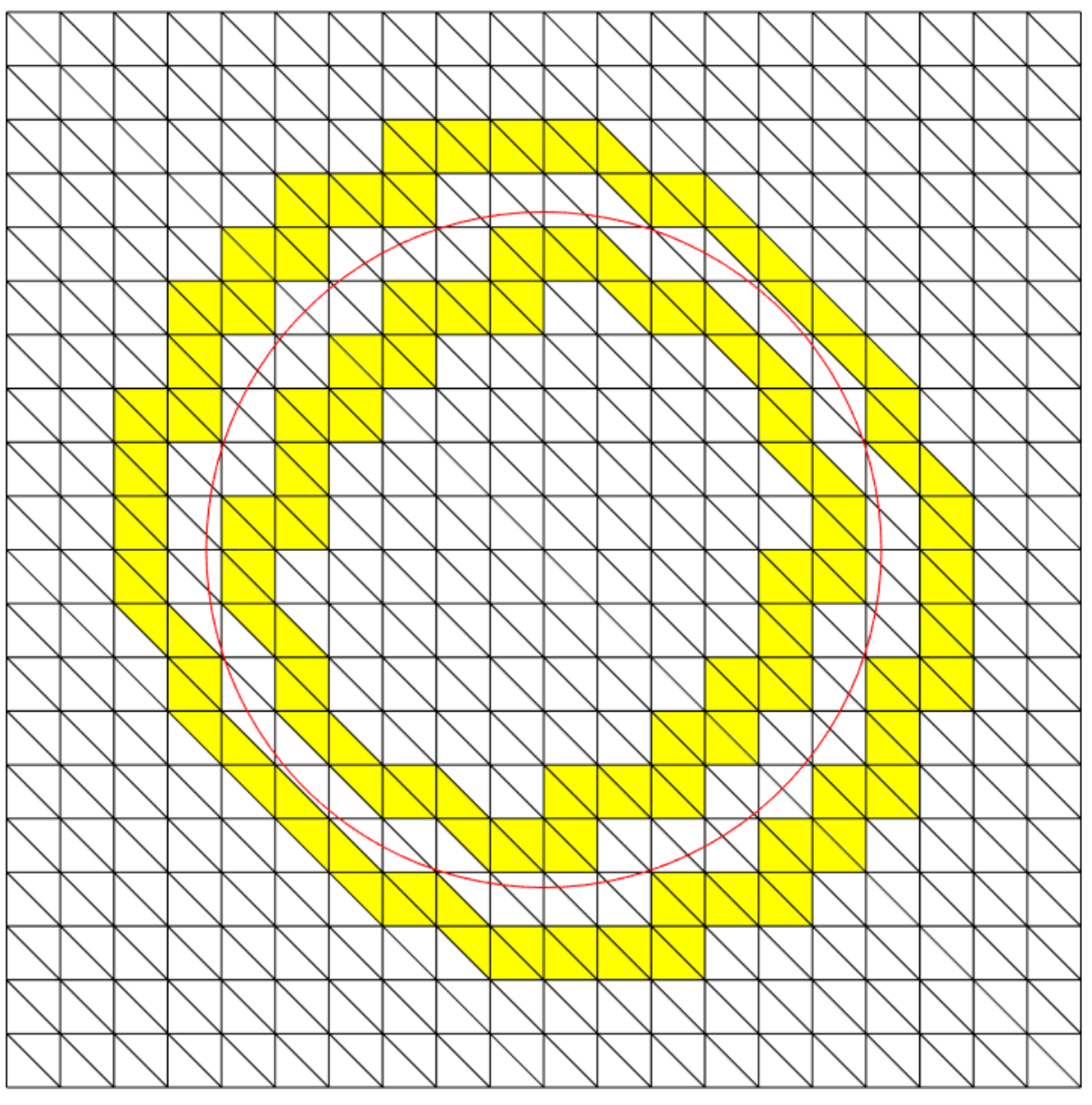}
     \label{fig:domain_out} 
\end{subfigure}
~~
\begin{subfigure}{.3\textwidth}
     \includegraphics[width=2in]{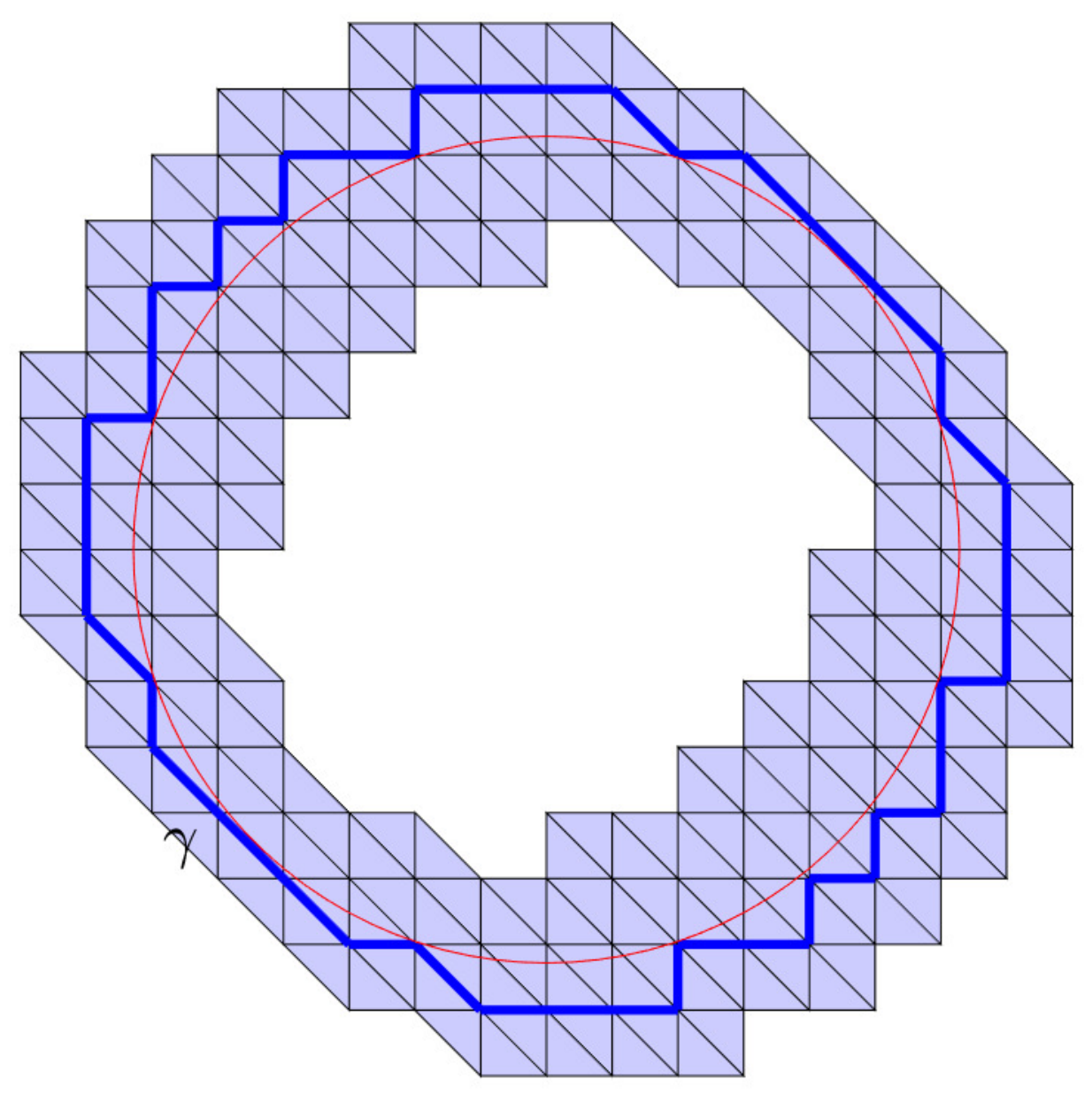}
     \label{fig:domain_in} 
\end{subfigure}
~~
\begin{subfigure}{.3\textwidth}
     \includegraphics[width=2in]{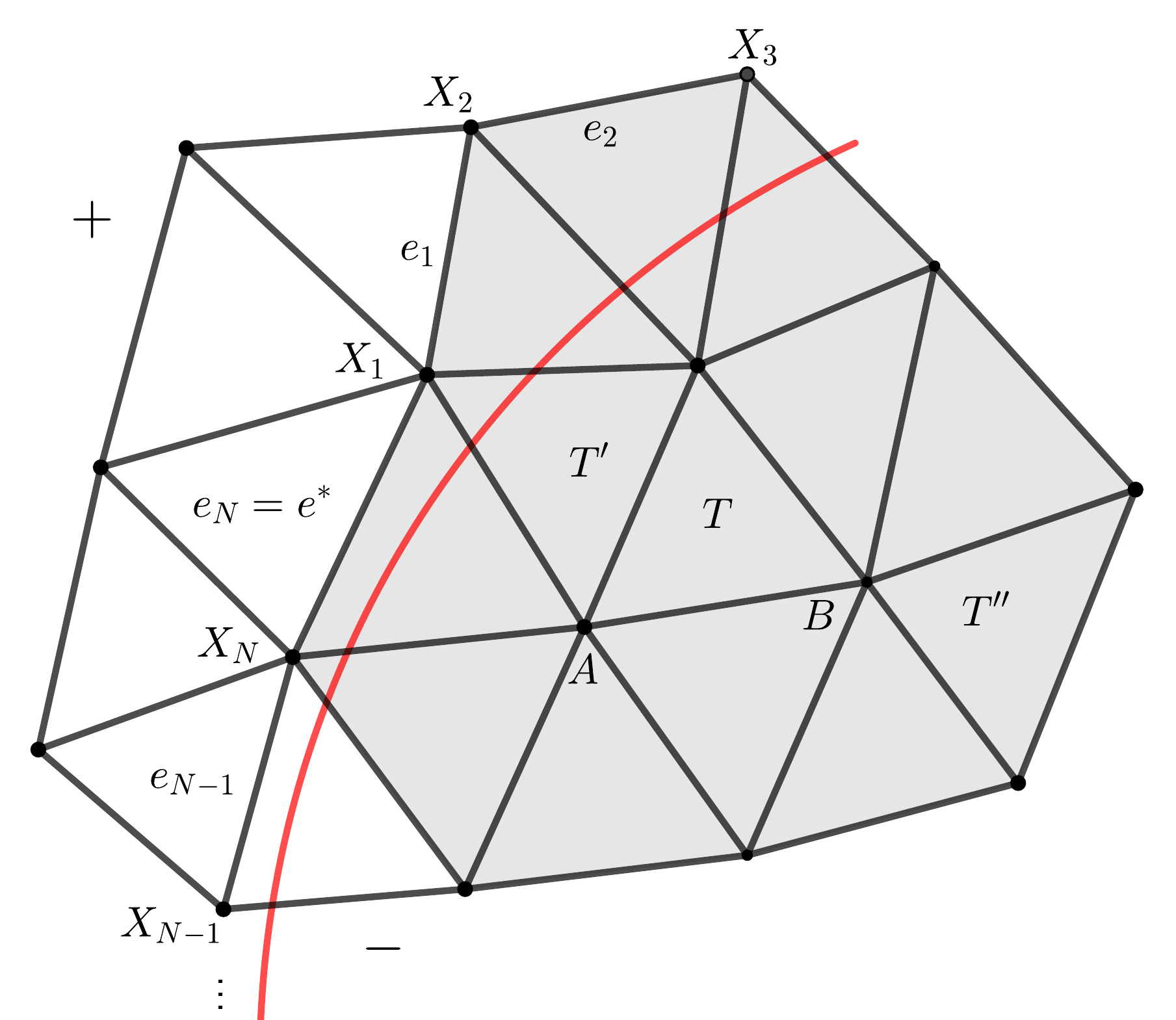}
     \label{fig:domain_star} 
\end{subfigure}
     \caption{The left plot: $\widetilde{\Omega}^{\pm}_h$, the meddle plot: $\Omega^{\Gamma}_h$ and $\gamma$, and the right plot: a patch near the interface.}
  \label{fig:submesh} 
\end{figure}

\subsection{The \textit{inf-sup} stability}
In the analysis of a general Petrov-Galerkin system, 
the most important and often also difficult step is to establish a so-called \textit{inf-sup} stability \cite{1972BabuskaAziz}.
This refers to the following inequality between two spaces $\mathcal{ND}_{h,0}(\Omega)$ and $\mathcal{IND}_{h,0}(\Omega)$
for our Petrov-Galerkin IFE system \eqref{weak_form_3}:
\begin{equation}
\label{inf_sup}
\sup_{\bfv_h\in\mathcal{ND}_{h,0}(\Omega)} \frac{a(\bfu_h,\bfv_h)}{\| \bfv_h \|_{\bfH(\text{curl};\Omega)}} \ge C_s \| \bfu_h \|_{\bfH(\text{curl};\Omega)} ~~~~ \forall \bfu_h \in \mathcal{IND}_{h,0}(\Omega)
\end{equation}
where $C_s>0$ should be uniformly bounded away from $0$ regardless of the mesh size $h$ and interface location. We mention that the rigorous analysis on the \textit{inf-sup} stability of PG-IFE methods still remains open even for $H^1$-elliptic interface problems \cite{2013HouSongWangZhao}. For the considered $\bfH(\text{curl})$-elliptic interface problems here, 
it is the first time to establish, in the rest of this subsection, the \textit{inf-sup} stability \eqref{inf_sup} 
(see Theorem\,\ref{thm_infsup})
for the general discontinuous magnetic permeability, but for the discontinuous conductivity 
whose jumps is less than a critical constant. The analysis is lengthy and we shall decouple it into several lemmas and steps. The basic idea is to apply a special discrete decomposition to functions $\bfu_h\in \mathcal{IND}_h(\Omega)$ such that their regular components are sufficiently small near the interface. 

For this purpose, we define the subdomain consisting of interface elements $\Omega^i_h=\cup\{ T\in \mathcal{T}^i_h \}$ where its boundary in $\Omega^{\pm}$ is denoted by $\partial \Omega^{i,\pm}_h$, and specifically, we let $\gamma =\partial\Omega^{i,+}_h$ as highlighted by the blue polyline in the meddle plot of Figure \ref{fig:submesh}. Furthermore we define the submesh and the corresponding subdomain near the interface elements:
\begin{equation}
\begin{split}
\label{mesh_tran}
\widetilde{\mathcal{T}}^{\pm}_h := \{ T\in\mathcal{T}_h\backslash \mathcal{T}^i_h~:~ T\subseteq \Omega^{\pm}, ~  \exists T'\in \mathcal{T}^i_h ~ \text{such that}~T~\text{and}~T' ~ \text{share at least one node} \}, ~~ \widetilde{\Omega}^{\pm}_h := \cup\{ T\in \widetilde{\mathcal{T}}^{\pm}_h \},
\end{split}
\end{equation}
which are illustrated by the yellow-shaded region in the left plot of Figure \ref{fig:submesh}. Moreover, we define the union of these elements $\mathcal{T}^+_h\cup\mathcal{T}^i_h\cup\mathcal{T}^-_h=:\mathcal{T}^{\Gamma}_h$ and the corresponding subdomain $\widetilde{\Omega}^+_h\cup\Omega^i_h\cup\widetilde{\Omega}^-_h=:\Omega^{\Gamma}_h$ which is indicated by the blue-shaded region in the meddle plot of Figure \ref{fig:submesh}. Basically, $\Omega^{\Gamma}_h$ contains the interface elements and the non-interface elements near the interface elements.

The first theorem is to handle the curl term, and it suggests that the stiffness matrix of the PG-IFE method is actually 
as the same as the standard IFE method which uses the IFE functions as the test functions.
\begin{thm}
\label{thm_stiff}
On each element $T$, there holds 
\begin{equation}
\label{thm_stiff_eq0}
\int_T \mu^{-1} \emph{curl}~\bfz_h~\emph{curl}~\mathbb{\Pi}_{h,T}\bfz_h dX = \int_T\mu^{-1}  \emph{curl}~\bfz_h~\emph{curl}~\bfz_h dX, ~~~~~~ \forall \bfz_h\in\mathcal{IND}_h(T).
\end{equation}
\end{thm}
\begin{proof}
It is trivial on non-interface elements. On an interface element, we note that $\mu^{-1}\text{curl}~\bfz_h$ is a constant, and then the integration by part yields
\begin{equation}
\begin{split}
\label{thm_stiff_eq1}
 \int_{T} \mu^{-1} \text{curl}~\bfz_h~\text{curl}~\mathbb{\Pi}_{h,T}\bfz_h dX = & \mu^{-1}\text{curl}~\bfz_h \int_{\partial T} \mathbb{\Pi}_{h,T}\bfz_h\cdot\bft ~  ds\\
 = &  \mu^{-1}\text{curl}~ \bfz_h \int_{\partial T}  \bfz_h\cdot\bft ~  ds  = \int_T \mu^{-1}  \text{curl}~\bfz_h~\text{curl}~\bfz_h dX
\end{split}
\end{equation}
where in the second identity we have also used the continuous tangential jump condition \eqref{weak_eqv_jc_1}.
\end{proof}

The next result is for a special discrete decomposition.
\begin{lemma}[A Special Discrete Decomposition]
\label{lemma_special_decomp}
For each
$\bfu_h\in \mathcal{ND}_{h,0}(\Omega)$ (or $\mathcal{IND}_{h,0}(\Omega)$), there exists a $\bfu^{\ast}_h\in\mathcal{ND}_{h,0}(\Omega)$ (or $\mathcal{IND}_{h,0}(\Omega)$) and
$\mathring{\bfu}_h\in\mathcal{ND}_{h,0}(\Omega)\cap\emph{Ker(curl)}$ (or $\mathcal{IND}_{h,0}(\Omega)\cap\emph{Ker(curl)}$) such that 
\begin{equation}
\label{lemma_special_decomp_eq0}
\bfu_h = \bfu^{\ast}_h + \mathring{\bfu}_h
\end{equation}
satisfying that
\begin{equation}
\label{lemma_special_decomp_eq01}
\| \bfu^{\ast}_h \|_{L^2(\Omega^{\Gamma}_h)} \le Ch \| \emph{curl}~ \bfu_h \|_{L^2(\Omega)}.
\end{equation}
\end{lemma}
\begin{proof}
The proof is lengthy, and we decompose it into several steps.
\vspace{0.1in}

\textit{Step 1}. We first focus on $\bfu_h\in\mathcal{ND}_{h,0}(\Omega)$. We need to construct a function $\bfv_h\in\mathcal{ND}_{h,0}(\Omega)\cap\text{Ker(curl)}$ such that its trace on $\gamma$ matches $\bfu_h$ except one edge denoted by $e^{\ast}$ of $\gamma$. Let's denote the edges on $\gamma$ by $e_1$, $e_2$, ..., $e_N$ having the clockwise orientation with the nodes $X_1$, $X_2$,..., $X_N$, and, without loss of generality, we assume $e^{\ast}=e_N$, as shown in the right plot of Figure \ref{fig:submesh}. We consider a linear finite element function $s_h \in S_{h,0}(\widetilde{\Omega}^+_h\cup\Omega^i_h)$ such that $v_h(X_1)=0$, $v_h(X_{n+1})= \sum_{j=1}^{n} \int_{e_j} \bfu_h\cdot\bft ds$, $n=1,2,...,N-1$. Then, by the exact sequence, we let $\bfv_h=\nabla s_h\in \mathcal{ND}_{h,0}(\widetilde{\Omega}^+_h\cup\Omega^i_h)$, and by the definition we clearly have $\int_{e_n}\bfv_h\cdot\bft ds = \int_{e_n}\bfu_h\cdot\bft ds$, $n=1,2,...,N-1$. Note that $\int_{\gamma}\bfv_h\cdot\bft ds =0$ but $\bfu_h$ may not have this property, so $\int_{e^{\ast}} \bfv_h\cdot\bft ds$ may not equal $\int_{e^{\ast}} \bfu_h\cdot\bft ds$. 

\vspace{0.1in}
\textit{Step 2}. Let us assume the closed polyline $\gamma$ partition the whole domain $\Omega$ into $\Omega_{h,+}$ and $\Omega_{h,-}$ where, by the definition of $\gamma$, we have $\Omega_{h,+}\subseteq\Omega^+$ and $\Omega^-\subseteq\Omega_{h,-}$. Since the triangulation is regular, we note that $\gamma$ is a Lipschitz curve, and thus both $\Omega_{h,\pm}$ have Lipschitz boundary. Denote $\bfk_h = \bfu_h - \bfv_h$ and $\bfk^{\pm}_h = (\bfu_h - \bfv_h)|_{\Omega_{h,\pm}}$. By the discussion in Step 1, we note that $\bfk^{+}_h$ has the zero trace on $\partial\Omega$ and $\gamma\backslash e^{\ast}$, and $\bfk^{-}_h$ simply has the zero trace on $\gamma\backslash e^{\ast}$. Therefore we can apply the discrete regular decomposition of Theorem 11 in
\cite{2019HiptmairPechstein} (also see the related discussion in \cite{2002Hiptmair,2017HiptmairPechstein}) to $\bfk^{\pm}_h$ on $\Omega_{h,\pm}$ which gives
$$
\bfk^{\pm}_h = \bfR^{\ast}_{\pm} \bfz^{\pm}_h  + \bfr^{\pm}_h + \bfh^{\pm}_h, ~~~~ \text{on} ~ \Omega_{h,\pm},
$$
where the regular components $\bfz^{\pm}_h\in \left[S_{h}(\Omega_{h,\pm})\right]^2$, the curl-free components $\bfh^{\pm}_h\in \mathcal{ND}_{h}(\Omega_{h,\pm})\cap\text{Ker(curl)}$, the reminders $\bfr^{\pm}_h\in\mathcal{ND}_{h}(\Omega_{h,\pm})$, and $\bfR^{\ast}_{\pm}:S_{h}(\Omega_{h,\pm})\rightarrow \mathcal{ND}_{h}(\Omega_{h,\pm})$ are special local projection operators preserving zero boundary conditions. Moreover these components inherit the zero traces of $\bfk^{\pm}_h$, namely, $\bfz^{+}_h$, $\bfr^{+}_h$ and $\bfh^{+}_h$ also have the corresponding zero traces on $\partial\Omega\cup(\gamma\backslash e^{\ast})$ and the $-$ components have the zero traces on $\gamma\backslash e^{\ast}$. Furthermore, since $e^{\ast}$ is a single edge, the continuity suggests that $\bfz^{\pm}_h$ and $\bfh^{\pm}_h$ must have the zero traces on the whole $\gamma$. Then we have $\bfr^+_h\cdot \bft=\bfk_h\cdot\bft = \bfr^-_h\cdot \bft$ on every edge of $\gamma$. That is, all these three decomposed components must be tangentially continuous on $\gamma$. Therefore, we put these components together to define piecewise-defined function $\bfz_h = \bfz^{\pm}_h$ in $\Omega_{h,\pm}$ belonging to $\left[S_{h,0}(\Omega)\right]^2$, $\bfh_h = \bfh^{\pm}_h$ in $\Omega_{h,\pm}$ belonging to $\mathcal{ND}_{h,0}(\Omega)\cap\text{Ker(curl)}$ and $\bfr_h = \bfr^{\pm}_h$ in $\Omega_{h,\pm}$ belonging to $\mathcal{ND}_{h,0}(\Omega)$ which satisfy
\begin{equation}
\label{lemma_special_decomp_eq1}
\bfu_h - \bfv_h = \bfk_h =  \bfR^{\ast}\bfz_h  + \bfr_h + \bfh_h
\end{equation}
with $\bfR^{\ast} = \bfR^{\ast}_{\pm}$ in $\Omega_{h,\pm}$. By this definition, using the estimates of the components $\bfz^{\pm}_h$, $\bfh^{\pm}_h$ and $\bfr^{\pm}_h$ in Theorem 11 in \cite{2019HiptmairPechstein}, and noting that $\bfv_h$ is curl-free, we have
\begin{subequations}
\label{lemma_special_decomp_eq2}
\begin{align}
 \| \nabla \bfz_h \|_{L^2(\Omega)} \le C \| \text{curl} ~ \bfu_h \|_{L^2(\Omega)}, ~~\| \bfr_h \|_{L^2(\Omega)} \le Ch \| \text{curl}~ \bfu_h \|_{L^2(\Omega)} ~~~ \text{and} ~~~  \| \bfh_h \|_{L^2(\Omega)} \le C \| \text{curl}~ \bfu_h \|_{L^2(\Omega)}, \label{lemma_special_decomp_eq23} 
\end{align}
and for each patch $\omega_T$ of an element $T$, there holds
\begin{align}
\label{lemma_special_decomp_eq21} 
&\| \bfR^{\ast} \bfz_h \|_{L^2(T)} \le C\| \bfz_h \|_{L^2(\omega_T)} + Ch_T \| \text{curl}~ \bfz_h \|_{L^2(\omega_T)}.
\end{align}
\end{subequations}

\textit{Step 3}. We estimate $\mathbf{ R}^{\ast}\bfz_h$ on $\Omega^{\Gamma}_h$. 
Applying the estimate in \eqref{lemma_special_decomp_eq21}, we obtain
\begin{equation}
\begin{split}
\label{lemma_special_decomp_eq3}
\| \bfR^{\ast} \bfz_h \|_{L^2(\Omega^{\Gamma}_h)} &\le C \sum_{T\in\mathcal{T}^{\Gamma}_h} \| \bfR^{\ast} \bfz_h \|_{L^2(T)}
\le C \sum_{T\in\mathcal{T}^{\Gamma}_h} \left( \| \bfz_h \|_{L^2(\omega_T)} + h_T \| \text{curl}~\bfz_h \|_{L^2(\omega_T)} \right).
\end{split}
\end{equation} 
For each patch $\omega_T$, $T\in\mathcal{T}^{\Gamma}_h$, without loss of generality we assume $T\in\widetilde{\mathcal{T}}^-_h$, i.e., it is not an interface element, and then there is at least one interface element denoted by $T'$ such that $T'$ and $T$ share at least one node denoted by $A$ as shown in the right plot of Figure \ref{fig:submesh}. Since $T'$ has one node on $\gamma$, $\bfz_h$ must vanish at this node, and thus the estimate of $\bfz_h$ on $T'$ is straightforward through the Poincar\'e-type inequality:
\begin{equation}
\label{lemma_special_decomp_eq4}
\| \bfz_h \|_{L^2(T')} \le C h_{T'} \| \nabla \bfz_h \|_{L^2(T')}.
\end{equation}
Then we also have the estimate for $|\bfz_h(A)|$ through the trace inequality and \eqref{lemma_special_decomp_eq4}
\begin{equation}
\label{lemma_special_decomp_eq5}
| \bfz_h(A) | \le C h^{-1}_{T'}\| \bfz_h \|_{L^2(T')} \le C  \| \nabla \bfz_h \|_{L^2(T')}.
\end{equation}
In addition, on $T$ we can write $\bfz_h$ as $\bfz_h(X)=\bfz_h(A) + \nabla\bfz_h(X-A)$ where $\nabla\bfz_h$ is a $2$-by-$2$ constant matrix. Thus \eqref{lemma_special_decomp_eq5} together with the continuity at $A$ leads to
\begin{equation}
\label{lemma_special_decomp_eq6}
\| \bfz_h \|_{L^2(T)} \le Ch_T (|\bfz_h(A)| + \| \nabla \bfz_h \|_{L^2(T)} ) \le Ch_T  \| \nabla \bfz_h \|_{L^2(T'\cup T)} \le C h \| \nabla\bfz_h \|_{L^2(\omega_T)} .
\end{equation}
Furthermore we note that any element $T''$ in $\omega_T$ must share at least one node with $T$ denoted by $B$ as shown in the right plot of Figure \ref{fig:submesh} for an example. Then the similar argument to \eqref{lemma_special_decomp_eq5} and \eqref{lemma_special_decomp_eq6} shows
\begin{equation}
\label{lemma_special_decomp_eq7}
\| \bfz_h \|_{L^2(T'')} \le C h_{T''}( |\bfz_h(B)| + \| \nabla\bfz_h \|_{L^2(T'')}) \le  C h_{T''} \| \nabla\bfz_h \|_{L^2(T\cup T' \cup T'')} \le C h \| \nabla\bfz_h \|_{L^2(\omega_T)}. 
\end{equation}
The results above give the estimate of $\bfz_h$ on $\omega_T$. Applying it to \eqref{lemma_special_decomp_eq7} and using the finite overlapping property together with the first estimate in \eqref{lemma_special_decomp_eq23}, we obtain
\begin{equation}
\begin{split}
\label{lemma_special_decomp_eq8}
\| \bfR^{\ast} \bfz_h \|_{L^2(\Omega^{\Gamma}_h)} \le C \sum_{T\in\mathcal{T}^{\Gamma}_h} h\left( \| \nabla \bfz_h \|_{L^2(\omega_T)} + \| \text{curl}~\bfz_h \|_{L^2(\omega_T)} \right) \le C h \| \text{curl}~\bfz_h \|_{L^2(\Omega)}.
\end{split}
\end{equation} 
Therefore, by \eqref{lemma_special_decomp_eq8} and the second estimate in \eqref{lemma_special_decomp_eq23}, setting $\bfu^{\ast}_h = \mathbf{ R}^{\ast}_h\bfz_h + \bfr_h$ and $\mathring{\bfu}_h=\bfv_h + \bfh_h$ fullfils the decomposition \eqref{lemma_special_decomp_eq0} and \eqref{lemma_special_decomp_eq01} for the case $\bfu_h\in\mathcal{ND}_{h,0}(\Omega)$.

\vspace{0.1in}

\textit{Step 4.} Finally if $\bfu_h\in\mathcal{IND}_{h,0}(\Omega)$, we have $\mathbb{\Pi}_h\bfu_h\in\mathcal{ND}_{h,0}(\Omega)$, and then the previous analysis shows the existence of functions $\bfw^{\ast}_h \in\mathcal{ND}_{h,0}(\Omega)$ and $\mathring{\bfw}_h\in\mathcal{ND}_{h,0}(\Omega)\cap\text{Ker(curl)}$ such that $\mathbb{\Pi}_h\bfu_h= \bfw^{\ast}_h + \mathring{\bfw}_h$. So we obtain $\bfu_h = \bfu^{\ast}_h + \mathring{\bfu}_h$ with $\bfu^{\ast}_h = \mathbb{\Pi}^{-1}_h\bfw^{\ast}_h$ and $\mathring{\bfu}_h = \mathbb{\Pi}^{-1}_h\mathring{\bfw}_h$. Here $\bfu^{\ast}_h$ satisfies \eqref{lemma_special_decomp_eq01} due to Theorem \ref{thm_iso_stab} and $\mathring{\bfu}_h\in\mathcal{IND}_{h,0}(\Omega)$ due to the curl-free-preserving property of the isomorphism in Theorem \ref{thm_curl0}.
\end{proof}


The next step is to establish an \textit{inf-sup} stability specially for curl-free subspaces near the interface elements. 
\begin{lemma}
\label{lem_loc_infsup}
For the Cartesian mesh, assume the contrast of the conductivity satisfies $\max\{\beta^+/\beta^-,\beta^-/\beta^+\}< 10.65$, 
then there holds
\begin{equation}
\label{lem_loc_infsup_eq0}
(\beta \bfu_h,\mathbb{\Pi}_h\bfu_h)_{L^2(\Omega^{\Gamma}_h)} \ge C\| \bfu_h \|^2_{L^2(\Omega^{\Gamma}_h)}, ~~~~ \forall \bfu_h \in \mathcal{IND}_h(\Omega)\cap\emph{Ker(curl)}.
\end{equation}
\end{lemma}
\begin{proof}
The argument is tedious based on direct calculation, so we put it in the Appendix \ref{App_A3}.
\end{proof}

Now we are ready to show the \textit{inf-sup} condition in \eqref{inf_sup}.
\begin{thm}
\label{thm_infsup}
Under the conditions of Lemma \ref{lem_loc_infsup} and for $h$ sufficiently small, the \textit{inf-sup} condition \eqref{inf_sup} holds regardless of interface location relative to the mesh.
\end{thm}
\begin{proof}
First all, Theorem \ref{thm_stiff} directly yields
\begin{equation}
\label{thm_infsup_eq1-1}
(\mu^{-1} ~\text{curl}~\bfu_h, \text{curl}~\mathbb{\Pi}_h\bfu_h)_{L^2(\Omega)} = (\mu^{-1} ~\text{curl}~\bfu_h, \text{curl}~\bfu_h)_{L^2(\Omega)} \ge C \| \text{curl}~\bfu_h \|^2_{L^2(\Omega)}.
\end{equation}
Since $\mathbb{\Pi}_h\bfu_h=\bfu_h$ on $\Omega\backslash\Omega^{\Gamma}_h$, we certainly have
\begin{equation}
\label{thm_infsup_eq1}
(\beta~ \bfu_h, \mathbb{\Pi}_h\bfu_h)_{L^2(\Omega\backslash\Omega^{\Gamma}_h)} \ge C \| \bfu_h \|^2_{L^2(\Omega\backslash\Omega^{\Gamma}_h)}.
\end{equation}
As for $\Omega^{\Gamma}_h$, applying the decomposition in Lemma \ref{lemma_special_decomp}, the estimate in Lemma \ref{lem_loc_infsup} and the stability in Theorem \ref{thm_iso_stab} together with the arithmetic inequality, for $h$ sufficiently small we have
\begin{equation}
\begin{split}
\label{thm_infsup_eq2}
 & (\beta~(\bfu^{\ast}_h+\mathring{\bfu}_h), \mathbb{\Pi}_h(\bfu^{\ast}_h+\mathring{\bfu}_h))_{L^2(\Omega^{\Gamma}_h)} \\
= & (\beta \bfu^{\ast}_h, \bfu^{\ast}_h)_{L^2(\Omega^{\Gamma}_h)} + (\beta~\mathring{\bfu}_h, \mathbb{\Pi}_h\mathring{\bfu}_h)_{L^2(\Omega^{\Gamma}_h)} \\
 + & (\beta \bfu^{\ast}_h, \mathbb{\Pi}_h\bfu^{\ast}_h - \bfu^{\ast}_h)_{L^2(\Omega^{\Gamma}_h)}  + (\beta~\bfu^{\ast}_h, \mathbb{\Pi}_h\mathring{\bfu}_h)_{L^2(\Omega^{\Gamma}_h)} + (\beta~\mathring{\bfu}_h, \mathbb{\Pi}_h\bfu^{\ast}_h)_{L^2(\Omega^{\Gamma}_h)} \\
 \ge & C( \| \bfu^{\ast}_h \|^2_{L^2(\Omega^{\Gamma}_h)} + \| \mathring{\bfu}_h \|^2_{L^2(\Omega^{\Gamma}_h)} ) - C(h^2  \| \text{curl}~ \bfu_h \|^2_{L^2(\Omega)} + h \| \text{curl}~ \bfu_h \|^2_{L^2(\Omega)} + h\| \mathring{\bfu}_h \|^2_{L^2(\Omega^{\Gamma}_h)}) \\
 \ge & C\| \bfu^{\ast}_h \|^2_{L^2(\Omega^{\Gamma}_h)} + (C-Ch) \| \mathring{\bfu}_h \|^2_{L^2(\Omega^{\Gamma}_h)} - Ch \| \text{curl}~ \bfu_h \|^2_{L^2(\Omega)} \\
\ge & C\| \bfu_h \|^2_{L^2(\Omega^{\Gamma}_h)}  - Ch \| \text{curl}~ \bfu_h \|^2_{L^2(\Omega)} .
\end{split}
\end{equation}
Noting that $\mathring{\bfu}_h$ is curl-free, we finally obtain from Theorem \ref{thm_stiff} and \eqref{thm_infsup_eq1-1}-\eqref{thm_infsup_eq2} that
\begin{equation}
\begin{split}
\label{thm_infsup_eq3}
a_h(\bfu_h,\mathbb{\Pi}_h\bfu_h) = & (\mu~\text{curl}~\bfu_h, \text{curl}~\mathbb{\Pi}_h\bfu_h)_{L^2(\Omega)} + (\beta~\bfu_h, \mathbb{\Pi}_h \bfu_h)_{L^2(\Omega\backslash\Omega^{\Gamma}_h)} +  (\beta~\bfu_h, \mathbb{\Pi}_h \bfu_h)_{L^2(\Omega^{\Gamma}_h)} \\ 
\ge & (C-Ch) \| \text{curl}~ \bfu_h \|^2_{L^2(\Omega)} + C\| \bfu_h \|^2_{L^2(\Omega^{\Gamma}_h)} + C \| \bfu_h \|^2_{L^2(\Omega\backslash\Omega^{\Gamma}_h)} \ge C \| \bfu_h \|^2_{\bfH(\text{curl};\Omega)}
\end{split}
\end{equation}
which yields the desired result for $h$ sufficiently small.
\end{proof}
\begin{rem}[Comments on the analysis \textit{inf-sup} stability.]
\label{rem_infsup_1}
When the discontinuity is only on permeability $\mu$, i.e., the conductivity is continuous ($\beta^-=\beta^+$), Remark \ref{rem_identity} indicates their curl-free subspaces are identical as constant vectors on each interface element $T$. So by the Poincar\'e-type inequality, we have the local estimates for every $\bfu_h\in\mathcal{IND}_h(T)$:
\begin{equation}
\label{rem_infsup_1_eq1}
\| \bfu_h - \mathbb{\Pi}_h\bfu_h \|_{L^2(T)} \le C h_T \| \emph{curl}~ \bfu_h \|_{L^2(T)}.
\end{equation}
Then the analysis of the \textit{inf-sup} stability is straightforward by using \eqref{rem_infsup_1_eq1} to estimate the lower-order term $(\beta\bfu_h,\mathbb{\Pi}_h\bfu_h)_{L^2(\Omega)}$. However the estimate \eqref{rem_infsup_1_eq1} is only true for continuous conductivity. Actually since $\mathcal{IND}_h(T)$ can not recover the local spaces of constant vectors when the conductivity is discontinuous. That is, essentially, $\mathcal{IND}_h(T)$ and $\mathcal{ND}_h(T)$ have different curl-free subspaces, and thus there is no approximation results between them. 

According to the explanation above, by inspecting the proof of Theorem \ref{thm_infsup}, we expect that the most difficult part in the analysis for discontinuous conductivity is the \textit{inf-sup} stability on their curl-free subspaces. By Lemma \ref{lemma_special_decomp}, the non-curl-free component near the interface can be bounded by its curl value with a scaling factor related to the mesh size which is then used in \eqref{thm_infsup_eq3} to obtain a positive lower bound. A similar argument can be found in \cite{2006HiptmairWidmerZou} (Lemma 5.3) on a strip around the domain boundary. As for the curl-free component, in this article, our approach is based on the direct calculation locally on each interface element of $(\beta\bfu_h,\mathbb{\Pi}_h\bfu_h)_{L^2(T)}$ to estimate the largest contrast of $\beta$ such that it is lower bounded by $\|\bfu_h\|^2_{L^2(T)}$, that is Lemma \ref{lem_loc_infsup}. The Cartesian assumption for the mesh is mainly for relatively easy calculation, and we expect other bounds may be also obtained for other meshes through a similar derivation. We also note that the Cartesian mesh is easy to construct and thus widely used for interface-unfitted methods.

Indeed, numerical results suggest that $(\beta\bfu_h,\mathbb{\Pi}_h\bfu_h)_{L^2(\Omega)}$ fails to satisfy the \textit{inf-sup} condition when $\mathbb{\Pi}_h\bfu_h$ is chosen as the test function and the contrast of $\beta$ is too large. However in our numerical experiments we have not observed any instability issue for large contrast of conductivity. So from the perspective of analysis, $\mathbb{\Pi}_h$ may not be a suitable operator to generate the test function, but we have not known yet which operator can fulfill the requirements especially the \text{inf-sup} stability on the curl-free subspaces. We think the challenge mainly arises from the insufficient approximation between $\mathcal{IND}_h(T)$ and  $\mathcal{ND}_h(T)$. Due to exact-sequence in Theorem \ref{thm_DR_3}, this is equivalent to show the inf-sup stability for the PG-IFE method for the $H^1$-elliptic interface problems \cite{2013HouSongWangZhao} which remains open for years. 

\end{rem}

Furthermore we have the continuity of the bilinear form $a(\cdot,\cdot)$.
\begin{thm}
\label{thm_contin}
There holds
\begin{equation}
\label{thm_contin_eq0}
a(\bfu_h,\bfv_h) \le C \| \bfu_h \|_{\bfH(\emph{curl};\Omega)} \| \bfv_h \|_{\bfH(\emph{curl};\Omega)}.
\end{equation}
\end{thm}
\begin{proof}
It directly follows from the H\"older's inequality.
\end{proof}

Finally the standard argument yields the following error estimates.
\begin{thm}
\label{thm_solu_err}
Let $\bfu_h$ be the solution to the scheme \eqref{weak_form_3}. Under the conditions of Theorem \ref{thm_infsup}, there holds
\begin{equation}
\label{thm_solu_err_eq0}
\| \bfu -\bfu_h \|_{\bfH(\emph{curl};\Omega)} \le C h \left( \| \bfu \|_{\bfH^1(\emph{curl};\Omega^-)} + \| \bfu \|_{\bfH^1(\emph{curl};\Omega^+)}  \right).
\end{equation}
\end{thm}
\begin{proof}
By the \textit{inf-sup} stability, Theorem \ref{thm_contin} and the interpolation error estimate in Theorem \ref{thm_interp_err_5}, we have
\begin{equation}
\begin{split}
\label{thm_solu_err_eq1}
\| \bfu_h -  \widetilde{\Pi}_h\bfu \|_{\bfH(\text{curl};\Omega)} &\le C \sup_{\bfv_h\in\mathcal{ND}_h(\Omega)} \frac{a(\bfu_h - \widetilde{\Pi}_h\bfu, \bfv_h ) }{ \| \bfv_h \|_{\bfH(\text{curl};\Omega)} } = C \sup_{\bfv_h\in\mathcal{ND}_h(\Omega)} \frac{a(\bfu - \widetilde{\Pi}_h\bfu, \bfv_h ) }{ \| \bfv_h \|_{\bfH(\text{curl};\Omega)} }\\
&\le C \| \bfu -  \widetilde{\Pi}_h\bfu_h \|_{\bfH(\text{curl};\Omega)} \le C h \left( \| \bfu \|_{\bfH^1(\text{curl};\Omega^-)} + \| \bfu \|_{\bfH^1(\text{curl};\Omega^+)}  \right).
\end{split}
\end{equation}
Then the triangular inequality yields the desired result.
\end{proof}


\section{Numerical Examples}
\label{sec:num_examp}
In this section, we present a group of numerical experiments to validate the previous analysis. We also compare the numerical performance of the proposed PG-IFE method with the classic IFE method and penalty-type IFE method. The latter two use the Galerkin formulation, namely the IFE functions are used as both the trial functions and test functions. More specifically, they are to find $\bfu_h\in \mathcal{IND}_h(\Omega)$ such that
\begin{equation}
\label{otherIFE_12}
a^{(i)}_h(\bfu_h,\bfv_h) = \int_{\Omega} \bff\cdot\bfv_h dX ~~~~  \forall \bfv_h \in  \mathcal{IND}_h(\Omega)
\end{equation}
where $i=1,2$, and the bilinear form for the penalty-type IFE method is given by
\begin{equation}
\begin{split}
\label{otherIFE_1}
a^{(1)}_h(\bfu_h,\bfv_h) &= \int_{\Omega}\mu^{-1}\text{curl}~\bfu_h\cdot\text{curl}~\bfv_h dX + \int_{\Omega} \beta \bfu_h\cdot\bfv_h dX  - \int_{\mathcal{E}^i_h} \{\mu^{-1}\text{curl}~\bfu_h\}_e[\bfv_h\cdot\bft]_e ds \\
&- \int_{\mathcal{E}^i_h} \{\mu^{-1}\text{curl}~\bfv_h\}_e[\bfu_h\cdot\bft]_e ds + \frac{c_0\max\{\beta^-,\beta^+\}}{h^r} \int_{\mathcal{E}^i_h} [\bfu_h\cdot\bft]_e \cdot [\bfv_h\cdot\bft]_e ds
\end{split}
\end{equation}
in which $\mathcal{E}^i_h$ denotes the collection of interface edges, $c_0$ is a positive constant parameter independent of the mesh size, $r$ is real number parameter, $[\bfw_h\cdot\bft]_e=\bfw_h|_{T_1}\cdot\bft-\bfw_h|_{T_2}\cdot\bft$, and $\{\mu^{-1}\text{curl}~\bfw_h\}_e= \frac{1}{2}\left(\mu^{-1}\text{curl}~\bfw_h|_{T_1} + \mu^{-1}\text{curl}~\bfw_h|_{T_2} \right)$, and the bilinear form for the classic IFE method is
\begin{equation}
\label{otherIFE_2}
a^{(2)}_h(\bfu_h,\bfv_h) = \int_{\Omega}\mu^{-1}\text{curl}~\bfu_h\cdot\text{curl}~\bfv_h dX + \int_{\Omega} \beta \bfu_h\cdot\bfv_h dX.
\end{equation}
Many IFE functions are in general discontinuous across interface edges, and thus the standard Galerkin scheme, referred as the classic IFE (C-IFE) method, may produce suboptimally convergent solutions due to the non-conformity errors on interface edges. This is indeed numerically observed in \cite{2015LinLinZhang} for $H^1$-elliptic interface problems. To fix this issue, the authors in \cite{2015LinLinZhang} proposed the penalty-type method, referred as the partially penalized IFE method (PP-IFE) method since the penalties are only added on interface edges to handle the discontinuities, and the PP-IFE method is then shown to work for various different interface problems. Here we follow the literature to call \eqref{otherIFE_1} and \eqref{otherIFE_2} the PP-IFE and C-IFE methods respectively. Unfortunately, even the PP-IFE method can not produce optimally convergent solutions for the considered $\bfH(\text{curl})$-elliptic interface problems, and actually it fails to converge near the interface, as shown by the numerical results below. As discussed in \cite{2016CasagrandeHiptmairOstrowski,2016CasagrandeWinkelmannHiptmairOstrowski} the similar issue also occurs in Nitsche's penalty methods, and the problem is on the stability term which yields the loss of convergence due to the $\mathcal{O}(h)$ approximation capability for $\bfH^1(\text{curl};\Omega)$ functions. Similarly, as for the IFE method, it is the last term in \eqref{otherIFE_1} that is responsible for the loss of accuracy, and one can quickly verify that the interpolation error gauged by the corresponding energy norm involving the stability term does not have optimal convergence rate.

We consider a domain $\Omega=(-1,1)\times(-1,1)$ which is partitioned into $N\times N$ squares and each square is then cut into two triangles along the diagonal, i.e., the semi Cartesian mesh shown in the right plot of Figure \ref{fig:submesh}. Let a circular interface $\Gamma:x^2+y^2=r^2_1$ cut $\Omega$ into the inside subdomain $\Omega^-$ and the outside subdomain $\Omega^+$. On $\Omega$ the exact solution is given by
\begin{equation}
\label{exact_solu}
\bfu = 
\begin{cases}
      & \left(\begin{array}{c} \mu^-\left(  - k_1(r_1^2 - x^2 -y^2)y \right) \\ \mu^-\left(  -k_1(r_1^2 - x^2 -y^2)x  \right) \end{array}\right) ~~~~ \text{in}~ \Omega^-  \\
      & \left(\begin{array}{c} \mu^+\left(  - k_2(r_2^2-x^2-y^2)(r_1^2-x^2-y^2)y  \right) \\  \mu^+\left( - k_2(r_2^2-x^2-y^2)(r_1^2-x^2-y^2)x \right) \end{array}\right) ~~~~ \text{in}~ \Omega^+
\end{cases}
\end{equation}
for which the boundary conditions and the right hand side $\bff$ are calculated accordingly, and $k_2=20$, $k_1=k_2(r_2^2-r_1^2)$ with $r_1=\pi/5$ and $r_2=1$. We focus the numerical experiments on four groups of parameters: fixing $\mu^-=\beta^-=1$ and varying $\mu^+=1/10$ or $1/100$ and $\beta^+=10$ or $100$. Here we emphasize that, although the analysis is only based on the small contrast of conductivities (less than 10.56), there is no issue in computation for larger contrast. Moreover we choose the stability parameters in \eqref{otherIFE_1} to be $c_0=10$ and $r=1$, and other choices such as $c_0=-10,0,100$ and $r=-1,-1/2,0,1/2,1$ can give the similar results. Furthermore, we let $e_0$ be the error $\| \bfu - \bfu_h \|_{\bfH(\text{curl};\Omega)}$, and in order to study the convergence behavior around the interface we also define the error
\begin{equation}
\label{err_interf}
e_1 = |\Omega^i_h|^{-1/2}\| \bfu - \bfu_h \|_{\bfH(\text{curl};\Omega^i_h)}.
\end{equation}
A similar indicator was also used in \cite{2010LiMelenkWohlmuthZou} to study the error near the interface for $H^1$-elliptic interface problems. 

The results for the errors $e_0$ are presented in Figure \ref{fig:err0} where the convergence behavior of PG-IFE, PP-IFE and C-IFE methods are indicated by black, red and blue curves, respectively. In addition there are three dashed lines with the corresponding color indicating the expected convergence rate $\mathcal{O}(h)$ for the PG-IFE method and the approximate rates $\mathcal{O}(h^{1/2})$ for the PP-IFE and C-IFE methods. As we can see from the plots, the black error curve almost overlaps with the related dashed line for the PG-IFE method, namely its convergence rate is certainly optimal. However, as for the PP-IFE and C-IFE methods, we note that the errors asymptotically have the suboptimal $\mathcal{O}(h^{1/2})$ convergence rates. Moreover, as the contrast of $\beta$ becomes larger, the advantages of the PG-IFE method over the other two are more evident. 

It is interesting to point out that straightforwardly applying the argument of Theorem 2 in \cite{2016CasagrandeHiptmairOstrowski} actually suggests that the PP-IFE method should not converge at all near the interface. So we expect the loss of $\mathcal{O}(h^{1/2})$ may be due to the pollution of the error near the interface over the whole domain. To further study this issue, we compute and plot the error $e_1$ defined in \eqref{err_interf} for PG-IFE, PP-IFE and C-IFE methods in Figure \ref{fig:err1}. Still the black dashed line indicates the expected optimal $\mathcal{O}(h)$ convergence rate for the PG-IFE method which matches the true error curve quite well. So the method has the optimal accuracy even near the interface. But the numerical results clearly suggest that both the PP-IFE and the C-IFE methods completely fail to converge near the interface. Combining the numerical experiments here and in \cite{2016CasagrandeHiptmairOstrowski}, we actually think this issue seems to be very difficult to overcome for penalty-type methods if the solutions' regularity are described by $\bfH^k(\text{curl};\Omega)$ spaces. We believe this clearly shows advantages for the proposed PG formulation without penalties.

\begin{figure}
        \centering
        \begin{subfigure}[b]{0.45\textwidth}
            \centering
            \includegraphics[width=2.3in]{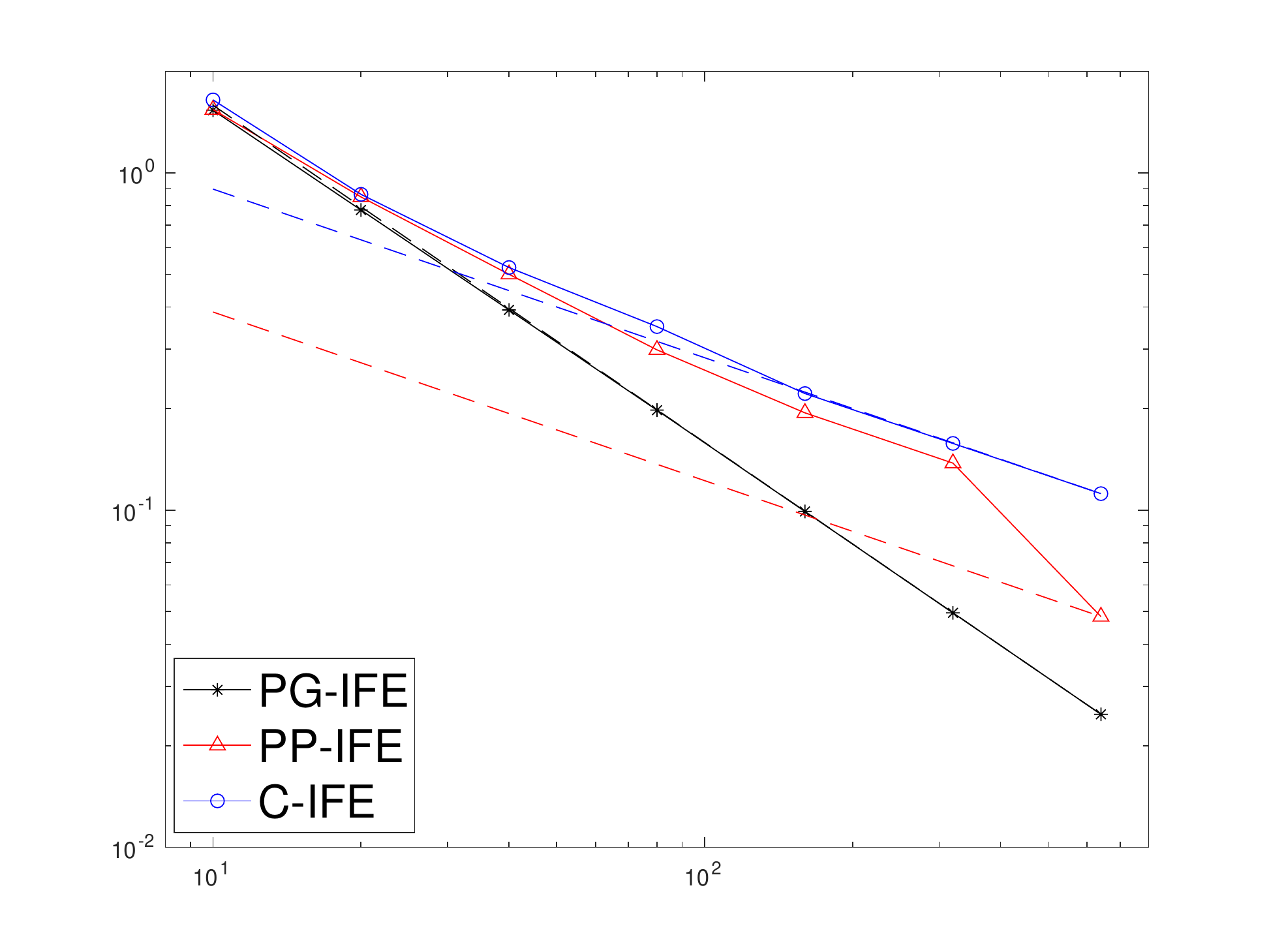}
            \caption{$\mu^+=1/10$ and $\beta^+=10$}%
            \label{fig:err01}
        \end{subfigure}
        \hfill
        \begin{subfigure}[b]{0.45\textwidth}  
            \centering 
            \includegraphics[width=2.3in]{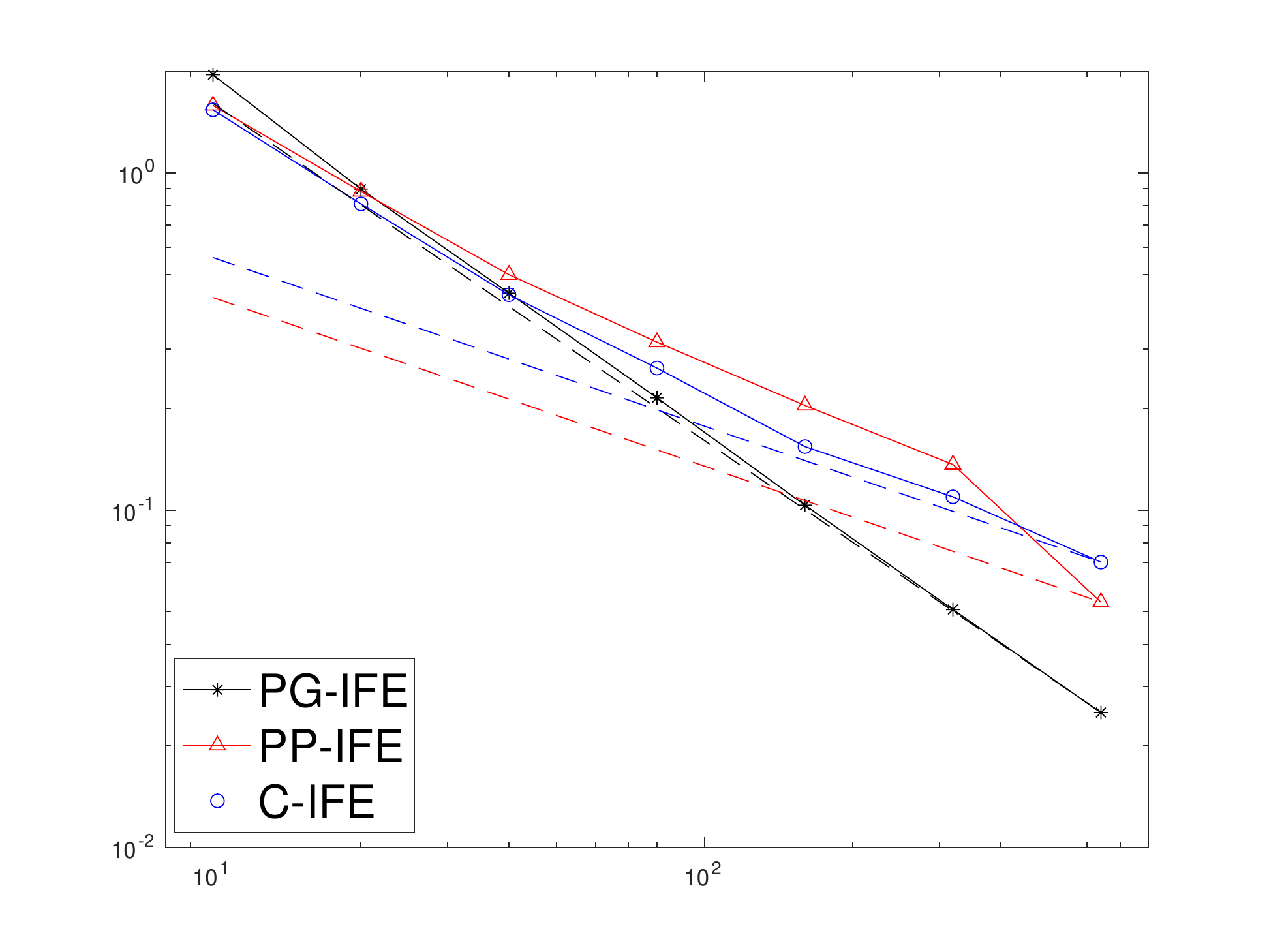}
             \caption{$\mu^+=1/10$ and $\beta^+=100$}%
            \label{fig:err02}
        \end{subfigure}
        \begin{subfigure}[b]{0.45\textwidth}   
            \centering 
            \includegraphics[width=2.3in]{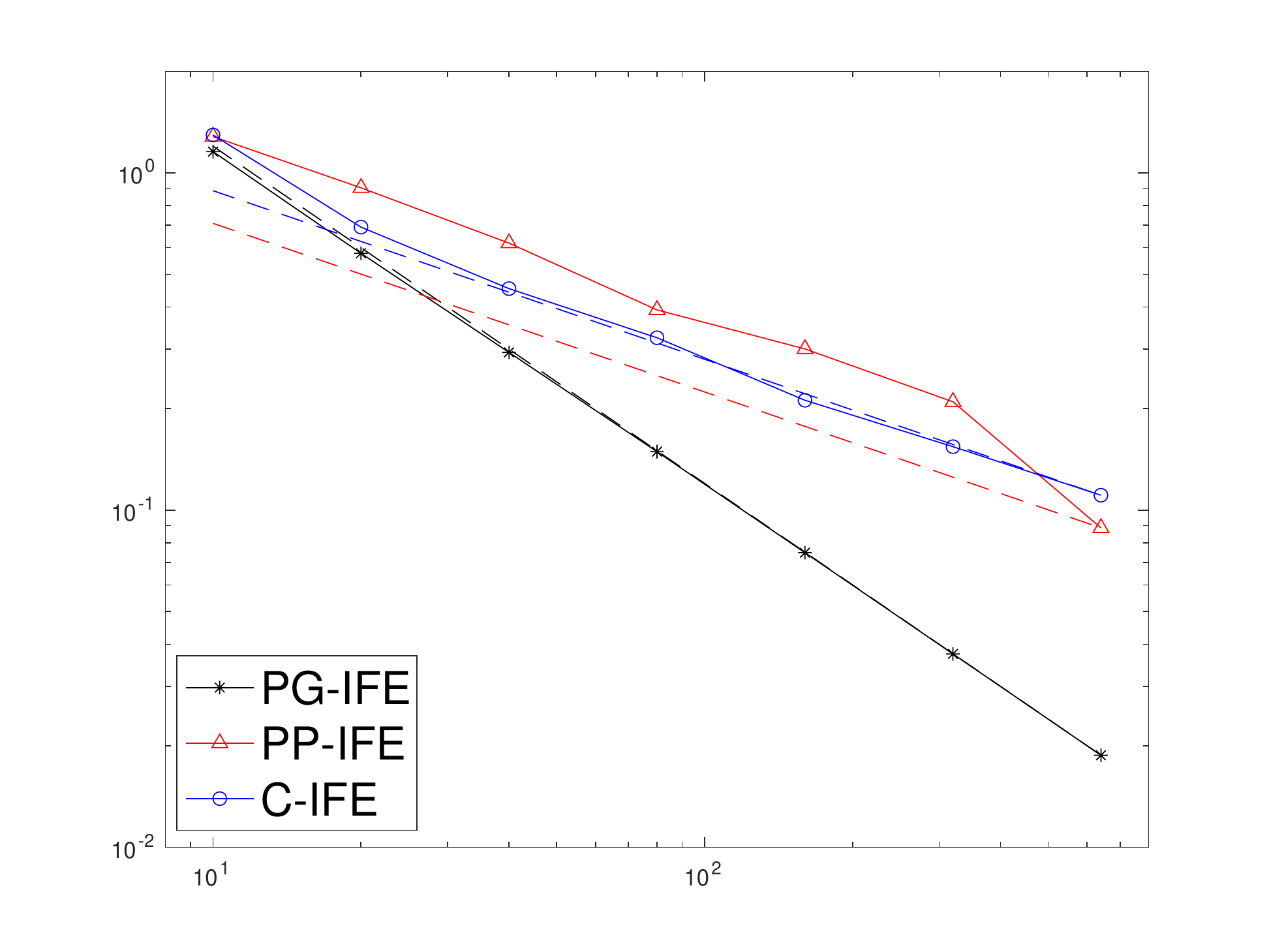}
             \caption{$\mu^+=1/100$ and $\beta^+=10$}%
            \label{fig:err03}
        \end{subfigure}
        \hfill
        \begin{subfigure}[b]{0.45\textwidth}   
            \centering 
            \includegraphics[width=2.3in]{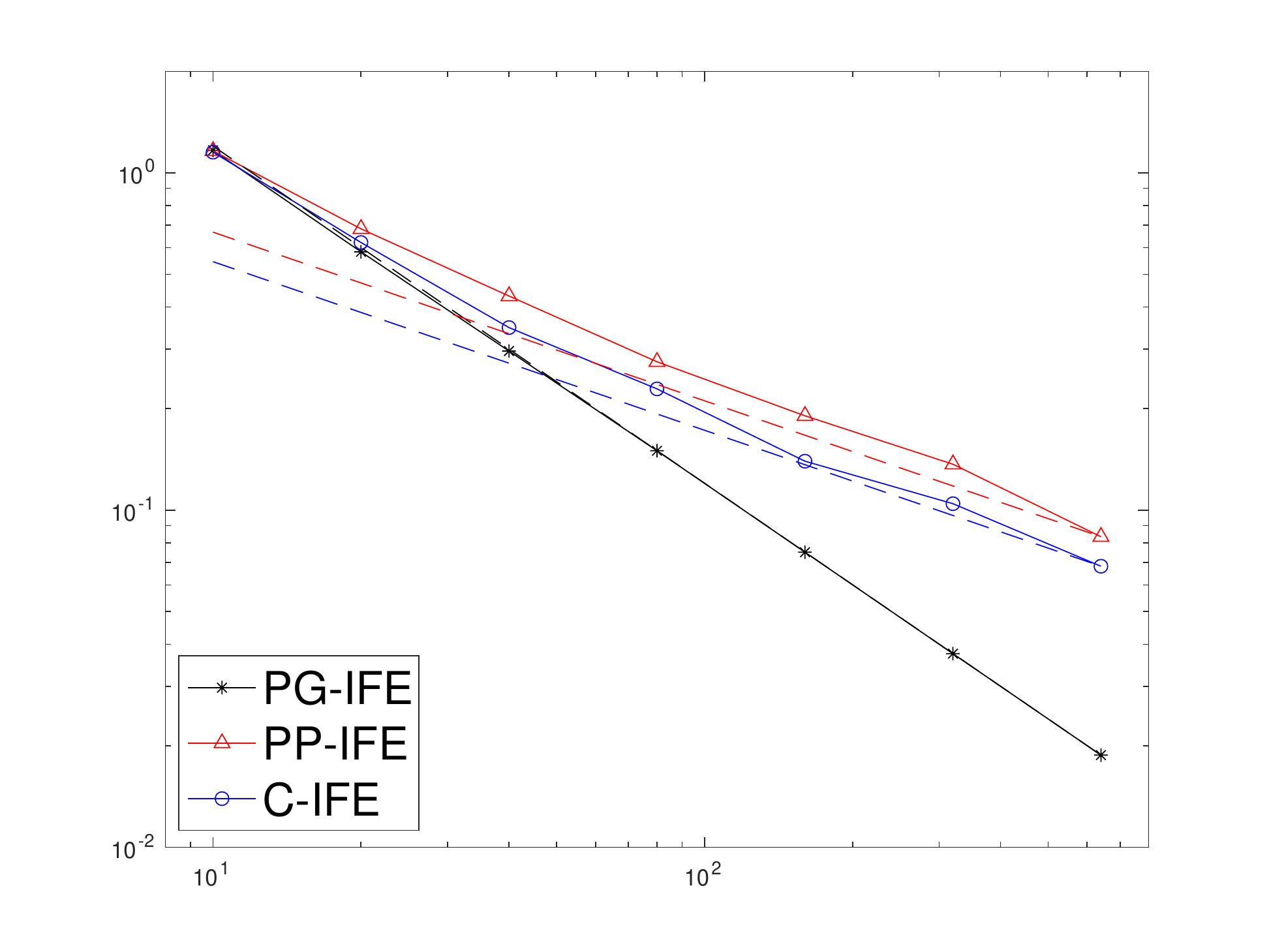}
             \caption{$\mu^+=1/100$ and $\beta^+=100$}%
            \label{fig:err04}
        \end{subfigure}
        \caption{The convergence rates for the errors $e_0$ of PG-IFE, PPIFE, and C-IFE methods.}
        \label{fig:err0}
    \end{figure}


\begin{figure}
        \centering
        \begin{subfigure}[b]{0.45\textwidth}
            \centering
            \includegraphics[width=2.3in]{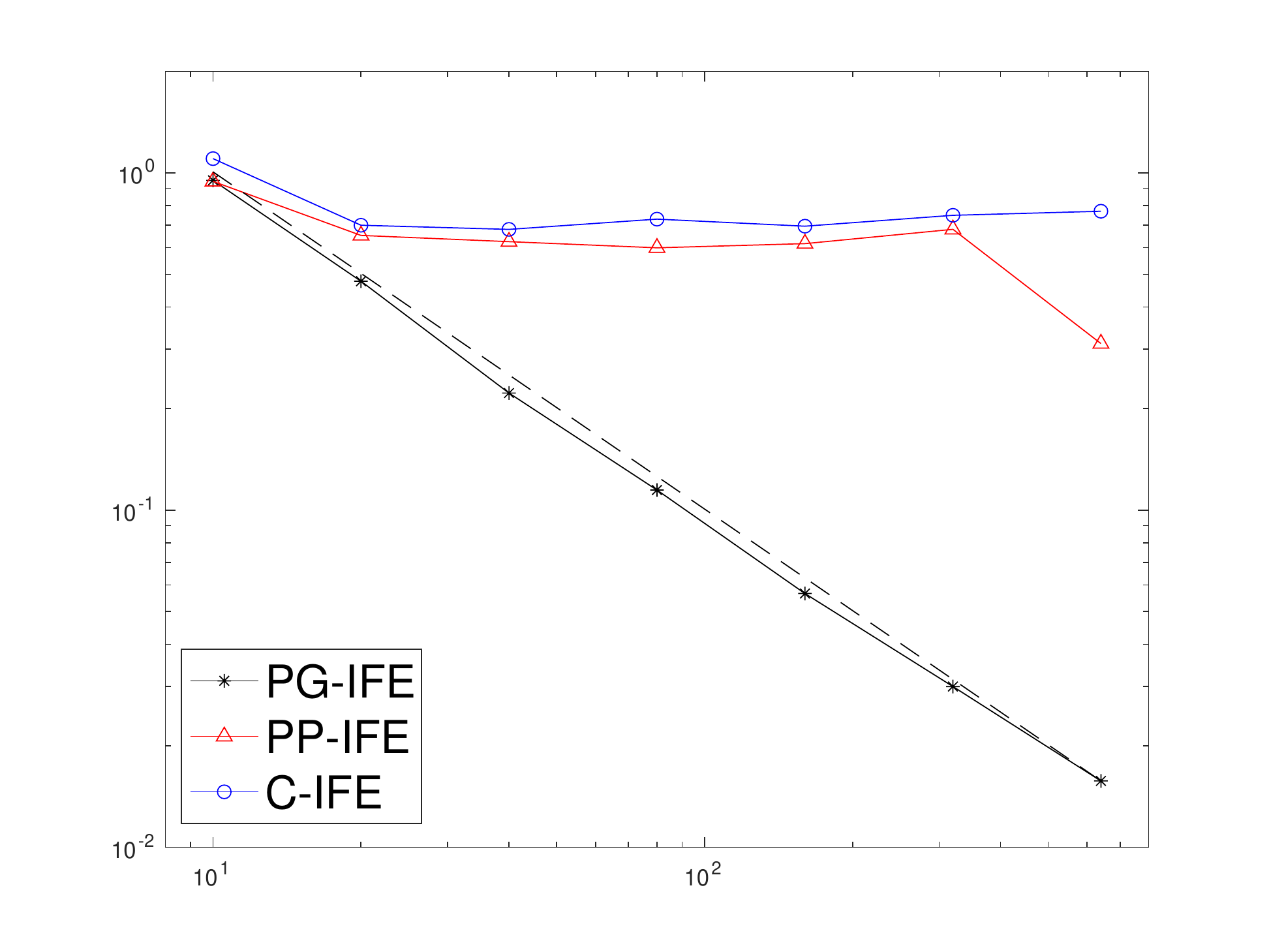}
            \caption{$\mu^+=1/10$ and $\beta^+=10$}%
            \label{fig:err11}
        \end{subfigure}
        \hfill
        \begin{subfigure}[b]{0.45\textwidth}  
            \centering 
            \includegraphics[width=2.3in]{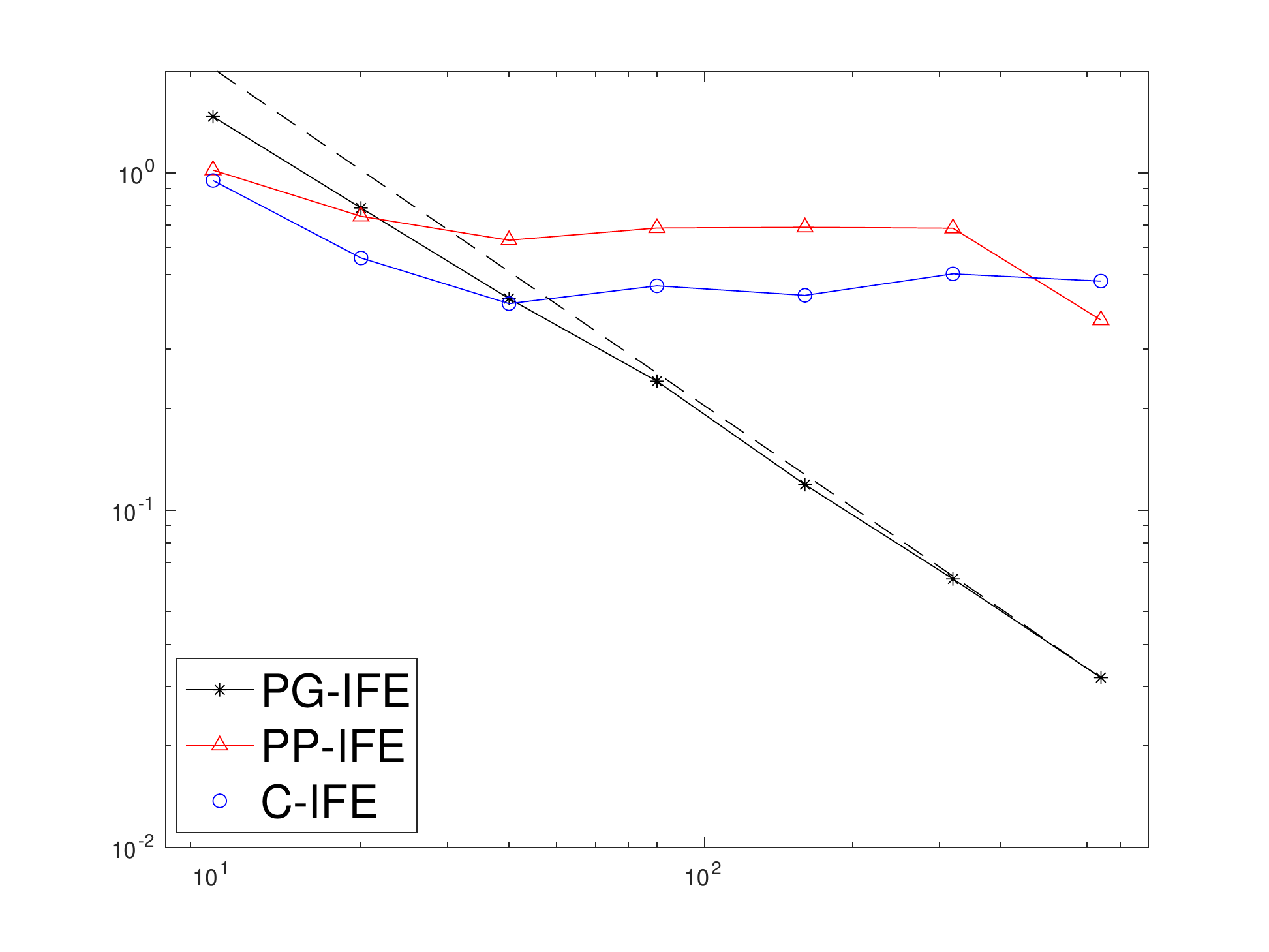}
             \caption{$\mu^+=1/10$ and $\beta^+=100$}%
            \label{fig:err12}
        \end{subfigure}
        \begin{subfigure}[b]{0.45\textwidth}   
            \centering 
            \includegraphics[width=2.3in]{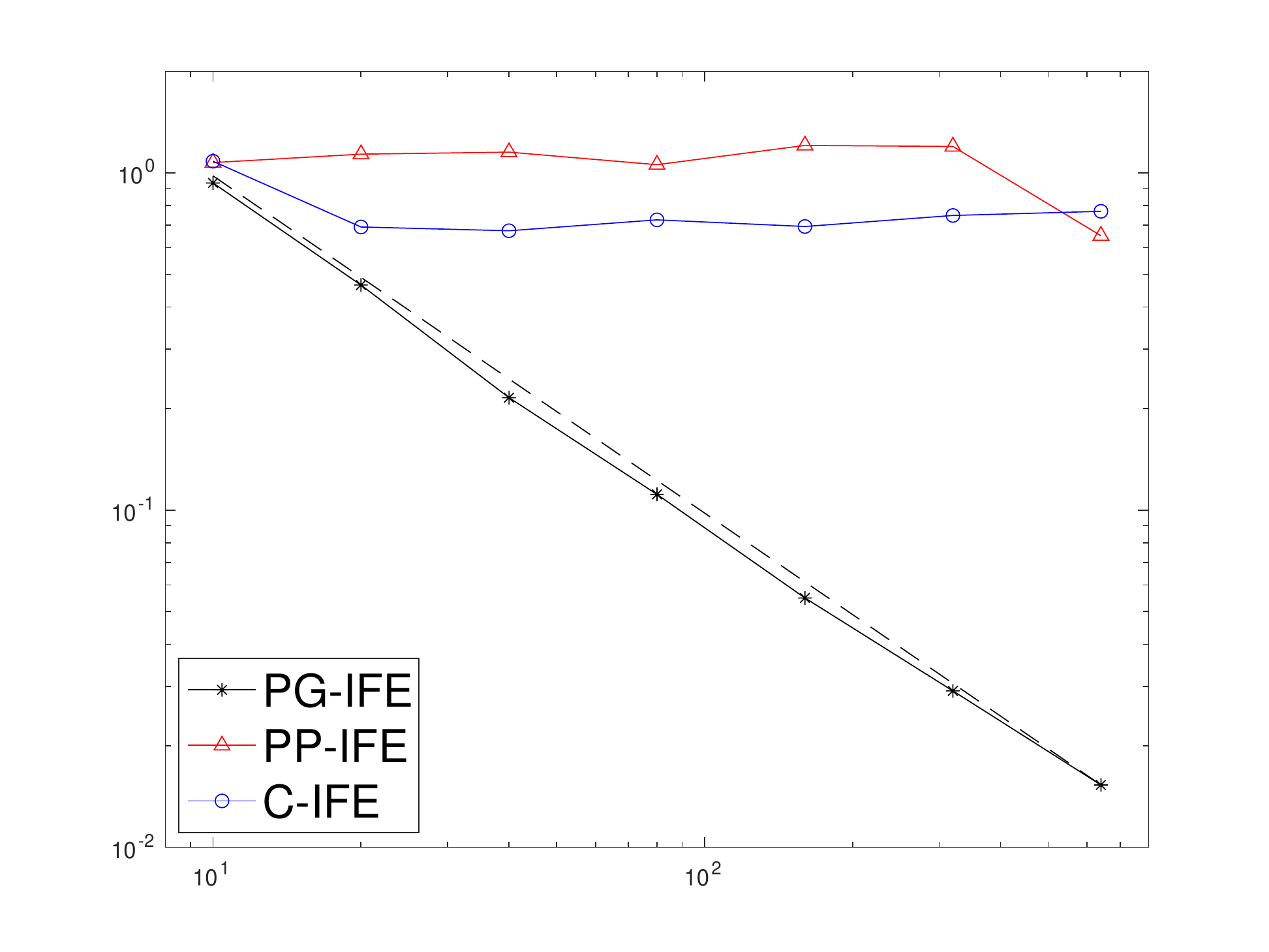}
             \caption{$\mu^+=1/100$ and $\beta^+=10$}%
            \label{fig:err13}
        \end{subfigure}
        \hfill
        \begin{subfigure}[b]{0.45\textwidth}   
            \centering 
            \includegraphics[width=2.3in]{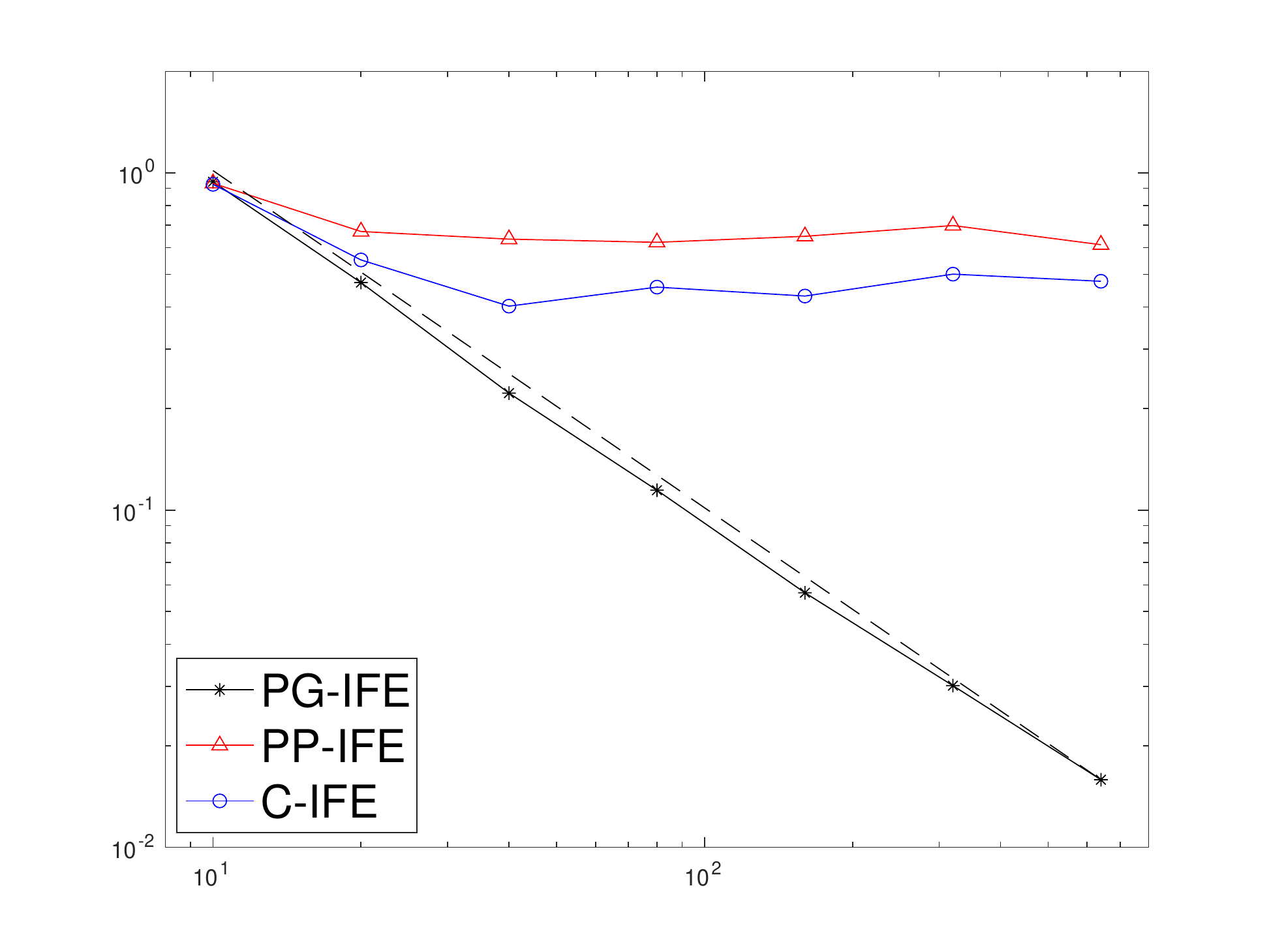}
             \caption{$\mu^+=1/100$ and $\beta^+=100$}%
            \label{fig:err14}
        \end{subfigure}
        \caption{The convergence rates for the errors $e_1$ of PG-IFE, PPIFE, and C-IFE methods}
        \label{fig:err1}
    \end{figure}


\begin{appendices}
  \section{Technical Results}
  In this section, we present the proof of some technical results.
  \subsection{Technical Results for \eqref{thm_unisol_eq1}}
  \label{App_A1}
  First of all, we let $l_2=|e_2|$, $l_3=|e_3|$ and $\theta=\angle A_3A_1A_2$ as shown by the left plot in Figure \ref{fig:ref_elem_1}. Then we can express $B_T$ as $B_T = \left[\begin{array}{cc} l_3 & l_2\cos(\theta) \\ 0 & l_2\sin(\theta) \end{array}\right]$.
Let $\delta$ be the angle between the normal vector $\bar{\bfn}$ and the $x_1$ axis, and it is easy to see $\delta\in[\theta-\pi/2,\pi/2]$. Then we can write
\begin{equation}
\label{App_A1_eq1_extra1}
\hat{\bar{\bfn}} = \frac{1}{l_2l_2 \sin(\theta)} \left[\begin{array}{c} l_2 \sin(\theta-\delta) \\ l_3\sin(\delta) \end{array}\right].
\end{equation}
  For the first part of this lemma, by direct calculation we have $\hat{\bar{\bfn}}'\cdot\hat{\bar{\bfn}} = \frac{|\Gamma^T_h|}{l_2l_3\sin(\theta)}>0$.
  \begin{equation}
\label{App_A1_eq2}
2\alpha = \hat{\bar{\bfn}}'\cdot\hat{\bar{\bfn}} = \frac{|\Gamma^T_h|}{l_2l_3\sin(\theta)}>0.
\end{equation}
For the second part of this lemma, we first show that $\hat{n}_1+\hat{n}_2\ge 0$. Also by direct calculation, we obtain
\begin{equation}
\label{App_A1_eq3}
\hat{n}_1+\hat{n}_2 = \frac{ l_2 \sin(\theta-\delta) + l_3\sin(\delta) }{l_2l_3 \sin(\theta)}
= \frac{\cos(\delta)}{\sin(\theta)} \left( \frac{ \sin(\theta) - \cos(\theta)\tan(\delta) }{l_3} + \frac{\tan(\delta)}{l_2} \right).
\end{equation}
Since $\cos(\delta),\sin(\theta)>0$ and the function in the parentheses is linear with respect to $\tan(\delta)$. So we only need to check its values at the two boundary $\delta=\theta-\pi/2$ and $\delta=\pi/2$. Thus we denote the function in \eqref{App_A1_eq3} by $f(\delta)$, and note that $f(\theta - \pi/2) = \frac{l_2 - l_3 \cos(\theta)}{l_2l_3 \sin(\theta)} \ge 0$
since $\angle A_2A_3A_1 \le \pi/2$ and similarly have $f(\theta)\ge 0$ since $\angle A_1A_2A_3 \le \pi/2$. Therefore the quantity $\frac{de(\hat{n}_1+\hat{n}_2)}{\hat{\bar{\bfn}}'\cdot\hat{\bar{\bfn}}}$ is non-negative. In addition, we can derive
\begin{equation}
\label{App_A1_eq5}
e\hat{n}_1 + d\hat{n}_2 - de( \hat{n}_1 + \hat{n}_2 ) = l_2l_3 \cos(\theta) \left( e^2(1-d) \frac{l_2}{l_3\cos(\theta)} + d^2(1-e)\frac{l_3}{l_2\cos(\theta)} - (2de-d^2e-e^2d) \right).
\end{equation}
Again, since $\angle A_1A_2A_3\le \pi/2$ and $\angle A_2A_3A_1\le \pi/2$, we have $l_3\ge l_2 \cos(\theta)$ and $l_2\ge l_3\cos(\theta)$. Therefore, by $\theta\le \pi/2$ and $d,e\in[0,1]$, we obtain $e\hat{n}_1 + d\hat{n}_2 - de( \hat{n}_1 + \hat{n}_2 ) \ge l_2l_3 \cos(\theta) ( e^2+d^2-2de ) \ge 0$
which yields the desired result. 


  \subsection{Proof of Theorem \ref{thm_bound}}
  \label{App_A2}
We first give the explicit expression of the matrix in \eqref{IFE_fun_ref_5}:
\begin{equation}
\label{App_A2_eq1}
\mathbf{ I}  +  \mathbf{ R} \mathbf{ A}^{-1} \mathbf{ B} =
\left[\begin{array}{cc}
1 + \frac{de}{2}\left( -\kappa + \frac{\lambda s_2}{\alpha} \right) & \frac{de}{2}\left( -\kappa + \frac{\lambda s_3}{\alpha} \right) \\
\frac{de}{2}\left( -\kappa - \frac{\lambda s_2}{\alpha} \right) & 1 + \frac{de}{2}\left( -\kappa - \frac{\lambda s_3}{\alpha} \right)
\end{array}\right]
\end{equation}
where $s_2=\hat{\bfphi}_2(X_m)\cdot\hat{\bar{\bfn}}=-\frac{e}{2}\hat{n}_1 + \left( \frac{d}{2}-1 \right)\hat{n}_2$ and $s_3=\hat{\bfphi}_3(X_m)\cdot\hat{\bar{\bfn}}= \left( 1-\frac{e}{2} \right)\hat{n}_1 + \frac{d}{2} \hat{n}_2$. \eqref{App_A1_eq1_extra1} and \eqref{App_A1_eq2} yield $\frac{ s_2}{\alpha} = \frac{( -e l_2 \sin(\theta - \delta) +(d-2) l_3 \sin(\delta) )}{2|\Gamma^T_h|}$.
We note that $|\Gamma^T_h| \ge  \max\{dl_3,el_2\}\sin(\theta)$, so we have
\begin{equation}
\label{App_A2_eq3}
\verti{ \frac{de s_2}{\alpha} } \le \frac{ de^2l_2  +d^2el_3 +2del_3 }{2 \max\{dl_3,el_2\}\sin(\theta)} \le C
\end{equation}
due to the shape regularity of $T$ and $d,e\le1$. A similar bound applies to $\verti{ \frac{de s_3}{\alpha} }$. Therefore by \eqref{thm_unisol_eq2} we have
\begin{equation}
\label{App_A2_eq4}
\| (\mathbf{ I}  +  \mathbf{ R} \mathbf{ A}^{-1} \mathbf{ B})^{-1} \|_{\infty}\le C \frac{  \max\Big\{ 1, \frac{\mu^+}{\mu^-} \Big\} + \max\Big\{ 1, \frac{\beta^-}{\beta^+} \Big\} }{ \min\Big\{ 1, \frac{\mu^+}{\mu^-} \Big\} \min\Big\{ 1, \frac{\beta^-}{\beta^+} \Big\} } \le C. 
\end{equation}
Furthermore, we note that
\begin{equation}
\label{App_A2_eq5}
\mathbf{ A}^{-1}\bfgamma = \left[\begin{array}{c} \kappa \\ -\frac{1}{2}(\kappa+ \frac{\lambda}{\alpha}s_1) \end{array}\right],
\end{equation}
where $s_1=\hat{\bfphi}_1(X_m)\cdot\hat{\bfn}=(e\hat{n}_1 - d\hat{n}_2)/2$. Using \eqref{App_A1_eq1_extra1} again and similar derivation above, we have
\begin{equation}
\label{App_A2_eq6}
\verti{\frac{s_1}{\alpha}}=\verti{ \frac{ e\hat{n}_1 - d\hat{n}_2 }{2\alpha} } = \frac{ | el_2 \sin(\theta-\delta) + dl_3 \sin(\delta) | }{2|\Gamma^T_h|} 
\le \frac{ | el_2 \sin(\theta-\delta) + dl_3 \sin(\delta) | }{2 \max\{dl_3,el_2\}\sin(\theta)} \le C
\end{equation}
due to the shape regularity and $d,e\le 1$. Hence, similar to \eqref{App_A2_eq4} we obtain
\begin{equation}
\label{App_A2_eq7}
\| \mathbf{ A}^{-1} \bfgamma \|_{\infty} \le C \max\Big\{ 1, \frac{\mu^+}{\mu^-} \Big\} \max\Big\{ 1, \frac{\beta^-}{\beta^+} \Big\}  \le C .
\end{equation}
Putting \eqref{App_A2_eq4} and \eqref{App_A2_eq7} into the formula \eqref{IFE_fun_ref_5}, we have
\begin{equation}
\begin{split}
\label{App_A2_eq8}
\| \bfc \|_{\infty}   \le \| (\mathbf{ I}  +  \mathbf{ R} \mathbf{ A}^{-1} \mathbf{ B})^{-1} \|_{\infty}\| \bfv \|_{\infty} + 
\| (\mathbf{ I}  +  \mathbf{ R} \mathbf{ A}^{-1} \mathbf{ B})^{-1} \|_{\infty} \| \mathbf{ A}^{-1} \bfgamma \|_{\infty} |v_1|  \le  C .
\end{split}
\end{equation}
Moreover we note that
\begin{equation}
\label{App_A2_eq9}
\mathbf{ A}^{-1}\mathbf{ B} = 
\left[\begin{array}{cc} \kappa & \kappa \\  \left(-\kappa - \lambda\frac{s_2}{\alpha} \right)   & \left(-\kappa - \lambda\frac{s_2}{\alpha} \right) \end{array}\right].
\end{equation}
Therefore, putting \eqref{App_A2_eq5} and \eqref{App_A2_eq9} into the formula for $\bfb$ in \eqref{IFE_fun_ref_2}, we have $|b_1| = |\kappa| ~|v_1+c_2+c_3|  \le C$, and
$|b_2|  =  \verti {-\kappa\left( \frac{v_1}{2} + c_2 +c_3 \right) - \frac{\lambda}{\alpha} \left( s_1 + s_2 + s_3  \right) }  \le C + \frac{Ch}{|\Gamma^T_h|}$
where we have used the estimates for $s_i/\alpha$, $i=1,2,3$ from \eqref{App_A2_eq3} and \eqref{App_A2_eq6} together with the shape regularity. Then using the estimate $|\Gamma^T_h|\ge  \max\{dl_3,el_2\}\sin(\theta)$ we clearly have $|b_2|e, |b_2|d \le C$.
Finally the estimates above together with \eqref{ref_jump_cond} yield 
$\| \hat{\bfz} \|_{L^{\infty}(T)} \le C$.
Hence the desired result \eqref{thm_bound_eq01} follows from the Piola transformation in \eqref{piola_trans}. Besides, \eqref{thm_bound_eq02} can be derived by direct calculation.

\subsection{Proof of Lemma \ref{lem_loc_infsup}}
  \label{App_A3}
We let $\bar{\bfn}=[n_1,n_2]$ and $\bar{\bft}=[-n_2,n_1]$ be the normal and tangential vectors to linearly-approximate interface $\Gamma^T_h$. Define a transmission orthogonal matrix $\bfQ=[\bar{\bft},\bar{\bfn}]^t$ and a diagonal matrix $\bf\Lambda=\left[\begin{array}{cc} 1 & 0 \\ 0 & \rho \end{array}\right]$ with $\rho=\beta^-/\beta^+$. By the exact sequence for IFE functions, we know $\mathcal{IND}_h(T)\cap\text{Ker(curl)}$ consists of piecewise constant vectors on $T$, and thus, without loss of generality, we assume $\bfu_h^-=\bfu_h|_{T^-_h}$ is a constant unit vector. Then we write $\bfu^+_h = \bfQ^t\bfLambda\bfQ\bfu^-_h$. So the object is to show that the inner product of $\bfu_h$ and $\mathbb{\Pi}_h\bfu_h\ge Ch^2$ over each interface element (or plus its neighborhood elements) is lower bounded by $Ch^2$ with some constant $C$ independent of interface location. We organize our arguments for two possible cases: the interface cuts the two adjacent edges of the right angle or non-right angle as shown in Figure \ref{fig:infsup_case}.
\vspace{0.1in}

\textbf{Case 1}. In this case, by the geometry assumption, we assume the element has the configuration, shown by the left plot in Figure \ref{fig:infsup_case}, with the vertices $A_1=(0,0)$, $A_2=(h,0)$ and $A_3=(0,h)$, and we let the interface-intersection points be $D=(dh,0)$ and $E=(0,eh)$, $d,e\in[0,1]$. Then we can express $n_1=e/\sqrt{d^2+e^2}$ and $n_2=d/\sqrt{d^2+e^2}$. The direct calculation yields
\begin{equation}
\begin{split}
\label{App_A3_eq1}
\mathbb{\Pi}_{h,T}\bfu_h &= \left[\begin{array}{cc} 1-d  & 0  \\ 0 & 1-e \end{array}\right]\bfu^-_h + \left[\begin{array}{cc} d  & 0  \\ 0 & e \end{array}\right] \bfQ^t \bfLambda \bfQ \bfu^-_h\\
&= \bfu^-_h + \frac{(\rho-1)de}{d^2+e^2} \left[\begin{array}{cc} e & d  \\ e  & d \end{array}\right] \bfu^-_h = (\bfI_2 + (\rho-1)\bfA_1)\bfu^-_h.
\end{split}
\end{equation}
The eigenvalues of $(\bfA^t_1+ \bfA_1)/2$ are 
\begin{equation}
\label{App_A3_eq3}
\lambda_{1} = \frac{de(d+e-\sqrt{2(d^2+e^2)})}{2(d^2+e^2)} \in [(5-3\sqrt{3})/8,0] ~~~ \text{and} ~~~ \lambda_{2} = \frac{de(d+e+\sqrt{2(d^2+e^2)})}{2(d^2+e^2)} \in [0,1]
\end{equation}
Then if $\rho=\beta^-/\beta^+<8/(3\sqrt{3}-5)+1\approx 41.816$, on $T^-_h$ we have
\begin{equation}
\label{App_A3_eq4}
\bfu^-_h \cdot \mathbb{\Pi}_{h,T}\bfu_h =  \bfu^-_h\cdot\bfu^-_h + \frac{(\rho-1)}{2}(\bfu^-_h)^t(\bfA^t_1+\bfA_1)\bfu^-_h > \min\{1+(\rho-1)\lambda_1,1+(\rho-1)\lambda_2\} \ge C.
\end{equation}
In addition, we note another equivalent expression of \eqref{App_A3_eq1}:
\begin{equation}
\begin{split}
\label{App_A3_eq5}
\mathbb{\Pi}_{h,T}\bfu_h &= \left[\begin{array}{cc} 1-d  & 0  \\ 0 & 1-e \end{array}\right]\bfQ^t\bfLambda^{-1}\bfQ \bfu^+_h + \left[\begin{array}{cc} d  & 0  \\ 0 & e \end{array}\right] \bfu^+_h\\
&=  \bfu^+_h + \frac{(\rho^{-1}-1)}{d^2+e^2} \left[\begin{array}{cc} (1-d)e^2 & (1-d)de  \\ (1-e)de  & (1-e)d^2 \end{array}\right] \bfu^-_h = (\bfI_2 + (\rho^{-1}-1)\bfA_2)\bfu^+_h.
\end{split}
\end{equation}
Similarly, by estimating the eigenvalues of $(\bfA^t_2+ \bfA_2)/2$,
if $\rho^{-1}=\beta^+/\beta^-<8(3\sqrt{3}-5)+1\approx 41.816$, on $T^-_h$ we have
\begin{equation}
\label{App_A3_eq7}
\bfu^+_h \cdot \mathbb{\Pi}_{h,T}\bfu_h =  \bfu^+_h\cdot\bfu^+_h + \frac{(\rho-1)}{2}(\bfu^+_h)^t(\bfA^t_2+\bfA_2)\bfu^+_h > \min\{1+(\rho-1)\lambda_{1},1+(\rho-1)\lambda_{2}\} \ge Ch^2.
\end{equation}
Note that $\bfu^+_h\cdot\bfu^+_h=(\bfu^-_h)^t\bfQ^t\bfLambda^2\bfQ\bfu^-_h\approx C$ since $\|\bfu^-_h\|_2=1$. Combining \eqref{App_A3_eq4} and \eqref{App_A3_eq7}, we obtain
\begin{equation}
\label{App_A3_eq8}
(\beta\bfu_h, \mathbb{\Pi}_{h,T} \bfu_h )_{L^2(T)} \ge Ch^2 \approx C \| \bfu_h \|^2_{L^2(T)}.
\end{equation}

\vspace{0.1in}

\textbf{Case 2}. In this case, by the geometry assumption, we assume the element has the configuration shown by the right plot in Figure \ref{fig:infsup_case}, with the vertices $A_1=(0,0)$, $A_2=(h,0)$ and $A_3=(h,h)$, and we let the interface-intersection points be $D=(dh,0)$ and $E=(eh,eh)$, $d,e\in[0,1]$. Then we can express $n_1=e/\sqrt{(d-e)^2+e^2}$ and $n_2=(d-e)/\sqrt{(d-e)^2+e^2}$. Using the argument above, we can show the positive lower bound for $\rho\in[1/9,9]$. But we can actually get a slightly better bound for $\rho$ if the neighborhood non-interface element denoted by $T'$ is included. To see this, we use the direct calculation to obtain
\begin{equation}
\begin{split}
\label{App_A3_eq9}
\mathbb{\Pi}_{h,T}\bfu_h &= \left[\begin{array}{cc} 1-d & 0  \\  d-e & 1-e \end{array}\right]  \bfu^-_h +  \left[\begin{array}{cc} d & 0  \\  e-d & e \end{array}\right]\bfQ^t\bfLambda\bfQ \bfu^-_h \\
& = \bfu^-_h + \frac{(\rho-1)}{(d-e)^2 + e^2} \left[\begin{array}{cc} de^2 & de(d-e)  \\  0 & 0 \end{array}\right] \bfu^-_h = (\bfI_2 + (\rho-1) \bfB_1)\bfu^-_h.
\end{split}
\end{equation}
We denote the unit normal and tangential vectors to the non-interface edge $e_1$ connecting $(h,0)$ and $(h,h)$ by $\bfn_1$ and $\bft_1$, and let $\bfu^-_h=a\bfn_1 + b\bft_1$ with $a^2+b^2=1$. Without loss of generality, we assume $a \ge 0$. We further let $\alpha_1=\bfn_1^t\bfB_1\bfn_1$ and $\alpha_2=\bfn_1^t\bfB_1\bft_1$. It is easy to see that $\bft_1^t\bfB_1=0$, and we then observe
\begin{equation}
\label{App_A3_eq10}
\bfu^-_h\cdot\mathbb{\Pi}_h\bfu_h =  1 + (\rho-1)(a^2\alpha_1 + ab\alpha_2).
\end{equation}
The direct calculation yields the following estimates
\begin{equation}
\begin{split}
\label{App_A3_eq11}
& \alpha_1 = \frac{de^2}{(d-e)^2 + e^2} \in [0,1], ~~ \alpha_2 = \frac{de(d-e)}{(d-e)^2 + e^2} \in [(1-\sqrt{2})/2,0.5]~~ \text{and} ~~ \alpha_1^2 + \alpha_2^2 \le 1.
\end{split}
\end{equation}
Then if $\rho \le 1$, using the estimate $a^2\alpha_1+ab\alpha_2 \le a\sqrt{a^2+b^2}\sqrt{\alpha_1^2+\alpha_2^2} \le a$, we have
\begin{equation}
\label{App_A3_eq12}
\bfu^-_h\cdot\mathbb{\Pi}_{h,T}\bfu_{h} \ge 1 - a + \rho a  \ge C.
\end{equation}
It remains to show the estimate for $\rho > 1$. If $a^2\alpha_1 + ab\alpha_2\ge 0$, then $\bfu^-_h\cdot\mathbb{\Pi}_h\bfu_h \ge C$. In addition, if $a^2\alpha_1 + ab\alpha_2 < 0$, the direct calculation yields
\begin{equation}
\label{App_A3_eq14}
\frac{-(1+b^2)}{a^2\alpha_1 + ab\alpha_2} \ge 2 \sqrt{2} \sqrt{ \frac{((d^2 - 2 d e + 2 e^2)^2 (d^2 - 2 d e + 3 e^2))}{(
  d^2 (d - e)^4 e^2)} } + \frac{4((d - e)^2 + e^2)}{d(d - e)^2} \ge 2 (4 + 2 \sqrt{2}).
\end{equation}
Hence, the following estimate is true for $\rho=\beta^-/\beta^+\in (1,9+4\sqrt{2})$ where $9+4\sqrt{2}\approx 14.65$
\begin{equation}
\label{App_A3_eq13}
\bfu^-_h\cdot\mathbb{\Pi}_{h,T}\bfu_h + (\bfu^-_h\cdot\bft_1)^2 = \bfu^-_h\cdot\mathbb{\Pi}_{h,T}\bfu_h + b^2 \ge  1+ b^2 + (\rho-1)(a^2\alpha_1 + ab\alpha_2) \ge C.
\end{equation}
As for $\bfu^+_h$, the similar derivation yields
\begin{equation}
\label{App_A3_eq13_1}
\bfu^-_h\cdot\mathbb{\Pi}_{h,T}\bfu_h \ge C ~~~ \text{if} ~ \rho^{-1} \le 1 ~~~ \text{and} ~~~ \bfu^+_h\cdot\mathbb{\Pi}_{h,T}\bfu_h + \rho(\bfu^-_h\cdot\bft_1)^2 \ge C ~~~ \text{if} ~ \rho^{-1} \le 10.655.
\end{equation}
Due to the continuity of $\bfu_h$ along the tangential direction of the non-interface edge, we have $\bfu_h=a'\bfn_1 + b\bft_1$ on $T'$ where we recall $\bfu^-_h|_T\cdot\bft_1=b$, and then obtain $\bfu_h\cdot\mathbb{\Pi}_{h,T'}\bfu_h = a'^2+(\bfu^-_h\cdot\bft_1)^2$ on $T'$. Therefore, assuming $\rho\in[1/10.655,14.65]$, the estimates above lead to
\begin{equation}
\begin{split}
\label{App_A3_eq15}
&(\beta\bfu_h, \mathbb{\Pi}_{h,T}\bfu_h)_{L^2(T)} + (\beta^-\bfu_h,\mathbb{\Pi}_{h,T'}\bfu_h )_{L^2(T')} \\
 = & \beta^-(\bfu^-_h\cdot\mathbb{\Pi}_{h,T}\bfu_h + (\bfu^-_h\cdot\bft_1)^2 )|T^-_h| + \beta^+( \bfu^+_h\cdot\mathbb{\Pi}_{h,T}\bfu_h +  \rho(\bfu^-_h\cdot\bft_1)^2 ) |T^+_h| + a'^2|T'|\\
 \ge & C h^2 + a'^2|T'| \ge C(1+b^2+a'^2)h^2 \approx C \| \bfu_h \|^2_{L^2(T\cup T')}.
\end{split}
\end{equation}
\begin{figure}[H]
\centering
\begin{subfigure}{.32\textwidth}
     \includegraphics[width=2in]{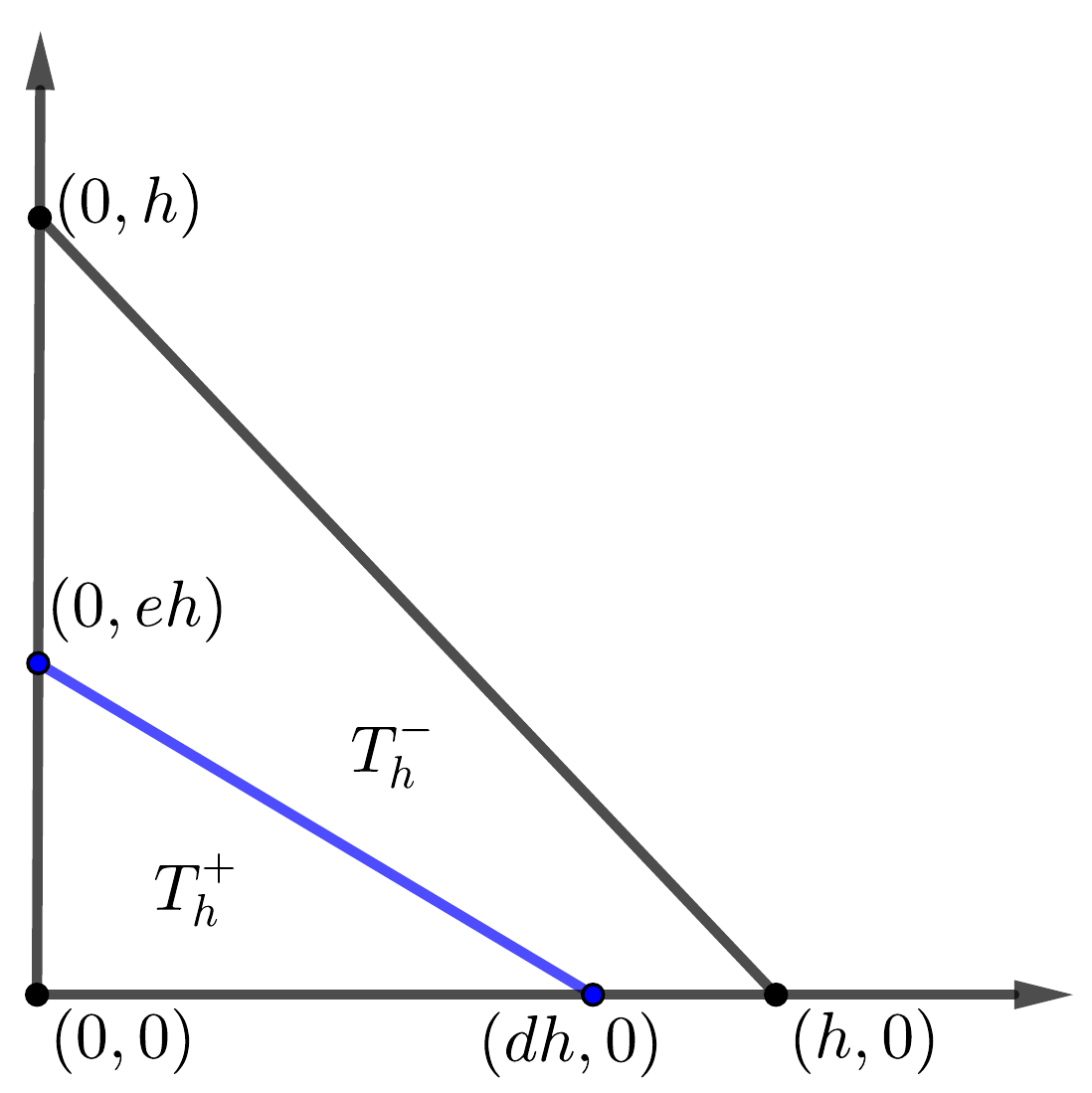}
     \label{infsup_case1} 
\end{subfigure}
~~~~
\begin{subfigure}{.37\textwidth}
     \includegraphics[width=2.5in]{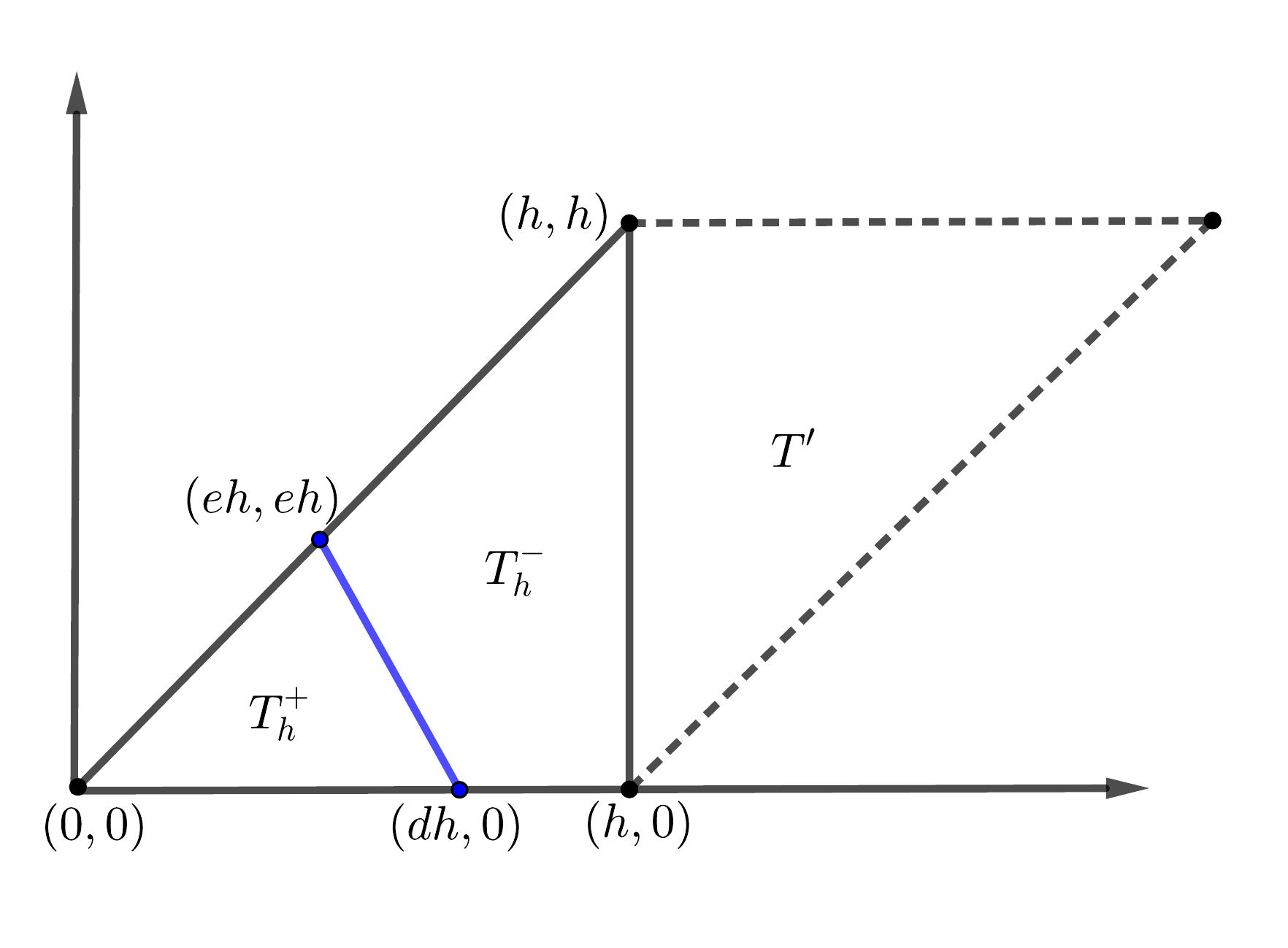}
     \label{infsup_case2} 
\end{subfigure}
     \caption{Interface element configuration: Case 1(left) and Case 2(right).}
  \label{fig:infsup_case} 
\end{figure}

\subsection{Proof of Lemma \ref{lem_u_gam_h}}
\label{app_4}
Let $\bfw = \bfu^+_E\cdot\bar{\bft} - \bfu^-_E$ and then $\bfw\cdot\bar{\bft}\in \bfH^1(\text{curl};T_{\epsilon})$ because of Theorem \ref{thm_ext}. By the density argument, we only need to prove the result for sufficient smooth $\bfw$. Let $T_{\epsilon}$ be $A^{\epsilon}_1A^{\epsilon}_2A^{\epsilon}_3$ and let $\Gamma$ intersect $\partial T_{\epsilon}$ with the points $D^{\epsilon}$ and $E^{\epsilon}$. We first consider the curved-edge quadrilateral $T^+_{\epsilon}=A^{\epsilon}_1A^{\epsilon}_2E^{\epsilon}D^{\epsilon}$. The basic idea is to construct a finite number of strips with bounded width to cover the whole curved-edge quadrilateral and the number is bounded independent of interface location. For this purpose, we let $P_0$ be $A^{\epsilon}_1$ and let $P_1$ be the point on $A^{\epsilon}_1A^{\epsilon}_2$ such that $P_1E^{\epsilon}$ is parallel to $D^{\epsilon}A^{\epsilon}_1$, then the first strip is the curved-edge quadrilateral $A^{\epsilon}_1P_1E^{\epsilon}D^{\epsilon}$ denoted as $s_1$. Then we proceed by induction to construct $P_n$ on the edge $A^{\epsilon}_1A^{\epsilon}_2$ such that $P_nE^{\epsilon}$ is parallel to $P_{n-1}D^{\epsilon}$, $n\ge2$, and the $n$-th strip denoted by $s_n$ is the curved-edge quadrilateral $P_{n-1}P_nE^{\epsilon}D^{\epsilon}$. The last point $P_N$ may locate outside the edge $A^{\epsilon}_1A^{\epsilon}_2$ and if it happens we then simply let $P_N=A^{\epsilon}_2$. Thus we have totally $N$ strips $s_1$, $s_2$,...,$s_N$ and $T_{\epsilon}^+=\cup_{i=1}^N s_i$ as shown in Figure \ref{fig:strip_1}. Without loss of generality, we assume $|A^{\epsilon}_3D^{\epsilon}|/|A^{\epsilon}_3A^{\epsilon}_1|\ge |A^{\epsilon}_3E^{\epsilon}|/|A^{\epsilon}_3A^{\epsilon}_2|$, and then \eqref{interf_ext} yields
\begin{equation}
\label{lem_u_gam_eq1}
|P_{N-1}P_{N-2}| \ge |P_{N-2}P_{N-3}| \ge \cdots \ge |P_1P_0| \ge \delta_1\sin(\delta_2) h_T.
\end{equation}
Using \eqref{interf_ext} and \eqref{lem_u_gam_eq1}, we have $(N-1)\delta_1\sin(\delta_2) h_T\le  h_T$, and thus $N\le \frac{1}{\delta_1\sin(\delta_2)} + 1$ which is independent of the interface location since $\delta_1,\delta_2$ are lower bounded from $0$ regard less of interface location. Now on each strip $s_n=P_{n-1}P_nE^{\epsilon}D^{\epsilon}$, $n=1,2,\cdots,N$, we consider a local Cartesian system with the $\xi$-axis perpendicular to $P_{n-1}D^{\epsilon}$ and the $\eta$-axis long the side $P_{n-1}D^{\epsilon}$ as shown by Figure \ref{fig:strip_2}. In this local system, let $f_1(\xi)$ and $f_2(\xi)$, $\xi\in[0,\xi_n]$, be the functions corresponding to the side $P_{n-1}P_n$ and the curved-side $D^{\epsilon}E^{\epsilon}$. Then the 1D Friedrichs inequality for functions vanishing on one boundary \cite{2003London} yields
\begin{equation}
\begin{split}
\label{lem_u_gam_eq2}
\| \bfw \|^2_{L^2(s_n)} & = \int^{\xi_n}_0 \int_{f_1(\xi)}^{f_2(\xi)} \| \bfw \|^2 d\eta d\xi \\
& \le \int_0^{\xi_n} |f_2(\xi) - f_1(\xi)|^2 \int_{f_1(\xi)}^{f_2(\xi)} \| \partial_{\eta}\bfw \|^2  d\eta d\xi + \int^{\xi_n}_0 |f_2(\xi) - f_1(\xi)| \| \bfw(\xi,f_2(\xi)) \|^2 d\eta d\xi \\
& \le Ch^2_T \| \bfw \|^2_{H^1(s_n)} + Ch_T \| \bfw \|^2_{L^2(\Gamma\cap T_{\epsilon})}.
\end{split}
\end{equation}
Furthermore, the jump condition \eqref{inter_jc_1}, the geometric estimate \eqref{lemma_interface_flat_eq2} and the trace inequality yield 
\begin{equation}
\begin{split}
\label{lem_u_gam_eq3}
 \| \bfw \|^2_{L^2(\Gamma\cap T_{\epsilon})} & = \| \bfu^+_E\cdot\bar{\bft} - \bfu^-_E\cdot\bar{\bft} \|^2_{L^2(\Gamma\cap T_{\epsilon})} = \| (\bfu^+_E - \bfu^-_E)\cdot(\bar{\bft}-\bft) \|^2_{L^2(\Gamma\cap T_{\epsilon})} \le Ch_T^2 \| \bfu^+_E - \bfu^-_E \|^2_{L^2(\Gamma\cap T_{\epsilon})} \\
 & \le \left( Ch_T \| \bfu^+_E - \bfu^-_E \|^2_{L^2( T_{\epsilon})} + Ch^3_T | \bfu^+_E - \bfu^-_E |^2_{H^1(T_{\epsilon})} \right).
 \end{split}
\end{equation}
Putting \eqref{lem_u_gam_eq3} into \eqref{lem_u_gam_eq2} yields $\| \bfw \|_{L^2(s_n)} \le Ch_T \left( \| \bfu_E^+ \|_{H^1(T_{\epsilon})} +  \| \bfu_E^- \|_{H^1(T_{\epsilon})}  \right)$.
Summing this estimate over all the strips, we obtain
\begin{equation}
\label{lem_u_gam_eq5}
\| \bfw \|_{L^2(T^+_{\epsilon})} \le \sum_{n=1}^N \| \bfw \|_{L^2(s_n)} \le CNh_T \left( \| \bfu_E^+ \|_{H^1(T_{\epsilon})} +  \| \bfu_E^- \|_{H^1(T_{\epsilon})}  \right)
\end{equation}
which yields the desired result on $T^+_{\epsilon}$ since $N$ is uniformly bounded. Similar argument can be also applied to $T^-_{\epsilon}$.
\begin{figure}[h]
\centering
\begin{minipage}{.42\textwidth}
  \centering
  \includegraphics[width=2.5in]{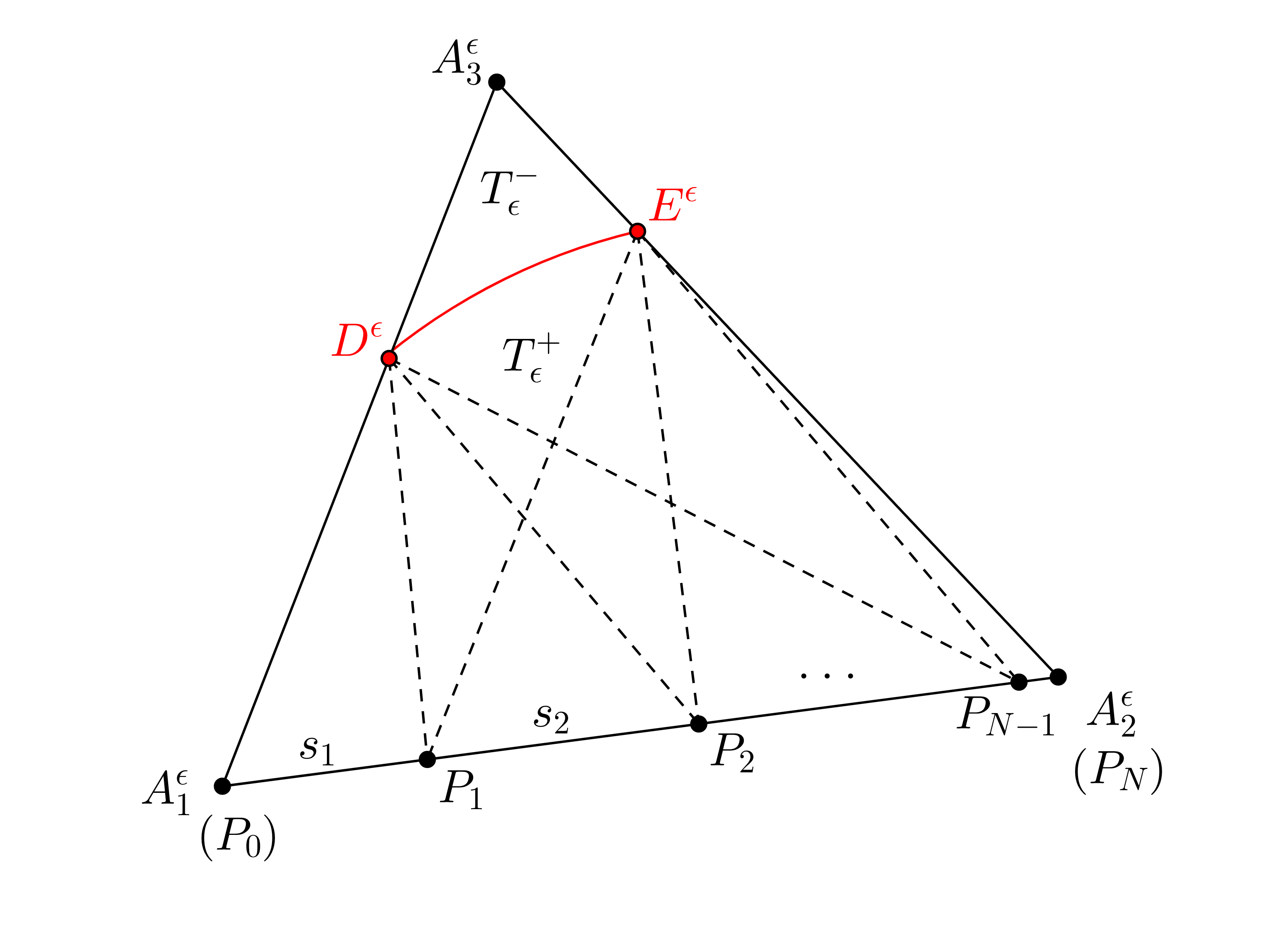}
  \caption{The partition of subelements into strips}
  \label{fig:strip_1}
\end{minipage}
\hspace{2cm}
\begin{minipage}{.4\textwidth}
  \centering
  \ \includegraphics[width=2.1in]{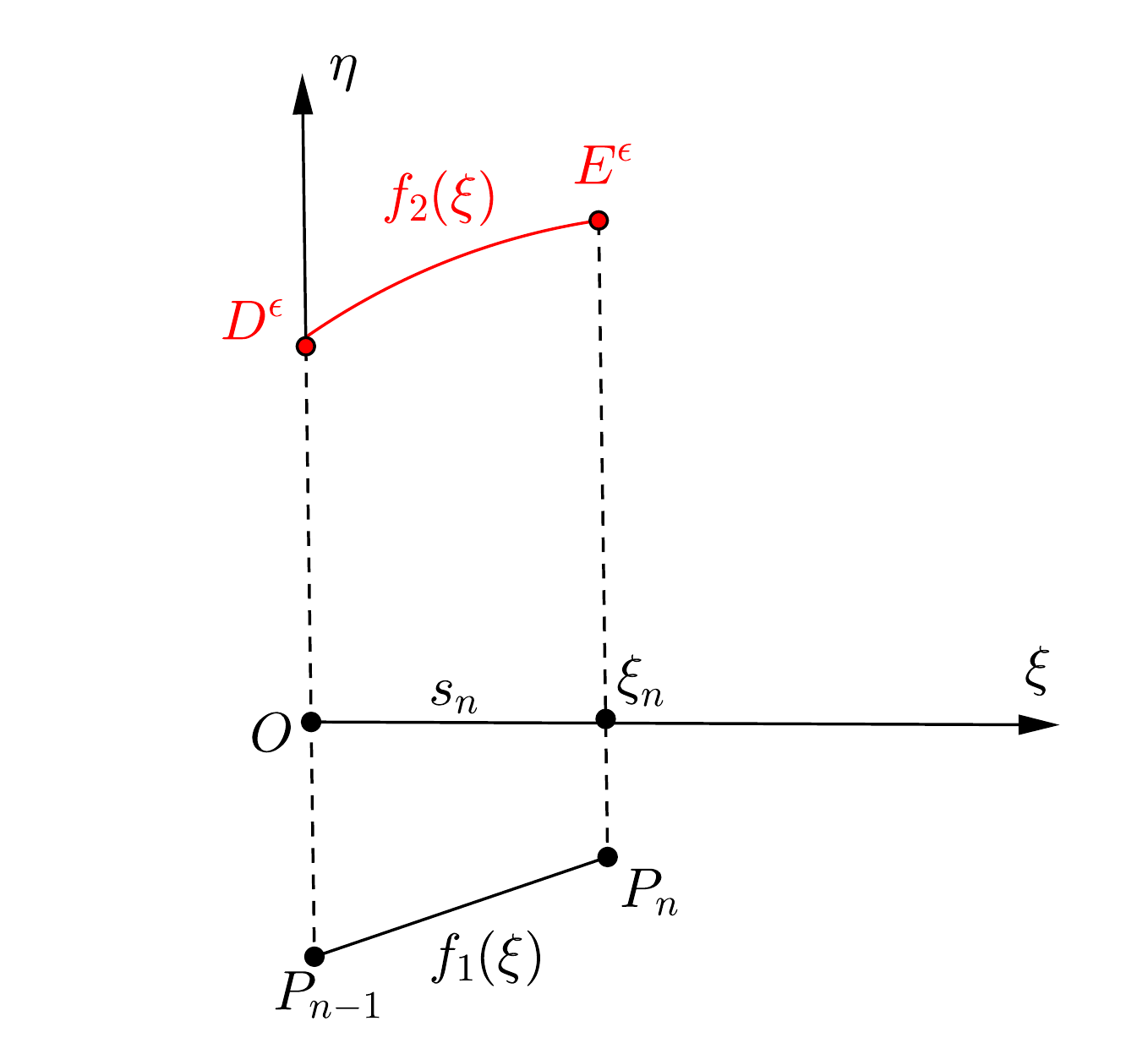}
  \caption{The local system for each strip}
  \label{fig:strip_2}
  \end{minipage}
\end{figure}

\end{appendices}

\bibliographystyle{plain}
\clearpage

\begin{thebibliography}{10}

\bibitem{2003London}
R.~A. Adams and J.~J.~F. Fournier.
\newblock {\em Sobolev spaces}.
\newblock Pure and applied mathematics. Academic Press, 2003.

\bibitem{2020AdjeridBabukaGuoLin}
S.~Adjerid, I.~Babu{\v s}ka, R.~Guo, and T.~Lin.
\newblock An enriched immersed finite element method for interface problems
  with nonhomogeneous jump conditions.
\newblock {\em Submitted,(arXiv:2004.13244)}, 2020.

\bibitem{2018Alberti}
G.~S. Alberti.
\newblock H{\"o}lder regularity for {M}axwell's equations under minimal
  assumptions on the coefficients.
\newblock {\em Calc. Var. Partial Differential Equations}, 57(3):71, 2018.

\bibitem{2000AmmariBuffaNedelec}
H.~Ammari, A.~Buffa, and J.~C. N{\'e}d{\'e}lec.
\newblock A justification of eddy currents model for the {M}axwell equations.
\newblock {\em SIAM J. Appl. Math.}, 60(5):1805--1823, 2000.

\bibitem{2015AmmariChenChenVolkov}
H.~Ammari, J.~Chen, Z.~Chen, D.~Volkov, and H.~Wang.
\newblock Detection and classification from electromagnetic induction data.
\newblock {\em J. Comput. Phys.}, 301:201 -- 217, 2015.

\bibitem{2006ArnoldFalkWinther}
D.~N. Arnold, R.~S. Falk, and R.~Winther.
\newblock Differential complexes and stability of finite element methods {I}.
  the de {R}ham complex.
\newblock In D.~N. Arnold, P.~B. Bochev, R.~B. Lehoucq, R.~A. Nicolaides, and
  M.~Shashkov, editors, {\em Compatible Spatial Discretizations}, pages 23--46,
  New York, 2006. Springer.

\bibitem{2006ArnoldFalkWintherExterior}
D.~N. Arnold, R.~S. Falk, and R.~Winther.
\newblock Finite element exterior calculus, homological techniques, and
  applications.
\newblock {\em Acta Numerica}, 15:1--155, 2006.

\bibitem{1972BabuskaAziz}
I.~Babu{\v s}ka and A.~K. Aziz.
\newblock Survey lectures on the mathematical foundations of the finite element
  method with applications.
\newblock {\em The Mathematical Foundations of the Finite Element Method with
  Applicaions to Partial Differential Equations}, pages 3--359, 1972.

\bibitem{1994BabuskaCalozOsborn}
I.~Babu{\v{s}}ka, G.~Caloz, and J.~E. Osborn.
\newblock Special finite element methods for a class of second order elliptic
  problems with rough coefficients.
\newblock {\em SIAM J. Numer. Anal.}, 31(4):945--981, 1994.

\bibitem{1983BabuskaOsborn}
I.~Babu{\v{s}}ka and J.~E. Osborn.
\newblock Generalized finite element methods: their performance and their
  relation to mixed methods.
\newblock {\em SIAM J. Numer. Anal.}, 20(3):510--536, 1983.

\bibitem{2000BabuskaOsborn}
I.~Babu{\v{s}}ka and J.~E. Osborn.
\newblock Can a finite element method perform arbitrarily badly?
\newblock {\em Math. Comp.}, 69(230):443--462, 2000.

\bibitem{2000BeckHiptmairHoppeWohlmuth}
R.~Beck, R.~Hiptmair, R.~H.~W. Hoppe, and B.~Wohlmuth.
\newblock Residual based a posteriori error estimators for eddy current
  computation.
\newblock {\em ESAIM: M2AN}, 34(1):159--182, 2000.

\bibitem{1991BrezziFortin}
F.~Brezzi and M.~Fortin.
\newblock {\em Mixed and hybrid finite element methods}, volume~15 of {\em
  Springer Ser. Comput. Math.}
\newblock Springer-Verlag, New York, 1991.

\bibitem{2015BurmanClaus}
E.~Burman, S.~Claus, P.~Hansbo, M.~G. Larson, and A.~Massing.
\newblock Cut{FEM}: Discretizing geometry and partial differential equations.
\newblock {\em Internat. J. Numer. Methods Engrg.}, 104(7):472--501, 2015.

\bibitem{2015CaiCao}
Z.~Cai and S.~Cao.
\newblock A recovery-based a posteriori error estimator for {H}(curl) interface
  problems.
\newblock {\em Comput. Methods Appl. Mech. Engrg.}, 296:169 -- 195, 2015.

\bibitem{2016CasagrandeHiptmairOstrowski}
R.~Casagrande, R.~Hiptmair, and J.~Ostrowski.
\newblock An a priori error estimate for interior penalty discretizations of
  the {C}url-{C}url operator on non-conforming meshes.
\newblock {\em J. Math. Ind.}, 6(1):4, 2016.

\bibitem{2016CasagrandeWinkelmannHiptmairOstrowski}
R.~Casagrande, C.~Winkelmann, R.~Hiptmair, and J.~Ostrowski.
\newblock {DG} treatment of non-conforming interfaces in 3{D} {C}url-{C}url
  problems.
\newblock In A.~Bartel, M.~Clemens, M.~G\"unther, and E.~J.~W. Maten, editors,
  {\em Scientific Computing in Electrical Engineering}, pages 53--61, Cham,
  2016. Springer International Publishing.

\bibitem{2000ChenDuZou}
Z.~Chen, Q.~Du, and J.~Zou.
\newblock Finite element methods with matching and nonmatching meshes for
  {M}axwell equations with discontinuous coefficients.
\newblock {\em SIAM J. Numer. Anal.}, 37(5):1542--1570, 2000.

\bibitem{2009ChenXiaoZhang}
Z.~Chen, Y.~Xiao, and L.~Zhang.
\newblock The adaptive immersed interface finite element method for elliptic
  and {M}axwell interface problems.
\newblock {\em J. Comput. Phys.}, 228(14):5000 -- 5019, 2009.

\bibitem{2010ChuGrahamHou}
C.-C. Chu, I.~G. Graham, and T.-Y. Hou.
\newblock A new multiscale finite element method for high-contrast elliptic
  interface problems.
\newblock {\em Math. Comp.}, 79(272):1915--1955, 2010.

\bibitem{1999CiarletZou}
P.~Ciarlet, Jr and J.~Zou.
\newblock Fully discrete finite element approaches for time-dependent
  {M}axwell's equations.
\newblock {\em Numer. Math.}, 82(2):193--219, 1999.

\bibitem{1999MartinMoniqueSerge}
M.~Costabel, M.~Dauge, and S.~Nicaise.
\newblock Singularities of {M}axwell interface problems.
\newblock {\em ESAIM Math. Model. Numer. Anal.}, 33(3):627--649, 1999.

\bibitem{1996Dirks}
H.~K. Dirks.
\newblock Quasi-stationary fields for microelectronic applications.
\newblock {\em Electrical Engineering}, 79(2):145--155, 1996.

\bibitem{2012DuanLiTanZheng}
H.~Duan, S.~Li, R.~C.~E. Tan, and W.~Zheng.
\newblock A delta-regularization finite element method for a double curl
  problem with divergence-free constraint.
\newblock {\em SIAM J. Numer. Anal.}, 50(6):3208--3230, 2012.

\bibitem{2016DuanQiuTanZheng}
H.~Duan, F.~Qiu, R.~C.~E. Tan, and W.~Zheng.
\newblock An adaptive {FEM} for a {M}axwell interface problem.
\newblock {\em J. Sci. Comput.}, 67(2):669--704, 2016.

\bibitem{2007GongLiLi}
Y.~Gong, B.~Li, and Z.~Li.
\newblock Immersed-interface finite-element methods for elliptic interface
  problems with nonhomogeneous jump conditions.
\newblock {\em SIAM J. Numer. Anal.}, 46(1):472--495, 2008.

\bibitem{2016GuoLin}
R.~Guo and T.~Lin.
\newblock A group of immersed finite element spaces for elliptic interface
  problems.
\newblock {\em IMA J.Numer. Anal.}, 39(1):482--511, 2017.

\bibitem{2019GuoLin}
R.~Guo and T.~Lin.
\newblock A higher degree immersed finite element method based on a {C}auchy
  extension.
\newblock {\em SIAM J. Numer. Anal.}, 57(4):1545--1573, 2019.

\bibitem{2018GuoLinLin}
R.~Guo, T.~Lin, and Y.~Lin.
\newblock Error estimates for a partially penalized immersed finite element
  method for elasticity interface problems.
\newblock {\em ESAIM Math. Model. Numer. Anal.}, 54(1):1--24, 2020.

\bibitem{2005PeterCarloSangalliGiancarlo}
P.~Hansbo, C.~Lovadina, I.~Perugia, and G.~Sangalli.
\newblock A {L}agrange multiplier method for the finite element solution of
  elliptic interface problems using non-matching meshes.
\newblock {\em Numer. Math.}, 100(1):91--115, 2005.

\bibitem{2002Hiptmair}
R.~Hiptmair.
\newblock Finite elements in computational electromagnetism.
\newblock {\em Acta Numer}, 11:237--339, 2002.

\bibitem{2012HiptmairLiZou}
R.~Hiptmair, J.~Li, and J.~Zou.
\newblock Convergence analysis of finite element methods for {$H(curl;
  \Omega)$}-elliptic interface problems.
\newblock {\em Numer. Math.}, 122(3):557--578, Nov 2012.

\bibitem{2017HiptmairPechstein}
R.~Hiptmair and C.~Pechstein.
\newblock Discrete regular decompositions of tetrahedral discrete 1-forms.
\newblock SAM-Report 2017-47, ETH Zurich, 2017.

\bibitem{2019HiptmairPechstein}
R.~Hiptmair and C.~Pechstein.
\newblock Regular decompositions of vector fields - continuous, discrete, and
  structure-preserving.
\newblock SAM-Report 2019-18, ETH Zurich, 2019.

\bibitem{2006HiptmairWidmerZou}
R.~Hiptmair, G.~Widmer, and J.~Zou.
\newblock Auxiliary space preconditioning in {${H}_0(curl; \Omega)$}.
\newblock {\em Numer. Math.}, 103(3):435--459, 2006.

\bibitem{2013HouSongWangZhao}
S.~Hou, P.~Song, L.~Wang, and H.~Zhao.
\newblock A weak formulation for solving elliptic interface problems without
  body fitted grid.
\newblock {\em J. Comput. Phys.}, 249:80 -- 95, 2013.

\bibitem{2004HouWuZhang}
T.-Y. Hou, X.~Wu, and Y.~Zhang.
\newblock Removing the cell resonance error in the multiscale finite element
  method via a {P}etrov-{G}alerkin formulation.
\newblock {\em Commun. Math. Sci.}, 2(2):185--205, 2004.

\bibitem{2007HuangZou}
J.~Huang and J.~Zou.
\newblock Uniform a priori estimates for elliptic and static {M}axwell
  interface problems.
\newblock {\em Disc. Cont. Dynam. Sys., Series B}, 7, 2007.

\bibitem{2017HuangWuXiao}
P.~Huang, H.~Wu, and Y.~Xiao.
\newblock An unfitted interface penalty finite element method for elliptic
  interface problems.
\newblock {\em Comput. Methods Appl. Mech. Engrg.}, 323:439--460, 2017.

\bibitem{2010LiMelenkWohlmuthZou}
J.~Li, J.~M. Melenk, B.~Wohlmuth, and J.~Zou.
\newblock Optimal a priori estimates for higher order finite elements for
  elliptic interface problems.
\newblock {\em Appl. Numer. Math.}, 60(1):19--37, 2010.

\bibitem{2004LiLinLinRogers}
Z.~Li, T.~Lin, Y.~Lin, and R.~C. Rogers.
\newblock An immersed finite element space and its approximation capability.
\newblock {\em Numer. Methods Partial Differential Equations}, 20(3):338--367,
  2004.

\bibitem{2015LinLinZhang}
T.~Lin, Y.~Lin, and X.~Zhang.
\newblock Partially penalized immersed finite element methods for elliptic
  interface problems.
\newblock {\em SIAM J. Numer. Anal.}, 53(2):1121--1144, 2015.

\bibitem{2020LiuZhangZhangZheng}
H.~Liu, L.~Zhang, X.~Zhang, and W.~Zheng.
\newblock Interface-penalty finite element methods for interface problems in
  h1, h(curl), and h(div).
\newblock {\em Comput. Methods Appl. Mech. Engrg.}, 367:113137, 2020.

\bibitem{2019LuYangBaiCaoHe}
C.~Lu, Z.~Yang, J.~Bai, Y.~Cao, and X.~He.
\newblock Three-dimensional immersed finite element method for anisotropic
  magnetostatic/electrostatic interface problems with non-homogeneous flux
  jump.
\newblock {\em Internat. J. Numer. Methods Engrg.}, 2019.

\bibitem{2003Monk}
P.~Monk.
\newblock {\em Finite Element Methods for Maxwell's Equations}.
\newblock Oxford University Press, 2003.

\bibitem{1980Nedelec}
J.~C. Nedelec.
\newblock Mixed finite elements in $\mathbb{R}^3$.
\newblock {\em Numer. Math.}, 35(3):315--341, 1980.

\bibitem{2003WarburtonHesthaven}
T.~Warburton and J.~S. Hesthaven.
\newblock On the constants in {$hp$}-finite element trace inverse inequalities.
\newblock {\em Comput. Methods Appl. Mech. Engrg.}, 192(25):2765--2773, 2003.

\bibitem{2011XuZhu}
J.~Xu and Y.~Zhu.
\newblock Robust preconditioner for {$H$}(curl) interface problems.
\newblock In Y.~Huang, R.~Kornhuber, O.~Widlund, and J.~Xu, editors, {\em
  Domain Decomposition Methods in Science and Engineering XIX}, pages 173--180,
  Berlin, 2011. Springer.

\bibitem{2004ZhaoWei}
S.~Zhao and G.~W. Wei.
\newblock High-order {FDTD} methods via derivative matching for {M}axwell's
  equations with material interfaces.
\newblock {\em J. Comput. Phys.}, 200(1):60--103, 2004.

\end{thebibliography}

\end{document}